\documentclass[UTF8,10pt]{article}
\usepackage[left=1.91cm,right=1.91cm,top=2.54cm,bottom=2.54cm]{geometry}
\usepackage{amsfonts}
\usepackage{amsmath}
\usepackage{amssymb}
\usepackage{pgfplots}
\usepackage{tikz}
\usetikzlibrary{arrows}
\usepackage{titlesec}
\usepackage{bm}
\usepackage{appendix}
\usepackage{tikz-cd}
\usepackage{mathtools}
\usepackage{amsthm}
\usepackage{enumitem}
\setenumerate[1]{itemsep=0pt,partopsep=0pt,parsep=\parskip,topsep=5pt}
\setitemize[1]{itemsep=0pt,partopsep=0pt,parsep=\parskip,topsep=5pt}
\setdescription{itemsep=0pt,partopsep=0pt,parsep=\parskip,topsep=5pt}
\usepackage[colorlinks,linkcolor=black,anchorcolor=black,citecolor=black]{hyperref}
\usepackage{setspace}
\usepackage[numbers]{natbib}
\usepackage{cases}
\usepackage{fancyhdr}
\usepackage{txfonts}
\theoremstyle{definition}
\newtheorem{thm}{Theorem}[section]
\newtheorem{nota}[thm]{Notation}

\newtheorem{assump}[thm]{Assumption}

\newtheorem{lem}[thm]{Lemma}
\newtheorem{prop}[thm]{Proposition}

\newtheorem{cor}[thm]{Corollary}

\newtheorem*{rmk}{Remark}

\DeclareMathOperator{\id}{Id}
\newcommand{\R}{\mathbb{R}}

\newcommand{\T}{\mathbb{T}}

\newcommand{\TP}{\overline{\partial}{}}

\newcommand{\TJ}{\langle\overline{\partial}{}\rangle}

\newcommand{\curl}{\text{curl }}
\newcommand{\dive}{\text{div }}
\newcommand{\ST}{\text{ ST}}
\newcommand{\RT}{\text{ RT}}
\newcommand{\eql}{\overset{L}{=}}
\newcommand{\la}{\lambda}
\newcommand{\q}{\quad}
\newcommand{\p}{\partial}
\newcommand{\bz}{\mathbf{b}_0}
\newcommand{\bp}{(\bz\cdot\partial)}
\newcommand{\btp}{(\bz\cdot\overline{\partial})}
\newcommand{\dd}{\mathfrak{D}}
\newcommand{\DD}{\mathcal{D}}

\newcommand{\nab}{\nabla}
\newcommand{\pa}{\nabla_A}

\newcommand{\lap}{\Delta}

\newcommand{\di}{\text{div}\,}

\newcommand{\cp}{\overline{\partial}{}}

\newcommand{\dy}{\,\mathrm{d}y}

\newcommand{\dt}{\,\mathrm{d}t}
\newcommand{\dS}{\,\mathrm{d}S}
\newcommand{\vol}{\text{vol}\,}

\newcommand{\eps}{\varepsilon}

\newcommand{\VV}{\mathbf{V}}
\newcommand{\QQ}{\mathbf{Q}}

\newcommand{\ff}{\mathfrak{f}}

\newcommand{\AAA}{\mathbf{A}}
\newcommand{\BB}{\mathfrak{B}}

\newcommand{\PP}{\mathcal{P}}

\newcommand{\uu}{\mathbf{u}}
\newcommand{\ww}{\mathbf{w}}
\newcommand{\vv}{\mathbf{v}}

\newcommand{\wt}{\sqrt{\sigma}}

\newcommand{\hh}{\mathcal{H}}

\newcommand{\io}{\int_{\Omega}}

\newcommand{\ig}{\int_{\Gamma}}
\newcommand{\nn}{\hat{n}}

\numberwithin{equation}{section}
\usepackage{xcolor}

\newcommand{\si}{\sigma}
\newcommand{\red}{}

\begin{document}
\bibliographystyle{plain}
\title{\textbf{Zero Surface Tension Limit of the Free-Boundary Problem in Incompressible Magnetohydrodynamics}}
\author{Xumin Gu\thanks{School of Mathematics, Shanghai University of Finance and Economics, Shanghai 200433, China. Email: \texttt{gu.xumin@shufe.edu.cn}. XG was supported in part by NSFC Grant 12031006},\,\, Chenyun Luo\thanks{Department of Mathematics, The Chinese University of Hong Kong, Shatin, NT, Hong Kong.
Email: \texttt{cluo@math.cuhk.edu.hk}.\,\,\,\,\,\,\,
CL was supported in part by the RGC grant of Hong Kong CUHK-24304621 and CUHK Direct Grant for Research No. 4053457},\,\,  Junyan Zhang \thanks{Department of Mathematics, National University of Singapore, Singapore.
Email: \texttt{zhangjy@nus.edu.sg}}}
\date{\today}
\maketitle

\begin{abstract}
We show that the solution of the free-boundary incompressible ideal magnetohydrodynamic (MHD) equations with surface tension converges to that of the free-boundary incompressible ideal MHD equations without surface tension given the Rayleigh-Taylor sign condition holds initially. This result is a continuation of the authors' previous works \cite{gu2016construction,luozhangMHDST3.5, GLZ}. Our proof is based on the combination of the techniques developed in our previous works \cite{gu2016construction,luozhangMHDST3.5, GLZ}, Alinhac good unknowns, and a crucial anti-symmetric structure on the boundary.

\end{abstract}

\setcounter{tocdepth}{1}

\tableofcontents

\section{Introduction}
We consider the following 3D incompressible ideal MHD system which describes the motion of a conducting fluid with free surface boundary in an electro-magnetic field under the influence of surface tension
\begin{equation}\label{MHD}
\begin{cases}
D_t u-B\cdot\nabla B+\nabla P=0,~~P:=p+\frac{1}{2}|B|^2~~~& \text{in}~\DD; \\
D_t B-B\cdot\nabla u=0,~~~&\text{in}~\DD; \\
\dive u=0,~~\dive B=0,~~~&\text{in}~\DD,
\end{cases}
\end{equation} with boundary conditions
\begin{equation}\label{MHDB}
\begin{cases}
D_t|_{\p\DD}\in\mathcal{T}(\p\DD), \\
P=\sigma\mathcal{H}~~~& \text{on}~\p\DD, \\
B\cdot \nn=0~~~& \text{on}~\p\DD.
\end{cases}
\end{equation} 
Here,  $D_t:=\p_t+u\cdot\nabla$ denotes the material derivative, and $\DD:=\cup_{0\leq t\leq T}\{t\}\times\DD_t$, where $\DD_t\subseteq\R^3$ is a bounded domain occupied by the conducting fluid (plasma) whose boundary $\p\DD_t$ moves with the velocity of the fluid. Also, we denote by $v=(v_1,v_2,v_3)$ the fluid's velocity, $B=(B_1,B_2,B_3)$ the magnetic field, $p$ the fluid's pressure, and $P:=p+\frac12|B|^2$ the total pressure. The quantity $\mathcal{H}$ is the mean curvature of the free surface $\p\DD_t$ and $\sigma\geq 0$ is a given constant, called surface tension coefficient. 
Finally, we use $\nn$ to denote the exterior unit normal to $\p\DD_t$.  The equations $\dive B=0$ and $B\cdot\nn|_{\p\DD}=0$ are only the constraints for initial data that propagate to any $t>0$ if initially hold (cf. \cite{haoluo2014}). 

When $\sigma>0$ (i.e., the MHD equations are under the influence of the surface tension), we proved in \cite{GLZ} that given a simply connected domain $\DD_0\subseteq\R^3$ and initial data $u_0$ and $B_0$ satisfying $\dive u_0=0$ and $\dive B_0=0, B_0\cdot \nn|_{\p\DD_0}=0$, there exist a set $\DD$ and vector fields $u$ and $B$ that solve \eqref{MHD}-\eqref{MHDB} with initial data
\begin{equation}
\DD_0=\{x: (0,x)\in \DD\},\q (u,B)=(u_0, B_0),\q \text{in}\,\,\{t=0\}\times \Omega_0.\label{data}
\end{equation}
In addition, when $\sigma=0$ (i.e., the surface tension is neglected) the initial-boundary value problem \eqref{MHD}, \eqref{MHDB}, \eqref{data} is known to be ill-posed \cite{haoluo2018,ebin1987} unless the Rayleigh-Taylor sign condition is imposed 
\begin{align}
-\nab_n P \geq c_0>0, \quad \text{on}\,\,\Gamma,
\label{Taylor}
\end{align}
for some $c_0>0$. 
In fact, \textit{we only need to impose \eqref{Taylor} on the initial data and then show that it propagates within the interval of existence.}  The local well-posedness (LWP) of the initial-boundary value problem \eqref{MHD}, \eqref{MHDB}, \eqref{data}, \eqref{Taylor} was obtained by the first author and Wang in \cite{gu2016construction}.

The following terminologies will be used throughout the rest of this manuscript for the sake of clarity: We use ``$\sigma>0$ problem" to denote the initial-boundary value problem \eqref{MHD}, \eqref{MHDB}, \eqref{data} with a positive surface tension coefficient $\sigma$. In addition, we use ``$\sigma=0$ problem" to denote the initial-boundary value problem \eqref{MHD}, \eqref{MHDB}, \eqref{data} without the surface tension but with the physical condition \eqref{Taylor}.  In this paper, we would like to prove the zero surface tension limit of \eqref{MHD}-\eqref{data}, i.e., the solution to the ``$\sigma>0$ problem" converges to the solution to the ``$\sigma=0$ problem" as $\sigma\to 0_+$, provided the Rayleigh-Taylor sign condition \eqref{Taylor} holds initially.

Before stating our results, we would like to review the related previous results. The free-boundary problem of MHD can be considered \red{as} the simplified version of the plasma-vacuum free-interface model which is the basic theoretical model for plasma confinement and some astrophysical phenomena. We refer to our previous work \cite[Section 1.1.1]{GLZ} for a detailed discussion. In particular, the surface tension cannot be neglected when we use a free-boundary MHD system to model the liquid metal, film flow, etc \cite{MHDSTphy1, MHDSTphy2, MHDSTphy3}. Even if we consider the MHD flows in astrophysical plasmas, where the surface tension effect and magnetic diffusion are usually neglected, it is still useful to keep surface tension as a stabilization effect in numerical simulations of the magnetic Rayleigh-Taylor instability \cite{MHDSTphy4, MHDSTphy5}. 

In the absence of the magnetic field, i.e., $B=\mathbf{0}$, the free-boundary MHD system is reduced to the free-boundary incompressible Euler equations. The first breakthrough came into Wu \cite{wu1997LWPww,wu1999LWPww} for the LWP of 2D and 3D incompressible irrotational gravity water waves. In the case of nonzero vorticity, Christodoulou-Lindblad \cite{christodoulou2000motion} first established the nonlinear a priori estimates and Lindblad \cite{lindblad2005well} proved the LWP by Nash-Moser iteration. Later, Coutand-Shkoller \cite{coutand2007LWP} avoided the Nash-Moser iteration and extended the result to the case $\sigma>0$. See also \cite{zhangzhang08Euler,shatah2008a,shatah2008b,shatah2011,DK2017,DKT} and references therein for the related works.

For the mathematical studies of free-boundary incompressible MHD system without surface tension, Hao-Luo \cite{haoluo2014,haoluo2019} proved the a priori estimates and the linearized LWP by generalizing \cite{christodoulou2000motion,lindblad2003well}. The first author and Wang \cite{gu2016construction} proved the LWP for the nonlinear problem. See also Lee \cite{leeMHD1,leeMHD2} for an alternative proof by using the vanishing viscosity-resistivity method, Sun-Wang-Zhang \cite{sun2015well} for the incompressible MHD current-vortex sheets, Sun-Wang-Zhang \cite{sun2017well} and the first author \cite{guaxi1,guaxi2} for the plasma-vacuum interface model, and the second and the third authors \cite{luozhangMHD2.5} for the minimal regularity estimates in a small fluid domain. For the compressible ideal MHD, we refer to \cite{chengqMHD,trakhininMHD2009,wangyuMHD2d,secchi2013well,trakhininMHD2020,Zhang2021CMHD} and references therein.

However, when the surface tension is not neglected, much less has been developed and most of the previous results focus on viscous or resistive MHD \cite{chendingMHDlimit, GuoMHDSTviscous,wangxinMHD}. To the best of our knowledge, our previous works \cite{luozhangMHDST3.5, GLZ} are the first breakthrough about incompressible ideal MHD with surface tension. In addition, we refer to \red{Li-Li \cite{sb} for well-posedness and vanishing surface tension limit under a non-collinearity condition for a two-phase model by assuming the free interface is a graph}; and Trakhinin-Wang \cite{trakhininMHD2021} for the compressible case.

This paper is the continuation of our previous works \cite{gu2016construction,luozhangMHDST3.5, GLZ}: we aim to prove the zero surface tension limit of the free-boundary problem in incompressible ideal MHD with surface tension. This can be achieved by establishing the uniform-in-$\sigma$ energy estimates. Our proof is based on the combination of the techniques in \cite{gu2016construction,luozhangMHDST3.5, GLZ}, Alinhac good unknowns, and the \textit{newly-developed} anti-symmetric structure on the boundary. 

\subsection{Reformulation in Lagrangian coordinates}
We reformulate the MHD equations in the Lagrangian coordinates and thus the free-surface domain becomes fixed. Let $\Omega\subseteq\R^3$ be a bounded domain. Denoting coordinates on $\Omega$ by $y=(y_1,y_2,y_3)$, we define $\eta:[0,T]\times \Omega\to\DD$ to be the flow map of the velocity $v$, i.e., 
\begin{equation}
\p_t \eta (t,y)=u(t,\eta(t,y)),\q
\eta(0,y)=\eta_0(y),
\end{equation}where $\eta_0:\Omega\to\DD_0$ is a diffeomorphism. We introduce the Lagrangian velocity, magnetic field, and pressure respectively by
\begin{equation}
v(t,y)=u(t,\eta(t,y)),\q
b(t,y)=B(t,\eta(t,y)),\q
q(t,y)=P(t,\eta(t,y)).
\end{equation}
Let $\p$ be the spatial derivative concerning the $y$ variable. We define the cofactor matrix $A=[\p\eta]^{-1}$ \red{and the Jacobian $J:=\det[\p\eta]$}. Then direct calculation gives the explicit form of $A$
\begin{align} \label{A represent}
A=\red{J^{-1}}
\begin{pmatrix}
\p_2 \eta \times \p_3 \eta\\
\p_3 \eta \times \p_1\eta\\
\p_1\eta\times\p_2 \eta
\end{pmatrix}
=
\begin{pmatrix}
\epsilon^{\alpha\lambda \tau} \p_2\eta_\lambda\p_3 \eta_\tau\\
\epsilon^{\alpha\lambda \tau} \p_3\eta_\lambda\p_1 \eta_\tau\\
\epsilon^{\alpha\lambda \tau} \p_1\eta_\lambda\p_2 \eta_\tau
\end{pmatrix},
\end{align}
where $\epsilon^{\alpha\la \tau}$ is the fully antisymmetric symbol with $\epsilon^{123}=1$. 
In light of \eqref{A represent}, it is easy to see that $A$ verifies the Piola's identity
\begin{equation}\label{piola}
\p_\mu (\red{J}A^{\mu\alpha}) = 0,
\end{equation}
where the summation convention is used for repeated upper and lower indices. Above and throughout, all Greek indices range over $1, 2, 3$, and the Latin indices range over $1, 2$. 
For the sake of simplicity and clean notation, we consider the model case when
\begin{align}
\Omega = \T^2\times (0,1),
\end{align}
where $ \partial\Omega=\Gamma_0\cup\Gamma$ and $\Gamma=\T^2\times \{1\}$ is the top (moving) boundary, $\Gamma_0=\T^2\times\{0\}$ is the fixed bottom.
This is known to be the reference domain. 
  Using a partition of unity, e.g., \cite{coutand2007LWP, DKT,luozhangMHD2.5}, a general domain can also be treated with the same tools we shall present. However, choosing $\Omega$ as above allows us to focus on the real issues of the problem without being distracted by the cumbersomeness of the partition of unity. 
  
\begin{rmk}
The use of Lagrangian coordinates loosens the restriction on the geometry of the free surface in comparison to treating the free interface as a graph. Specifically, with a large initial velocity,  the free surface may fail to form a graph within a very short time interval. Nonetheless, using Lagrangian coordinates opens the possibility to go beyond that time.
\end{rmk}

The system \eqref{MHD}-\eqref{MHDB} can be reformulated as:
\begin{equation}\label{MHDL1}
\begin{cases}
\partial_tv_{\alpha}-b_{\beta}A^{\mu\beta}\partial_{\mu}b_{\alpha}+A^{\mu}_{\alpha}\partial_{\mu}q=0~~~& \text{in}~[0,T]\times\Omega;\\
\partial_t b_{\alpha}-b_{\beta}A^{\mu\beta}\partial_{\mu}v_{\alpha}=0~~~&\text{in}~[0,T]\times \Omega ;\\
\red{\dive_A v:=}A^{\mu\alpha}\partial_{\mu}v_{\alpha}=0,~~\red{\dive_A b:=}A^{\mu\alpha}\partial_{\mu}b_{\alpha}=0~~~&\text{in}~[0,T]\times\Omega;\\
v\cdot N=b\cdot N=0~~~&\text{on}~\Gamma_0;\\
A^{\mu\alpha}N_{\mu}q+\sigma(\sqrt{g}\Delta_g \eta^{\alpha})=0 ~~~&\text{on}~\Gamma;\\
A^{\mu\nu}b_{\nu}N_{\mu}=0 ~~~&\text{on}~\Gamma,\\
(\eta,v,b)=(\red{\eta_0},v_0,b_0)~~~&\text{on}~\{t=0\}{\times}\overline{\Omega}.
\end{cases}
\end{equation}
where \red{$A^{\mu}_{\alpha}:=\delta_{\alpha\beta} A^{\mu \beta}$},  $N$ is the unit outer normal vector to $\p\Omega$, particularly $N=(0,0,-1)$ on $\Gamma_0$ and $N=(0,0,1)$ on $\Gamma$, and $\Delta_g$ is the Laplacian of the metric $g_{ij}$ induced on $\p\Omega(t)$ by the embedding $\eta$. Specifically, we have:
\begin{equation}\label{gij}
g_{ij}=\TP_i\eta^{\mu}\TP_j\eta_{\mu},~\Delta_g(\cdot)=\frac{1}{\sqrt{g}}\TP_i(\sqrt{g}g^{ij}\TP_j(\cdot)),\text{ where } g:=\det (g_{ij}),
\end{equation} 
where $\TP = (\TP_1, \TP_2)$ denotes the tangential derivative (with respect to $\Gamma$). \red{The symmetric matrix $\{g_{ij}\}\in\R^{2\times2}$ is indeed a metric on the free surface, as one can easily show that all of its leading principal minors are positive to prove the positive-definiteness. }

\red{One can show that $\p_t J=J\text{tr}(A\p_t[\p\eta])=J\dive_A v=0$ and thus the Jacobian $J=J_0=\det[\p\eta_0]$ is independent of time. This property corresponds to incompressibility. To get rid of unnecessary terms involving $J_0$, we may assume $J_0=1$ for cleanness in the calculation. Under the Lagrangian formulation, we can express the magnetic field $b$ in terms of its initial data and the flow map. Specifically, by the second equation of \eqref{MHDL1} and the divergence-free condition on $b$, we get $\p_t(A^{\mu\nu}b_\nu)=0$ which implies $A^{\mu\nu}b_{\nu}=A_0^{\mu\nu}b_{0\nu}$ and thus $b_{\alpha}=A_0^{\mu\nu}b_{0\nu}\p_\mu \eta_{\alpha}:=\left((Ab_0)\cdot\p\right)\eta^{\alpha}$. We refer to Gu-Wang \cite[(1.13)-(1.15)]{gu2016construction} for the proof. Therefore, we denote $\bz=(Ab)_0$ and thus $b=\bp\eta$.}

We formulate the free-boundary MHD equations with and without surface tension. For each $\sigma>0$,  let $(\eta^\sigma, v^\sigma, q^\sigma)$ verifies
\begin{equation}\label{MHDLST}
\begin{cases}
\p_t\eta^\sigma=v^\sigma~~~& \text{in}~[0,T]\times\Omega;\\
\p_tv^\sigma-\bp^2\eta^\sigma+\nab_{A(\eta^\sigma)} q^\sigma=0~~~& \text{in}~[0,T]\times\Omega;\\
\di_{A(\eta^\sigma)} v^\sigma := A(\eta^\sigma)^{\mu\nu} \p_\mu v_\nu^\sigma=0,&\text{in}~[0,T]\times\Omega;\\
\di \bz:=\delta^{\mu}_{\nu}\p_\mu (\bz)^{\nu}=0~~~&\text{in}~\{0\}\times \Omega ;\\
(v^\sigma)^3=\bz^3=0~~~&\text{on}~\Gamma_0;\\
A(\eta^\sigma)^{3\alpha}q^\sigma+\sigma(\sqrt{g^\sigma}\Delta_{g^\sigma} (\eta^\sigma)^{\alpha})=0 ~~~&\text{on}~\Gamma;\\
\bz^3=0 ~~~&\text{on}~\Gamma,\\
(\eta^\sigma,v^\sigma)=(\red{\eta_0},v_0)~~~&\text{on}~\{t=0\}{\times}\overline{\Omega},
\end{cases}
\end{equation}
where $\nab_{A(\eta^\sigma)}^{\alpha}:=A(\eta^\sigma)^{\mu\alpha}\p_\mu$ denotes the covariant derivative, and the induced metric $g_{ij}=g_{ij}(\eta^\sigma)$ is given in \eqref{gij}. \red{Let $\nab_{A_0}:=\nab_{A(\eta_0)}$, and $\Delta_{A_0}:= \nab_{A_0}\cdot \nab_{A_0}$}. The initial pressure $q_0$ is determined by $\eta_0$, $v_0$ and $\bz$ through the elliptic equation
\begin{equation}\label{flat bdy eq}
\red{-\lap_{A_0}  q_0=(\nab_{A_0} v_0):(\nab_{A_0} v_0)-(\nab_{A_0} \bz):(\nab_{A_0} \bz)},\quad \text{in}\,\, \Omega, 
\end{equation}
with boundary conditions
\begin{equation}\label{flat bdy cond}
\begin{aligned}
q_0=\red{\sigma \mathcal{H}_0},\quad \text{on}\,\,\,\Gamma,\\
\frac{\p q_0}{\p N}=0,\quad \text{on}\,\,\,\Gamma_0.
\end{aligned}
\end{equation}
 
Note that the Neumann boundary condition  follows from restricting the normal component of the momentum equation to $\Gamma_0$. In \cite{luozhangMHDST3.5}, the second and third authors proved the local a priori energy estimate of \eqref{MHDLST} for each fixed $\sigma>0$. The local well-posedness of this problem is established by us very recently in \cite{GLZ}. 

On the other hand, when $\sigma=0$, let $(\zeta, w, r)$ verifies
\begin{equation}\label{MHDL0ST}
\begin{cases}
\p_t\zeta=w~~~& \text{in}~[0,T]\times\Omega;\\
\p_tw-\bp^2\zeta+\nab_{A(\zeta)} r=0~~~& \text{in}~[0,T]\times\Omega;\\
\di_{A(\zeta)}  w =0,&\text{in}~[0,T]\times\Omega;\\
\di  \bz=0~~~&\text{in}~\{0\}\times \Omega ;\\
w^3=\bz^3=0~~~&\text{on}~\Gamma_0;\\
r=0 ~~~&\text{on}~\Gamma;\\
\bz^3=0 ~~~&\text{on}~\Gamma,\\
(\zeta,w)=(\red{\eta_0},v_0)~~~&\text{on}~\{t=0\}{\times}\overline{\Omega}.
\end{cases}
\end{equation}
The initial pressure $r_0$ is determined by the elliptic equation
\begin{equation}
\red{-\Delta_{A_0} r_0=(\nab_{A_0} v_0):(\nab_{A_0} v_0)-(\nab_{A_0} \bz):(\nab_{A_0} \bz)},\quad \text{in}\,\, \Omega,
\end{equation}
with the boundary condition $r_0=0$ on $\Gamma$ and $\p r_0/\p N=0$ on $\Gamma_0$.
Moreover, for \eqref{MHDL0ST} to be well-posed, we assume the Rayleigh-Taylor sign condition
\begin{equation}
-\p_3 r_0 \geq c_0 >0,\quad \text{on}\,\,\Gamma \label{TaylorL'}
\end{equation}
holds for some constant $c_0>0$.  
\begin{rmk}
\red{In fact, the original Rayleigh-Taylor sign condition for \eqref{MHDL0ST} is $-\nn_0\cdot\nabla_{A_0(\zeta_0)}r_0\geq c_0>0$. But since $r=0$ on $\Gamma$ implies $\TP_1r=\TP_2r=0$ on $\Gamma$, then the sign condition becomes $-\sqrt{g_0(\zeta_0)}\nn_{0\alpha}\nn_{0}^{\alpha}\p_3 r_0\geq c_0>0$. So it is reasonable to directly assume $-\p_3 r_0\geq c_0>0.$}
\end{rmk}

\subsection{The main result}
This paper aims to show that if $q_0$ verifies the Rayleigh-Taylor sign condition, i.e., 
\begin{equation}\label{TaylorL}
-\p_3 q_0 \geq c_0 >0,\quad \text{on}\,\,\Gamma,
\end{equation} 
then  the solution of the ``$\sigma>0$ problem" \eqref{MHDLST} converges to the solution of the ``$\sigma=0$ problem" \eqref{MHDL0ST} as $\sigma\rightarrow 0$. Specifically, we prove
\begin{thm} \label{main converge}
Suppose the initial data $(v_0, \bz, q_0)$ verifies:
\begin{enumerate}
\item $v_0, \bz$ are divergence-free vector fields with $v_0^3 = 0 $ on $\Gamma_0$ and  $\bz^3 =0$ on $\Gamma\cup\Gamma_0$.
\item \red{$\eta_0, v_0, \bz \in H^{5.5}(\Omega)$; $\eta_0,v_0\in H^7(\Gamma)$; the initial mean curvature $\hh_0=\Delta_{g_0}\eta_0\cdot\nn_0\in H^6(\Gamma)$.}
\item The Rayleigh-Taylor sign condition 
\begin{equation}\label{RTsign}
-\p_3 q_0 \geq c_0>0,\quad\text{on}\,\,\Gamma.
\end{equation}
\item The compatibility conditions up to the $4$-th order, where the $j$-th order ($0\leq j\leq 4$) condition reads
\begin{equation} \label{compatibility}
\begin{aligned}
\p_t^j q(0)=\sigma\p_t^{j}\mathcal{H}(0), ~~~&\text{ on }\Gamma, \\
 \p_3 \p_t^j q(0)=0,~~~&\text{ on }\Gamma_0.
 \end{aligned}
\end{equation}
\end{enumerate}
Then there exists some $T>0$ independent of $\sigma$, such that the solutions to \eqref{MHDLST} and \eqref{MHDL0ST} exists, and as $\si\to 0$, we have
\begin{align}
(v^\sigma,\bp\eta^{\sigma}, q^\sigma)\xrightarrow{\red{C^1([0,T]\times\Omega)}} (w,\bp\zeta, r).
\end{align}
\end{thm}
\noindent We will drop the superscript $\sigma$ on $(\eta^\sigma, v^\sigma, q^\sigma)$ in the rest of this paper and simply denote them by $(\eta, v, q)$ when no confusion can be raised. 

\nota (The Sobolev norms) Let $f=f(t,y)$ be a smooth function on $[0,T]\times\Omega$ and $g=g(t,y)$ be a smooth function on $[0,T]\times \Gamma$.  We define $\|f\|_{s}: = \|f(t,\cdot)\|_{H^s(\Omega)}$ and $|g|_{s}: = |g(t,\cdot)|_{H^s(\Gamma)}$ throughout the rest of this paper. 

Theorem \ref{main converge} is a direct consequence of the following theorem where we establish an energy estimate of the $\sigma>0$ problem that is uniform as $\sigma\to 0$ via the standard compactness argument, which is also proved in section \ref{sect 0STlimit}.
\begin{thm}\label{main energy}
Let $(v_0, \bz, q_0)$ be the initial data given in Theorem \ref{main converge}. Let
\begin{equation}\label{energy}
E(t)=E_1(t)+\sigma E_2(t), 
\end{equation}
with
\begin{equation}\label{energy1}
E_1(t):=\left\|\eta (t)\right\|_{5}^2+\sum_{k=0}^5\left(\left\|\p_t^k v(t)\right\|_{5-k}^2+\left\|\p_t^k \bp\eta(t)\right\|_{5-k}^2\right)+\left|\TP^5\eta(t)\cdot\nn(t)\right|_0^2
\end{equation}
and
\begin{equation}\label{energy2}
E_2(t):=\left\|\eta(t)\right\|_{5.5}^2+\sum_{k=0}^4\left(\left\|\p_t^k v(t)\right\|_{5.5-k}^2+\left\|\p_t^k \bp\eta(t)\right\|_{5.5-k}^2\right)+\sum_{k=0}^5\left|\TP^{6-k}\p_t^k\eta(t)\cdot\nn(t)\right|_0^2+\left|\TP^5\bp\eta(t)\cdot\nn(t)\right|_0^2,
\end{equation}
where $\nn$ is the Eulerian unit normal vector.
\red{Then} 
\begin{equation}
\red{E(t) \leq P(E(0)),}\label{energy estimate E(t)}
\end{equation}
\red{where $P$ is a generic polynomial in its arguments. Also, 
there exists a constant
$$
\mathcal{C} = \mathcal{C}\left(c_0,\|\eta_0,v_0,\bz\|_5,\,\wt\|\eta_0,v_0,\bz\|_{5.5},\,\sum_{k=0}^{3}\sigma^{1+\frac{k}{2}}\left|\eta_0,v_0\right|_{5.5+\frac{k}{2}},\,\sum_{k=0}^{3}\sigma^{1+\frac{k}{2}}|\Delta_{g_0}\eta_0\cdot\nn_0|_{4.5+\frac{k}{2}}\right)>0
$$
such that}
%a constant $\mathcal{C}=\mathcal{C}_1+\mathcal{C}_2$, where $\mathcal{C}_1$ depending on $c_0$, $\|v_0\|_{5}, \|\bz\|_{5}$, $|\sigma \TL v_0^3|_{3.5},|\sigma \TL \bz^3|_{3.5}$,  and $\mathcal{C}_2$ depending on $\|\sqrt{\si}v_0\|_{5.5}, \|\sqrt{\si}\bz\|_{5.5}$, $|\sigma^{\frac{3}{2}}\TL v_0^3|_4,|\sigma^{\frac{3}{2}}\TL \bz^3|_4$ such that 
\begin{align}
\red{E(t) \leq \mathcal{C}, \quad \forall t\in [0,T].} \label{data representation}
\end{align}

%Moreover, part of the higher regularity of $\eta$, $v$ on the boundary can be recovered at a later time, i.e., 
%\begin{equation}\label{recover v} 
%\red{\sigma |\Delta_g\eta(t)\cdot\nn|_{4.5}+\sigma^{\frac32} |\Delta_g\eta(t)\cdot\nn|_{5}+\sigma |\eta(t)|_{5.5}+\sigma^{\frac{3}{2}} |\eta(t)|_6+\sigma |v(t)|_{5.5} + \sigma^{\frac{3}{2}} |v(t)|_6 \leq P(E(t)),\quad \forall t\in (0,T]. }
%\textcolor{red}{|\sigma \Delta_g \eta(t)\cdot\nn|_{4.5}+|\sigma \Delta_g v(t)\cdot\nn|_{3.5}+|\sigma \Delta_g b(t)\cdot\nn|_{3.5}+|\sigma^{\frac32} \Delta_g \eta(t)\cdot\nn|_{5}+|\sigma^{\frac32} \Delta_g v(t)\cdot\nn|_{4}+|\sigma^{\frac32} \Delta_g b(t)\cdot\nn|_{4}}\leq P(E(t)),\quad \forall t\in (0,T].
%\end{equation}
\end{thm}
Theorem \ref{main energy} states that given the initial data for the $\sigma=0$ problem, we want to prove the a priori energy estimate for the $\sigma>0$ problem, and such energy estimate has to be uniform in $\sigma$ as $\sigma$ tends to zero. That said, it is natural to come up with the energy \eqref{energy}, in which $E_1(t)$ corresponds to the $\sigma=0$ problem and $E_2(t)$ corresponds to the $\sigma>0$ problem. Also, $E(t)$ reduces to $E_1(t)$ when $\sigma=0$.  

\red{Let $\PP_0:=P(E(0))$. The energy estimate \eqref{energy estimate E(t)} is a direct consequence of 
\[
E(T) \leq \PP_0 +\epsilon E(T) + P(E(T))\int_0^T P(E(t))\dt
\] 
together with the Gr\"onwall's inequality.  }

\begin{rmk}
\red{The reason for us to require $\eta_0$ and $v_0$ to have higher boundary regularity is that we need to express $E(0)$, consisting of
\[
\p_t^k v(0), \quad \p_t^k b(0),\quad 1\leq k\leq 5,\quad \text{in}\,\,\,\,H^{5-k},
\]
\[
\sqrt{\sigma}\p_t^k v(0), \quad \sqrt{\sigma}\p_t^k b(0),\quad 1\leq k\leq 4,\quad \text{in}\,\,\,\,H^{5.5-k},
\]
 and 
\[
\p_t^\ell q(0),\quad 0\leq \ell\leq 4,\quad \text{in}\,\,\,\,H^{5-\ell}, 
\]
\[
\sqrt{\sigma}\p_t^\ell q(0),\quad 0\leq \ell\leq 3,\quad \text{in}\,\,\,\,H^{5.5-\ell}, 
\]
 in terms of $\eta_0$, $v_0$ and $\bz$.}

 \red{Specifically, the boundary regularity for $\eta_0$ is required when solving the initial data $\|v_t(0)\|_4\lesssim\|q_0\|_5+\|\bz\|_4\|\bz\|_5$ and $\|\sqrt{\sigma} v_t(0)\|_{4.5}\lesssim\|\sqrt{\sigma}q_0\|_{5.5}+\|\bz\|_{4.5}\|\sqrt{\sigma} \bz\|_{5.5}$. Then $q_0$ is determined by the elliptic equation with \textit{Dirichlet} boundary condition\footnote{\red{We cannot employ the Neumann boundary condition here since this would generate $v_{tt}(0)$. }} 
 on $\Gamma$
\[
\begin{cases}
\Delta_{A(0)}q_0=\nabla_{A(0)}^{\alpha} v_{0,\beta}\nabla_{A(0)}^{\beta} v_{0,\alpha}-\nabla_{A(0)}^{\alpha} b_{0,\beta}\nabla_{A(0)}^{\beta} b_{0,\alpha}  ~~~&\text{ in }\Omega\\
q_0=-\sigma\Delta_{g_0}\eta_0\cdot\nn_0~~~&\text{ on }\Gamma\\
\dfrac{\p q_0}{\p N}=0~~~&\text{ on }\Gamma_0.
\end{cases}
\]
The elliptic estimates for $\|q_0\|_5$ and $\|\sqrt{\sigma}q_0\|_{5.5}$ require the bounds for $|\sigma\Delta_{g_0}\eta_0\cdot\nn_0|_{4.5}$ and $|\sigma^{3/2}\Delta_{g_0}\eta_0\cdot\nn_0|_5$. Similarly, when solving $\p_{tt} v(0)$ in $H^3(\Omega)$ and $\sqrt{\si} \p_{tt} v(0)$ in $H^{3.5}(\Omega)$, we need the bounds for $\sigma | \p_t (\lap_{g} \eta \cdot \nn )(0)|_{3.5}$ and $\sigma^{\frac{3}{2}} |\p_t (\lap_g \eta\cdot \nn)(0)|_{4}$. These contribute to $\sigma |v_0|_{5.5},\sigma^{\frac{3}{2}} |v_0|_6$ and $\sigma |\eta_0|_{5.5},\sigma^{\frac{3}{2}} |\eta_0|_6$ at the leading order. }

\red{By continuing this process, we need to study the equations verified by $\p_t^k v(0)$, $k\geq 3$, where the Dirichlet boundary condition involves $\lap_{g_0} \p_t^\ell v\cdot \nn_0$, $\ell\geq 1$, and this introduces higher-order norms of $q$ (and its time derivatives). Then one has to perform higher-order elliptic estimates for $q$ (and its time derivatives) with the Dirichlet boundary condition. This results in the higher order boundary norms of $\eta_0$ and $v_0$ in the constant $\mathcal{C}$.} The detailed arguments to can be found in Appendix \ref{appendix}.  
%But for the solution, one may not enhance the regularity of $q$ in this way, and thus some of the higher-order regularity, namely the terms when $k=2,3$ in the arguments of $\mathcal{C}$, cannot be recovered. 
\end{rmk}
 
\begin{rmk}
\red{Note that $\nn(0) = (0,0,1)^\top$ and $g_0 = \delta$ when $\eta_0=\id$. Thanks to this,  the higher-order boundary terms of requiring only the third component of $\eta_0$ and $v_0$ in the constant $\mathcal{C}$. This agrees with the constant $C_0$ introduced in  \cite[Theorem 2.3]{DK2017}.}
\end{rmk}
 
 \begin{rmk}
\red{The initial data $v_0, \bz,\eta_0$ are the same for both $\sigma>0$ and $\sigma=0$ problems, but $q_0^\sigma \neq r_0$ since $q_0^\sigma =\sigma \mathcal{H}_0 \neq 0$ on $\Gamma$.} Therefore, the initial data for the $\sigma>0$ problem should be $(\eta_0, v_0, \bz, q_0^\sigma)$, where $q_0^\sigma$ verifies \eqref{TaylorL} for all $\sigma>0$ and $q_0^\sigma \rightarrow r_0$ as $\sigma\rightarrow 0$. 
\end{rmk}

\subsection{Key new difficulties and comparison with the existing literature}
\subsubsection{Difference between Euler and MHD equations}
The zero surface tension limit for the free-boundary (compressible) Euler equations is studied in \cite{coutand2013LWP}, but the method developed there does not apply to the free-boundary MHD equations. Their treatment depends on the assumption that $\eta$ is $1/2$-derivatives more regular in the interior than $v$ when passing $\sigma$ to $0$. This can be done by assuming the vorticity is more regular initially\footnotemark and this extra regularity can then be carried over to $\eta$ through the Cauchy invariance.  However, in MHD equations, the coupling between magnetic field and velocity denies the possibility of the higher regularity assumption on the vorticity and thus on $\eta$. This is due to that the Lorentz force (i.e., $B\cdot \nab B$) in the momentum equation destroys the Cauchy invariance. 
\footnotetext{It is also well-known that in Euler equations, the extra regularity assumption on vorticity can be propagated to a later time. }

On the other hand, we mention here that the zero surface tension limit of $2$D Euler equations was established by Ambrose-Masmoudi \cite{AM2005} but without vorticity. In this paper, we develop a unified framework for proving the zero surface tension limit for both incompressible Euler and MHD equations in $3$D, without imposing the irrotational assumption\footnotemark on $v_0$. That is, we can recover the zero surface tension limit of the $3$D Euler equations by setting $\bz=0$. 
\footnotetext{Physically, the Lorentz force twists the velocity field and thus introduces vorticity. In consequence, the irrotational assumption on the velocity becomes inadequate in MHD equations. }

\subsubsection{Treating the higher-order interior and boundary terms}
The aforementioned regularity issue prevents us from commuting $\TP^5$ with the equation 
$$
\p_t v - \bp^2 \eta +\nab_A q =0,
$$
because we are unable to control the term generated when all derivatives land on $A$. Our method to overcome this issue is to invoke the so-called Alinhac good unknowns for the leading order terms involving $v$ and $q$ (i.e., \eqref{AGU intro}).  The structure of these good unknowns is tied to the covariant differentiation in the Eulerian coordinates, and thus the higher-order term will not show up. Nevertheless, the use of good unknowns alone cannot fully resolve the problem, as the surface tension introduces higher-order boundary terms in the energy estimate (i.e., the integrals on the RHS of \eqref{higher-order terms intro}). None of these integrals can be controlled directly, but fortunately, we can overcome this issue by exploiting the structure of the equations as well as that of the good unknowns to generate cancellation schemes. We refer to Section \ref{sect. AGU intro} for the detailed explanations.

\subsubsection{Application to other free-boundary fluid models}
The method developed in this manuscript is applicable to study the zero surface limit problem for Euler equations without the extra regularity assumption on $\eta$. In particular, Theorem \ref{main energy} yields the uniform-in-$\sigma$ energy estimate for the free boundary incompressible Euler equations by setting $\bz=0$. 
Also, it is possible to adapt our method to study the zero surface tension limit in other complex fluid models, where no extra regularity assumption can be made on the flow map $\eta$. 

\section{Strategy of the proof and some auxiliary results}
\subsection{Proof of Theorem \ref{main energy}: An overview}

\subsubsection {Necessity of time derivatives} 
The energy \eqref{energy} consists of mixed space and time derivatives, and this is required as a result of the estimate of the pressure $q$. Specifically, we cannot equip the elliptic equation verified by $q$ with Dirichlet boundary condition when $\sigma>0$ (see, e.g.,  \cite{christodoulou2000motion,gu2016construction}). Instead, we should impose the Neumann boundary condition, which contains the time derivative of $v$ and thus forces us to analyze all the time derivatives of the variables $v$ and $b=\bp\eta$. 

\subsubsection{Interior tangential estimates: Alinhac good unknowns and cancellation structures} \label{sect. AGU intro}

The non-weighted full Sobolev norms are analyzed via div-curl analysis. The curl part can be controlled via the evolution equation of the Eulerian vorticity. As for tangential estimates, if the tangential derivatives contain at least one time derivatives, one may follow the ideas in our previous work \cite{luozhangMHDST3.5, GLZ} to close the energy estimates and we refer to Section \ref{sect tgmix} for the proof. 

However, when the tangential derivatives are purely spatial, we have to introduce the Alinhac good unknowns to proceed with the energy estimates because we cannot directly commute $\TP^5$ with the covariant derivative $\pa$ falling on $q$ or $v$. Instead, we rewrite the term $\TP^5 (\pa q)$ in terms of the covariant derivative of the Alinhac good unknowns plus a controllable ``error" term. See Section \ref{sect tgs5} for the proof. This remarkable observation was first due to Alinhac \cite{Alinhacgood89}. In the study of free-surface fluid, it is first (implicitly) applied to the $Q$-tensor energy introduced by Christodoulou-Lindblad \cite{christodoulou2000motion}. Here, following \cite[Section 4]{gu2016construction}, the Alinhac good unknowns of $v$ and $q$ with respect to $\TP^5$ are 
\begin{align}
\VV:=\TP^5v-\TP^5\eta\cdot\pa v,~~\QQ:=\TP^5 q-\TP^5\eta\cdot\pa q. \label{AGU intro}
\end{align}
Under this setting, it suffices to analyze the evolution equation of the good unknowns to derive the $\TP^5$-estimates.

Compared with the ``$\sigma=0$ problem", the boundary integral appearing in the tangential estimates, which reads $J:=-\ig A^{3\alpha}\QQ\VV_{\alpha}\dS$, becomes very difficult because the top order derivative of the pressure $\TP^5 q$ no longer vanishes due to the presence of surface tension. It contributes to
\red{\begin{align}
-\ig A^{3\alpha}\TP^5 q\VV_{\alpha}dS=&-\ig\TP^5(A^{3\alpha} q)\VV_{\alpha}\dS+\ig q(\TP^5 A^{3\alpha})\VV_{\alpha}\dS+\ig A^{3\alpha}(\TP^5\eta\cdot\pa q)\VV_{\alpha}\dS.+\text{lower order terms} \label{higher-order terms intro} \\
=&:\ST+J_1+\RT+\text{lower order terms},  \nonumber
\end{align} }

Due to the presence of $\p_3 q$, the term$\RT$ can be directly controlled by invoking the Rayleigh-Taylor sign condition and a standard cancellation structure of the Alinhac good unknowns which relies on the simple identity $\p_t A^{3\alpha}=-A^{3\gamma}\p_{\mu}v_{\gamma}A^{\mu\alpha}$ and also appears in the ``$\sigma=0$ problem". This part contributes to the non-weighted boundary energy $|\TP^5\eta\cdot\nn|_0^2$ that exactly controls the regularity of the second fundamental form of the free surface. The term$\ST$ contributes to the $\wt$-weighted energy $|\wt\TP^6\eta\cdot\nn|_0^2$ after plugging the surface tension equation $A^{3\alpha}q=-\sigma\sqrt{g}\Delta_g\eta^{\alpha}$ and integrating by parts, where the error terms can be either directly controlled or eliminated by using the structure of the good unknown $\VV$ (e.g., \eqref{ST1134}).

Finally, $\TP^5 A^{3\alpha}$ in $J_1$ has the top order contribution $\TP^6\eta\times \TP\eta$ that cannot be directly controlled. To overcome this difficulty, we write $A^{3\alpha}=(\TP_1\eta \times\TP_2\eta)^\alpha$ and use it to observe a crucial symmetric structure on the boundary, while the 2D version of this symmetric structure was developed in \cite{GuLeielastoST} and plays an important role in the local well-posedness of free-surface incompressible elastodynamics.

We illustrate the observation of symmetric structure briefly. By plugging  $A^{3\alpha}=(\TP_1\eta \times\TP_2\eta)^{\alpha}$ in $J_1$, the highest order terms are all of the form
\[
J_{11}=\ig\sigma\hh(\TP^5\TP_1\eta\times\TP_2\eta)\cdot\TP^5v\dS,
\]which cannot be controlled directly. To resolve this, we re-express $J_{11}$ as
\begin{equation*}
	\begin{split}
            J_{11}=&\dfrac{d}{dt}\ig\sigma\hh(\TP^5\TP_1\eta\times\TP_2\eta)\cdot\TP^5\eta\dS-\ig\sigma\hh(\TP^5\TP_1 v\times\TP_2\eta)\cdot\TP^5 \eta\dS-\ig\sigma\hh(\TP^5\TP_1 \eta\times\TP_2v)\cdot\TP^5 \eta\dS\\
                   &-\ig\sigma\p_t\hh(\TP^5\TP_1 \eta\times\TP_2\eta)\cdot\TP^5 \eta\dS\\
                   =&\dfrac{d}{dt}\ig\sigma\hh(\TP^5\TP_1\eta\times\TP_2\eta)\cdot\TP^5\eta\dS\underbrace{+\ig\sigma\hh(\TP^5v\times\TP_2\eta)\cdot\TP^5\TP_1 \eta\dS}_{\mathfrak A}+\ig\sigma(\TP^5v\times\TP_1(\hh\TP_2\eta))\cdot\TP^5 \eta\dS\\&-\ig\sigma\hh(\TP^5\TP_1 \eta\times\TP_2v)\cdot\TP^5 \eta\dS-\ig\sigma\p_t\hh(\TP^5\TP_1 \eta\times\TP_2\eta)\cdot\TP^5 \eta\dS
	\end{split}
\end{equation*}
Although we do not have control for $\mathfrak A$, by the antisymmetry of the vector identity $(\uu\times\vv)\cdot\ww=-(\uu\times\ww)\cdot\vv$, we have
$$(\TP^5v\times\TP_2\eta)\cdot\TP^5\TP_1 \eta=-(\TP^5\TP_1\eta\times\TP_2\eta)\cdot\TP^5v,$$
and thus $\mathfrak A=-J_{11}$. This implies
\begin{equation*}
        \begin{split}
                J_{11}=&\dfrac{1}{2}\dfrac{d}{dt}\ig\sigma\hh(\TP^5\TP_1 \eta\times\TP_2\eta)\cdot\TP^5\eta\dS+...
        \end{split}    
\end{equation*}
Nevertheless, the first term on the RHS does not have a positive sign. In consequence, we need to control it by the $\wt$-weighted energy $|\wt\TP^6\eta\cdot\nn|_0^2$. To achieve this, we need to invoke the decomposition $\delta^{\alpha\beta}=\nn^{\alpha}\nn^{\beta}+\tau^{\alpha}\tau^{\beta}=\nn^{\alpha}\nn^{\beta}+g^{ij}\TP_i\eta_{\alpha}\TP_j\eta_{\beta}$, where $\tau$ denotes the unit Eulerian tangential vector, and use the vector identity $\uu\times(\vv\times\ww)=(\uu\cdot\ww)\vv-(\uu\cdot\vv)\ww$ to obtain good cancellations of the error terms.  We refer to section \ref{sect RTSTerror} for the details.

\begin{rmk}
The aforementioned cancellation schemes appear to be crucial when treating the higher-order terms$\ST$ and $J_1$. Nevertheless, neither of these terms would appear in the case when Alinhac good unknowns are not needed. For instance, the good unknowns are not employed when proving the local existence of the $\sigma>0$ problem with fixed $\sigma$ in \cite{GLZ} since we do not need to consider the energy $E_1(t)$ (defined in \eqref{energy1}). In addition to this, although the zero surface tension limit is proved in \cite{coutand2013LWP}, no good unknowns are required owing to the extra regularity assumption on $\eta$. 
\end{rmk}

\subsubsection{Necessity of weighted energy and the control via surface tension} \label{sect. weighted energy intro}

In the tangential estimates, especially in the boundary integrals, many terms have $5$ derivatives weighted by the surface tension coefficients. Therefore, it is reasonable to include the weighted $H^{5.5}$-energy $\sigma E_2(t)$ to control these boundary terms via the trace lemma. \red{To control the weighted higher-order energy $\sigma E_2(t)$, we again do the div-curl decomposition, but no longer convert the $\wt$-weighted normal trace terms to the interior tangential estimates. Instead, we notice that the boundary energies contributed by the surface tension in the \textit{non-weighted tangential estimates} are exactly the $\wt$-weighted Eulerian normal traces with the same order as those $\wt$-weighted normal traces. Hence, the energy estimates for $E(t)=E_1(t)+\sigma E_2(t)$ are closed. The detailed discussion can be found in Section \ref{sect weighted}.}

\subsection{The auxiliary results}\label{prelemma}
In this subsection, we record some well-known results that will be used frequently (and sometimes silently) throughout the rest of this manuscript. 
\subsubsection{Geometric identities}
\begin{lem}\label{geometric}
Let $\nn$ be the unit outer normal to $\eta(\Gamma)$ and $\mathcal{T},\mathcal{N}$ be the tangential and normal bundle of $\eta(\Gamma)$ respectively. Denote $\Pi:\mathcal{T}|_{\eta(\Gamma)}\to\mathcal{N}$ to be the canonical normal projection. Denote $\TP_A$ to be $\p_t$ or $\TP_1,\TP_2$. Then we have the identities
\begin{align}
\label{n} \nn := n\circ\eta=&\frac{A^\top N}{|A^\top N|},\\
\label{an}|A^\top N|=&|(A^{31},A^{32},A^{33})|=\sqrt{g},\\
\label{Pi} \Pi^{\alpha}_{\lambda}=&\nn^{\alpha}\nn_{\lambda}=\delta^{\alpha}_{\lambda}-g^{kl}\TP_k\eta_\alpha\TP_l\eta_\lambda,\\
\label{Pipro} \Pi^{\alpha}_{\lambda}=&\Pi^{\alpha}_{\mu}\Pi^{\mu}_{\lambda},\\
\label{lapg0} -\Delta_g(\eta^{\alpha}|_{\Gamma})=&\mathcal{H}\circ\eta \nn^\alpha, \\
\label{lapg}\sqrt{g}\Delta_g\eta^{\alpha}=&\sqrt{g}g^{ij}\Pi^{\alpha}_{\lambda}\TP_i\TP_j\eta^{\lambda}=\sqrt{g}g^{ij}\TP_i\TP_j\eta^{\alpha}-\sqrt{g}g^{ij}g^{kl}\TP_k\eta^{\alpha}\TP_l\eta^{\mu}\TP_i\TP_j\eta_{\mu},\\
\label{tplapg}\TP_{A}(\sqrt{g}\Delta_g\eta^{\alpha})=&\TP_i\bigg(\sqrt{g}g^{ij}\Pi^{\alpha}_{\lambda}\TP_A\TP_j\eta^{\lambda}+\sqrt{g}(g^{ij}g^{kl}-g^{ik}g^{lj})\TP_j\eta^{\alpha}\TP_k\eta_{\lambda}\TP_A\TP_l\eta^{\lambda}\bigg),\\
\label{tpn} \TP_{A} \nn_{\mu} =& -g^{kl}\cp_k \TP_{A} \eta^\tau \nn_{\tau} \cp_l \eta_\mu,\\
\label{pt ggij} \p_t (\sqrt{g}g^{ij})=& \sqrt{g} (g^{ij}g^{kl}-2g^{lj}g^{ik})\TP_k v^\lambda \TP_l \eta_\lambda,\\
\label{tp ggij} \TP(\sqrt{g}g^{ij})=&\sqrt{g}\left(\frac12 g^{ij}g^{kl}-g^{ik}g^{jl}\right)(\underbrace{\TP\TP_k\eta^{\mu}\TP_l\eta_{\mu}+\TP_k\eta^{\mu}\TP\TP_l\eta_{\mu}}_{=\TP g_{kl}}).
\end{align}
\end{lem}
\begin{proof}
See Disconzi-Kukavica \cite[Lemma 2.5]{DK2017}.
\end{proof}

\nota \label{Q notation} We shall use the notation $Q(\p\eta)$ and $Q(\TP\eta)$ to denote the rational functions of $\p\eta$ and $\TP\eta$, respectively. 

This $Q$ notation allows us to record error terms concisely. For example, for any tangential derivative $\TP_A$, we have $\TP_AQ(\TP\eta)=\tilde{Q}^i_{\alpha}(\TP\eta)\TP_A\TP_i\eta^{\alpha}$ where the term $\tilde{Q}^i_{\alpha}(\TP\eta)$ is also a rational function of $\TP\eta$. Also, recall that $g_{ij}=\TP_i\eta_{\mu}\TP_j\eta^{\mu}$ and $g=\det[g_{ij}]$ and $[g^{ij}]=[g_{ij}]^{-1}$. This means that $g_{ij},~g$ and $g^{ij}$ are rational functions of $\TP\eta$ and so is $\Pi$. 

The following lemma will be employed frequently in the rest of this paper. 

\assump  \label{apriori assumpt}
\red{In the sequel of this paper, we assume $\|\eta\|_{5}, \|\sqrt{\sigma}\eta\|_{5.5}\leq N_0$ a priori for some $N_0>0$, where $N_0$ depends on $T$ and $\|\eta_0\|_{5.5})$. }

\red{This a priori assumption can be straightforwardly recovered thanks to $\eta= \eta_0 +\int_0^T v\dt$. }

\assump  \label{RTapriori}
\red{In the sequel of this paper, we assume the Rayleigh-Taylor sign function $-\p_3 q\geq\frac{c_0}{2}>0$ a priori in some time interval $[0,T]$. This a priori assumption can be justified by using Sobolev embedding: One can later show that $-\p_3 q$ is $C^{0,\frac14}([0,T]\times\Gamma)$-continuous and thus the positivity can be propagated to a positive time interval. See also \cite[Section 5.2]{luozhangMHD2.5}. }

\subsubsection{Sobolev inequalities}

First, we list the Kato-Ponce inequality and its corollary which will be used in nonlinear product estimates.

\begin{lem}[\textbf{Kato-Ponce type inequalities}]\label{katoponce}  Let $J=(I-\Delta)^{1/2},~s\geq 0$. Let $f,g$ be smooth functions. Then the following estimates hold:

(1) $\forall s\geq 0$, we have 
\begin{equation}\label{product}
\begin{aligned}
\|J^s(fg)\|_{L^2}&\lesssim \|f\|_{W^{s,p_1}}\|g\|_{L^{p_2}}+\|f\|_{L^{q_1}}\|g\|_{W^{s,q_2}},\\
\|\p^s(fg)\|_{L^2}&\lesssim \|f\|_{\dot{W}^{s,p_1}}\|g\|_{L^{p_2}}+\|f\|_{L^{q_1}}\|g\|_{\dot{W}^{s,q_2}},
\end{aligned}
\end{equation}with $1/2=1/p_1+1/p_2=1/q_1+1/q_2$ and $2\leq p_1,q_2<\infty$;

(2) $\forall s\geq 1$, we have
\begin{equation}\label{kato3}
\|J^s(fg)-(J^sf)g-f(J^sg)\|_{L^p}\lesssim\|f\|_{W^{1,p_1}}\|g\|_{W^{s-1,p_2}}+\|f\|_{W^{s-1,q_1}}\|g\|_{W^{1,q_2}}
\end{equation} for all the $1<p<p_1,p_2,q_1,q_2<\infty$ with $1/p_1+1/p_2=1/q_1+1/q_2=1/p$.
\end{lem}

\begin{proof}
See Kato-Ponce \cite{kato1988commutator}.
\end{proof}

The inequalities listed in the following corollary shall come in handy when estimating products in fractional Sobolev spaces.  All of them are direct consequences of \eqref{product}.
\begin{cor}
Let $f,g$ be given as above. We have
\begin{align}
\|fg\|_{0.5}&\lesssim \|f\|_{0.5}\|g\|_{1.5+\delta}\label{product0.5},\\
\|fg\|_{s} &\lesssim \|f\|_{s}\|g\|_{1.5+\delta}+\|f\|_{1.5+\delta}\|g\|_{s},\quad s\geq 1.5. \label{product s}
%\|fg\|_{2.5} \lesssim \|f\|_{2.5}\|g\|_{2.5}\label{product2.5}.
\end{align}
\end{cor}
\begin{proof}
\eqref{product0.5} follows from setting $s=0.5$, $p_1=2$, $p_2=\infty$, and $q_1=q_2=4$ in the first inequality of \eqref{product}. \eqref{product s} follows from setting $p_1=2$, $p_2=\infty$, $q_1=\infty$ and $q_2=2$ in the second inequality of \eqref{product}. 
\end{proof}

%\begin{lem}[\blue{\textbf{Normal trace lemma}}]\label{normaltrace}
%Let $X$ be a smooth vector field. Then
%\begin{equation}\label{ntr}
%\left|\TP X\cdot N\right|_{-0.5}\lesssim\|\TP X\|_0+\|\dive X\|_0
%\end{equation}
%\end{lem}
%\begin{proof}
%This can be proved by testing a $H^{0.5}(\Gamma)$ function and the divergence theorem. See \cite[Lemma 3.4]{gu2016construction}.
%\end{proof}

\subsubsection{Elliptic estimates}
We present the Hodge-type div-curl estimate, which will be adapted to study the full Sobolev norms of $v$ and $\bp \eta$. 

\begin{lem} [The Hodge-type elliptic estimate]\label{hodge}
Let $X$ be a smooth vector field and $s\geq 1$, then it holds that
\begin{align}
\label{divcurl1} \red{\|X\|_{s}^2\lesssim C(\|\TP\eta\|_{W^{1,\infty}},\|\eta\|_s^2)\left(\|X\|_0^2+\|\pa\cdot X\|_{s-1}^2+\|\pa\times X\|_{s-1}^2+\|\TP^s X\|_0^2\right),}\\
\label{divcurl2} \red{\|X\|_{s}^2\lesssim C(\|\TP\eta\|_{W^{1,\infty}},\|\eta\|_s^2)\left(\|X\|_0^2+\|\pa\cdot X\|_{s-1}^2+\|\pa\times X\|_{s-1}^2+|\TP^{s-\frac12} X\cdot\nn|_0^2\right),}
\end{align} \red{where the constant $C(\|\TP\eta\|_{W^{1,\infty}},\|\eta\|_s^2)$ depends linearly on $\|\eta\|_s^2$, and $\nn$ is the unit Eulerian normal vector. }
\end{lem}

\begin{proof}
\red{The first inequality is the same as Ginsberg-Lindblad-Luo \cite[Lemma B.2]{GLL2019LWP}. The second inequality is identical to Cheng-Shkoller \cite[Theorem 1.6]{hodgeinverse}.}
\end{proof}

\begin{rmk}
 \red{Note that the inequality \eqref{divcurl2} does not require the bound for $\|\eta\|_{s+\frac12}$. Indeed, $\|\eta\|_{s+\frac12}$ is only required if we replace the boundary term by $|X\cdot\nn|_{s-\frac12}^2$ (corresponding to the inequality (10) in Cheng-Shkoller \cite[Remark 1.4]{hodgeinverse}) because in that case $\TP^{s-\frac12}$ may fall on $\nn=\TP\eta\times\TP\eta$.}
\end{rmk}

Finally, the following $H^1$-elliptic estimates will be applied \red{to control $\|\dd^4 q\|_1$ and $\|\sqrt{\sigma}\TP^{\frac12}\dd^4 q\|_1$ for $\dd=\p_t$ or $\TP$. }
%Its proof can be found in \cite{ignatova2016}.
\begin{lem}[\textbf{Low regularity elliptic estimates}]\label{H1elliptic}
\red{Assume $\BB^{\mu\nu}$ satisfies $\|\BB\|_{L^{\infty}}\leq K$ and the ellipticity $\BB^{\mu\nu}(x)\xi_{\mu}\xi_{\nu}\geq \frac{1}{K}|\xi|^2$ for all $x\in\Omega$ and $\xi\in\mathbb{R}^3$. Assume $W$ to be an $H^1$ solution to 
\begin{equation}\label{lowelliptic}
\begin{cases}
-\p_{\nu}(\BB^{\mu\nu}\p_{\mu}W)=\dive \pi &\text{ in }\Omega \\
\BB^{\mu\nu}\p_{\mu}WN_{\nu}=h &\text{ on }\p\Omega,
\end{cases}
\end{equation} where $\pi,\dive \pi \in L^2(\Omega)$ and $h\in H^{-0.5}(\p\Omega)$ with the compatibility condition $$\int_{\p\Omega}(\pi\cdot N+h)dS=0.$$ 
Then we have:
\begin{equation}
\|W\|_{1}\lesssim_{\vol(\Omega)} \|\pi\|_{0}+|h+\pi\cdot N|_{-0.5}. %\text{ where }\overline{W}:=\frac{1}{|\Omega|}\int_{\Omega} W dy,
\end{equation}}
\end{lem}
\begin{proof}
\red{Testing the first equation of \eqref{lowelliptic} with $\psi\in H^1$ and then integrating by parts, we obtain
\begin{align}
\io \BB^{\mu\nu} \p_\mu W \p_\nu \psi\dy  = -\io \pi \cdot \nab \psi \dy + \int_{\p\Omega} (\pi\cdot N+h)\psi\dS \leq \|\pi\|_{0} \|\psi\|_{1} + |\pi\cdot N+h|_{-0.5} |\psi|_{0.5}.
\end{align}
Because $\BB$ is positive-definite, $\|\BB\|_{L^\infty} \leq K$ and $\psi $ is arbitrary, this yields 
\begin{align}
\|\p W\|_{0} \lesssim \|\pi\|_{0} \|\psi\|_{1} + |\pi\cdot N+h|_{-0.5} |\psi|_{0.5}.
\end{align}
Next, we need to control the $L^2$-norm of $W$. Invoking Poincar\'e's inequality $\|W-\overline{W}\|_0\lesssim \|\p W\|_0$ with $\overline{W}:=\frac{1}{\text{vol}(\Omega)}\io W\dy$, we have
\begin{equation}
\|W\|_0 \lesssim_{\vol (\Omega)} \|\p W\|_0 + \io W\dy.
\end{equation}
Let $Y=(y^1, 0,0)$. Then
\begin{align}
\io W\dy = \io \p_\mu Y^\mu W\dy = -\io y^1 \p_1 W \dy \lesssim_{\vol(\Omega)} \|\TP W\|_{0}. 
\end{align}
This concludes the proof. }
\end{proof}

\begin{rmk}
\red{Lemma \ref{H1elliptic} is essentially \cite[Lemma 3.2]{ignatova2016}. But we can drop the requirement on the smallness of $\|\BB - \mathbf{I}_3\|_{L^\infty}$, where $\mathbf{I}_3$ is the $3\times 3$ identity matrix, by invoking Poincar\'e's inequality. }
\end{rmk}

\section{Elliptic estimates for the pressure and its time derivatives}
In this section, we control the pressure $q$ and its time derivatives. The estimates presented in this section are essentially the same as the ones in \cite[Section~3]{GLZ}. Our conclusion is 
\begin{prop}\label{lem qell}
The total pressure $q$ \red{satisfies}
\begin{equation}\label{qell}
\sum_{k=0}^4\|\p_t^k q\|_{5-k}^2+\sum_{k=0}^3\|\sqrt{\sigma} \p_t^kq\|_{5.5-k}^2\leq P(E(t)).
\end{equation}
\end{prop}

%\subsection{Control of $\|q\|_5^2$ and $\sigma \|q\|_{5.5}^2$}

Due to the presence of the surface tension, we need to consider the elliptic equation verified by $q$ equipped with the Neumann boundary condition. We henceforth let $\lap_A q := \di_A (\nab_A q)=A^{\mu\alpha}\p_\mu (A^\nu_\alpha \p_\nu q)$.  \red{Taking $\di_A$ in the second equation of \eqref{MHDLST} and invoking the identity $\di_A v = \di_A \bp\eta=0$ and Piola's identity $\p_{\mu}A^{\mu\alpha}=0$, we obtain
\begin{equation}\label{qelliptic}
-\lap_A q = -\p_{\nu}(A^{\nu\alpha}{A^{\mu}}_{\alpha}\p_{\mu}q)= \pa^{\alpha}v_{\beta}\pa^{\beta}v_{\alpha}-\pa^{\alpha}\bp\eta_{\beta}\pa^{\beta}\bp\eta_{\alpha}.
\end{equation}}

\red{We define $\mathfrak{B}^{\mu\nu}=A^{\nu\alpha}{A^{\mu}}_{\alpha}$, i.e., $\mathfrak{B}=AA^\top$, then $\mathfrak{B}$ is a real symmetric positive-definite matrix and thus satisfies the assumption of Lemma \ref{H1elliptic}. In particular, we have $\mathfrak{B}^{33}\geq 1/K$ for some $K>0$. Therefore, we can express two normal derivatives of $q$ in terms of $\leq 1$ normal derivatives of $q$ plus the terms containing $v$ and $\bp\eta$ via the following identity}
\begin{equation}\label{qnormal}
\red{\p_3^2 q=\dfrac{1}{\mathfrak{B}^{33}}\left((\pa^{\alpha} v_{\beta})(\pa^{\beta}v_{\alpha})-(\pa^{\alpha} b_{\beta})(\pa^{\beta} b_{\alpha})-\sum_{\mu+\nu\neq 6}\p_{\mu}(\mathfrak{B}^{\mu\nu}\p_{\mu}q)-\p_3\mathfrak{B}^{3\nu}\p_{\nu}q\right).}
\end{equation} \red{Repeatedly, one can reduce up to 5 normal derivatives of $q$ to the control of $\leq 4$ derivatives of $v$ and $\bp\eta$, and $\leq 1$ \textit{normal} derivatives of $q$. We denote these terms of $q$ to be $\|\dd^4\p q\|_0$ where $\dd=\TP$ or $\p_t$. Such tangential derivatives will be controlled by using Lemma \ref{H1elliptic}.}

\red{To control these tangential derivatives of $\p q$, we first invoke the second equation of \eqref{MHDLST} to get the divergence form
\begin{equation}\label{qeq}
-\p_{\nu}(A^{\nu\alpha}{A^{\mu}}_{\alpha}\p_{\mu}q)=-\p_{\nu}\left(A^{\nu\alpha}(\p_tv-\bp^2\eta)_{\alpha}\right) \text{ in }\Omega,
\end{equation}
with Neumann boundary condition
\begin{equation}\label{qbdry}
A^{3\alpha}{A^{\mu}}_{\alpha}\p_{\mu} q=A^{3\alpha}(\p_t v-\bp^2\eta)_{\alpha}\text{ on }\Gamma.
\end{equation}}

\red{Now if we set 
\[
\BB^{\mu\nu}:=A^{\nu\alpha}{A^{\mu}}_{\alpha},\quad h_0:= \text{RHS of}\,\,\, \eqref{qbdry}
\]
and
\[
\pi_0^{\nu}:=-\left(A^{\nu\alpha}(\p_tv-\bp^2\eta)_{\alpha}\right)
\]
then the elliptic system \eqref{qeq}-\eqref{qbdry} is exactly of the form \eqref{lowelliptic} and satisfies $h_0+\pi_0\cdot N=0$ on $\Gamma$. Invoking Lemma \ref{H1elliptic}, we get the estimate for $\|q\|_1$
\begin{equation}\label{qH1}
\|q\|_1\lesssim \|\pi_0\|_0\lesssim P(N_0)(\|\p_t v\|_0+\|\bp\eta\|_1).
\end{equation}}

\red{Next we let $\dd=\TP$ or $\p_t$ be a tangential derivative and take $\dd^4$ in the elliptic equation \eqref{qeq}-\eqref{qbdry} . Using Piola's identity $\p_{\mu}A^{\mu\alpha}=0$, we can still write the $\dd^4$-differentiated equation into the divergence form 
\begin{equation}\label{qtttteq}
-\p_{\nu}(A^{\nu\alpha}A^{\mu\alpha}\dd^4\p_{\mu}q)=-\p_{\nu}\left(\left[A^{\nu\alpha}A^{\mu\alpha},\dd^4\right]\p_{\mu}q\right)-\p_{\nu}\dd^4\left(A^{\nu\alpha}(\p_tv-\bp^2\eta)_{\alpha}\right),
\end{equation}
with boundary condition
\begin{equation}\label{qttttbdry}
A^{3\alpha}A^{\mu\alpha}\p_{\mu}\dd^4 q=\left[A^{3\alpha}A^{\mu\alpha},\dd^4\right]\p_{\mu}q+\dd^4\left(A^{3\alpha}(\p_t v-\bp^2\eta)_{\alpha}\right),\text{ on }\Gamma.
\end{equation}}

\red{Similarly we define $\pi_4^{\nu}:=-\left[A^{\nu\alpha}A^{\mu\alpha},\dd^4\right]\p_{\mu}q-\dd^4\left(A^{\nu\alpha}(\p_tv-\bp^2\eta)_{\alpha}\right)$ and $h_4=\text{Right side of }\eqref{qttttbdry}$. It is also straightforward to see that $h_4+\pi_4\cdot N=0$ on $\Gamma$ and $\pi, \dive \pi \in L^2$ because the terms in $\pi$ has at most 4 derivatives.}

\red{Then again by Lemma \ref{H1elliptic}, we have
\begin{align}\label{qtttt11}
\|\dd^4q\|_1\lesssim\|\pi\|_0\lesssim P(E_1(t))+P(N_0)(\|\dd^3\p q\|_0+\|\dd^2\p q\|_0+\|\dd\p q\|_{L^{\infty}}+\|\p q\|_{L^{\infty}}).
\end{align} We find that the terms involving $q$ on the right side of \eqref{qtttt11} contain at most 4 derivatives. One can repeatedly use \eqref{qnormal} to reduce the normal derivatives to tangential ones, and then use $H^1$-type elliptic estimates to make the order lower until only $\|\p q\|_0$ appears.}

\red{Similarly, one can replace $\dd^4 q$ by $\sqrt{\sigma}\TP^{0.5}\dd^4 q$ and follow the same way as above to control $\|\sqrt{\sigma}\p_t^k q\|_{5.5-k}$ for $0\leq k\leq 3$, so we omit these steps and finish the proof of Proposition \ref{lem qell}.}

\section{Estimates for the non-weighted full Sobolev norms}\label{sect divcurl}
We study the estimates for $v$, $\bp\eta$, and their time derivatives in full Sobolev spaces. More precisely, we need to estimate 
$$
\|\p_t^k v\|_{5-k}^2, \quad \|\p_t^k \bp\eta\|_{5-k}^2,\quad \text{for}\,\,k = 0,1,2,3,4.
$$
We will adopt the Hodge-type elliptic div-curl estimate \eqref{divcurl1} to the quantities above. Specifically, we will replace $X$ by $\p_t^k v$ and $\p_t^k \bp\eta$, as well as their $\sigma$-weighted versions. \red{For $0\leq k\leq 4$, using \eqref{divcurl1} in Lemma \ref{hodge}, we have
\begin{align}
\label{divcurlv1} \|\p_t^k v\|_{5-k}^2&\lesssim C\left(\|\sqrt{\sigma} v\|_0^2+\|\sqrt{\sigma}\dive_A \p_t^k v\|_{4-k}^2+\|\curl_A \p_t^k v\|_{4-k}^2+\|\TP^{5-k}\p_t^kv\|_0^2\right),\\
\label{divcurlb1} \|\p_t^k \bp\eta\|_{5.5-k}^2&\lesssim C\left(\|\bp\eta\|_0^2+\|\dive_A \p_t^k\bp\eta\|_{4.5-k}^2+\|\curl_A \p_t^k \bp\eta\|_{4.5-k}^2+\|\TP^{5-k}\p_t^k\bp\eta\|_0^2\right)
\end{align}where the constant $C=C(|\TP\eta|_{W^{1,\infty}},\|\eta\|_{5-k}^2)>0$ depends linearly on $\|\eta\|_{5-k}^2$.}

\subsection{Divergence estimates} \label{sect. div bound}
First, we treat 
$$
\|\di \p_t^k v\|_{4-k}^2,\quad \|\di \p_t^k \bp\eta\|_{4-k}^2, \,\, k=0,\cdots, 4. 
$$

We recall that $\di_A X:= A^{\mu\alpha} \p_\mu X_\alpha$. When $k=0$, since $\di_A v=0$ and $\di_A \bp\eta =0$. \red{When $k\geq1$, we differentiate the div-free conditions in time and get
\begin{equation} \label{pt div v b}
\begin{aligned}
\dive_A\p_t^k v=&~\p_t^k(\dive_A v)-[\p_t^k,A^{\mu\alpha}]\p_{\mu}v_{\alpha}=-[\p_t^k,A^{\mu\alpha}]\p_{\mu}v_{\alpha},\\
\dive_A \p_t^k\bp\eta =&~-[\p_t^k,A^{\mu\alpha}]\p_{\mu}\bp\eta_{\alpha}.
\end{aligned}
\end{equation}}

\red{We notice that the right side of \eqref{pt div v b} only contains $\leq k-1$ time derivatives and the highest order terms contains $4$ derivatives. Thus, in light of the Sobolev product estimate \eqref{product s}, we have
\begin{equation}\label{div est}
\|\dive_A\p_t^k v\|_{4-k}^2 +\|\dive_A \p_t^k\bp\eta\|_{4-k}^2 \lesssim P\left(\sum_{i=0}^{k-1} \|\eta\|_4,\|\p_t^i v\|_{4-i},\|\p_t^i \bp\eta\|_{4-i}\right)\lesssim P(E_1(0))+ \int_0^T P(E_1(t))\dt.
\end{equation}}
%The estimates for $k=2,3,4$ are similar, so we omit the details. In the end, we obtain
%\begin{align}
%\sum_{k=0}^4\Big(\|\di \p_t^kv\|_{4-k}^2 +\|\di \bp\p_t^k\eta\|_{4-k}^2 \Big)\lesssim \eps E_1(t)+P(E_1(0))+ \int_0^T P(E_1(t)).
%\end{align}

\subsection{Curl estimates}\label{sect. curl bound}
In this part we aim to control $\|\curl_A \p_t^{k}v\|_{4-k}$ and $\|\curl_A \p_t^k\bp\eta\|_{4-k}$ for $0\leq k \leq 4$, where we recall $(\curl_A X)_\lambda := \epsilon_{\lambda\tau\alpha}A^{\mu\tau} \p_\mu X^\alpha$. 

Applying $\curl_A$ to the second equation of \eqref{MHDLST}, we get
\begin{equation}\label{curlbeq}
\p_t(\curl_A v)-\bp\left(\curl_A(\bp\eta)\right)=\curl_{\p_t A}v+[\curl_A,\bp]\bp\eta:=\mathcal{F},
\end{equation}
where $(\curl_{\p_t A}v)_\la:=\epsilon_{\lambda\tau\alpha}\p_t A^{\mu\tau} \p_\mu X^\alpha$. Now, when $k=0$, by taking $\p^4$ to \eqref{curlbeq},  testing it with $ \p^4\curl_{A} v$, we have
\begin{align}\label{curl energy pre}
&\frac{d}{dt}\frac{1}{2}\int_{\Omega} |\p^4\curl_{A}v|^2-\int_{\Omega} (\p^4 \curl_A v)\left(\bp \p^4 \left(\curl_A (\bp\eta)\right)\right)\nonumber\\
=&\int_{\Omega} \left([\p^4, \bp] \curl_A \bp\eta + \p^4\mathcal{F} \right)(\p^4 \curl_A v).
\end{align}
Since $\bz^3=0$, we integrate $\bp$ by parts in the second term on the LHS:
\begin{align}
&-\int_{\Omega} (\p^4 \curl_A v)\left(\bp \p^4 \left(\curl_A \bp\eta\right)\right) = \int_{\Omega} \bp (\p^4 \curl_A v) \p^4(\curl_A (\bp\eta))\nonumber\\
= & \int_{\Omega} \p^4 \curl_A (\bp v) \p^4\curl_A (\bp\eta)-\int_{\Omega} ([\p^4\curl_A,  \bp] v) \p^4\curl_A (\bp\eta).
\end{align}
Here, $\int_{\Omega} \p^4 \curl_A (\bp v) \p^4\curl_A (\bp\eta)$ contributes to
\begin{align}
\frac{d}{dt}\frac{1}{2}\int_{\Omega} |\p^4 \curl_A (\bp \eta) |^2- \int_{\Omega} \p^4 \curl_{\p_tA} (\bp \eta) \p^4\curl_A (\bp\eta).
\end{align}
Therefore, \eqref{curl energy pre} becomes
\begin{equation}\label{energy curl}
\begin{aligned}
&\frac{d}{dt}\frac{1}{2}\left(\int_{\Omega} |\p^4\curl_{A}v|^2+\int_{\Omega} |\p^4 \curl_A (\bp \eta) |^2\right)\\
=&\int_{\Omega} \left([\p^4, \bp] \curl_A \bp\eta + \p^4\mathcal{F} \right)(\p^4 \curl_A v)\\
&+\int_{\Omega} ([\p^4\curl_A,  \bp] v) \p^4\curl_A (\bp\eta)+\int_{\Omega} \p^4 \curl_{\p_tA} (\bp \eta) \p^4\curl_A (\bp\eta).
\end{aligned}
\end{equation}
Now, because 
\begin{align*}
\|\mathcal{F}\|_4^2 \leq \|\curl_{\p_t A} v\|_4^2+\|[\curl_A, \bp]\bp\eta\|_4^2 \leq P(\|\bz\|_5, \|\eta\|_5, \|v\|_5, \|\bp\eta\|_5), 
\end{align*}
it is not hard to see that the terms on the RHS of \eqref{energy curl} can be controlled by $P(E(t))$. Therefore, 
\begin{equation}\label{curl curl}
\|\curl_A v\|_4^2+\|\curl_A \bp\eta\|_4^2\bigg|_0^T \lesssim \int_0^T P(E(t)). 
\end{equation} 

The estimates for the cases when $k=1,2,3,4$ follow from a parallel argument: By taking $\p^{4-k}\p_t^k$ (or $\p^{4.5-k}\p_t^k$) to \eqref{curlbeq}, testing it with $\p^{4-k}\p_t^k\curl_A v$, the energy estimate then yields
\begin{equation}
\sum_{k=1}^4 \Big(\|\curl_A \p_t^kv\|_{4-k}^2 +\|\curl_A \bp\p_t^k\eta\|_{4-k}^2 \Big)\lesssim P(E_1(0))+\int_0^T P(E(t)). 
\end{equation} 
Finally, given \eqref{curl curl}, we obtain
\begin{align}\label{curl est}
\red{\sum_{k=0}^4 \Big(\|\curl_A \p_t^kv\|_{4-k}^2 +\|\curl_A \bp\p_t^k\eta\|_{4-k}^2\Big)\lesssim P(E_1(0))+ \int_0^T P(E_1(t)).}
\end{align}

%\subsection{Control of the boundary term}\label{sect. normaltrace1}
%Finally, we need to treat 
%$$
%|\TP\p_t^k v^3|_{3.5-k}^2,\quad |\TP\p_t^k \bp\eta^3|_{3.5-k}^2
%$$
%which correspond to the last term in \eqref{divcurls}. For these non-weighted terms, it suffices to control 
% $$
% |\TP^{5-k} \p_t^k v^3|_{-0.5}^2, \quad |\TP^{5-k}\p_t^k \bp\eta^3|_{-0.5}^2. 
% $$
%We invoke Lemma \eqref{normaltrace} to obtain
%\begin{align}
% |\TP^{5-k} \p_t^k v^3|_{-0.5}^2&\lesssim \|\TP^{5-k} \p_t^k v\|_{0}^2+\|\di \p_t^k v\|_{4-k}^2,\label{bdry v}\\
%|\TP^{5-k} \p_t^k \bp\eta^3|_{-0.5}^2 &\lesssim \|\TP^{5-k} \p_t^k \bp\eta\|_{0}^2+ \|\di \p_t^k\bp\eta\|_{4-k}^2.\label{bdry b}
%\end{align}
%Here, the second term on the RHS of \eqref{bdry v} and \eqref{bdry b} has been treated in Section \ref{sect. div bound}, while the first term corresponds to the tangential energy that will be studied in the upcoming sections.

\section{Tangential estimate of spatial derivatives}\label{sect tgs5}

\red{By \eqref{divcurlv1}-\eqref{divcurlb1}, we still need to do the interior tangential estimate.} In this section, we would like to do the $\TP^5$-estimates for $v$ and $b=\bp\eta$, i.e., the control of tangential spatial derivatives.

\subsection{Interior estimates: Alinhac good unknown}\label{sect AGU}

However, we cannot directly commute $\TP^5$ with the covariant derivative $\pa$ because the commutator contains $\TP^5 A=\TP^5\p\eta\times\p\eta$ whose $L^2$ norm cannot be controlled. The reason is that the essential highest order term in $\TP^5(\pa f)$, i.e., the standard derivatives of a covariant derivative, is the covariant derivative of Alinhac good unknown $\mathbf{f}:=\TP^5f-\TP^5\eta\cdot\pa f$ instead of the term produced by simply commuting $\TP^5$ with $\pa$. Specifically,
\begin{align*}
\TP^5(\pa^{\alpha}f)&=\pa^{\alpha}(\TP^5 f)+(\TP^5A^{\mu\alpha})\p_\mu f+[\TP^5,A^{\mu\alpha},\p_{\mu} f] \\
&=\pa^{\alpha}(\TP^5 f)-\TP^4(A^{\mu\gamma}\TP\p_{\beta}\eta_{\gamma}A^{\beta\alpha})\p_\mu f+[\TP^5,A^{\mu\alpha},\p_{\mu} f] \\
&=\pa^{\alpha}(\TP^5 f)-A^{\beta\alpha}\p_{\beta}\TP^5\eta_{\gamma}A^{\mu\gamma}\p_\mu f-([\TP^4,A^{\mu\gamma}A^{\beta\alpha}]\TP\p_{\beta}\eta_{\gamma})\p_\mu f+[\TP^5,A^{\mu\alpha},\p_{\mu} f] \\
&=\underbrace{\pa^{\alpha}(\TP^5 f-\TP^5 \eta_{\gamma}A^{\mu\gamma}\p_\mu f)}_{=\pa^{\alpha}\mathbf{f}}+\underbrace{\TP^5 \eta_{\gamma}\pa^{\alpha}(\pa^{\gamma} g)-([\TP^4,A^{\mu\gamma}A^{\beta\alpha}]\TP\p_{\beta}\eta_{\gamma})\p_\mu f+[\TP^5,A^{\mu\alpha},\p_{\mu} f] }_{=:C^{\alpha}(f)},
\end{align*} 

We introduce the Alinhac good unknowns of $v$ and $q$ with respect to $\TP^5$ by
\begin{equation}\label{good5}
\VV:=\TP^5 v-\TP^5\eta\cdot\pa v,~~\QQ:=\TP^5 q-\TP^5\eta\cdot\pa q.
\end{equation} Then direct computation (e.g., see \cite[Section 4.2.4]{gu2016construction}) shows that the good unknowns enjoy the following properties:
\begin{equation}\label{goodid}
\TP^5 \underbrace{(\pa\cdot v)}_{=0}=\pa\cdot\VV+C(v),~~\TP^5(\pa q)=\pa\QQ+C(q)
\end{equation} and
\begin{equation}\label{goodco}
\|C(f)\|_0\lesssim P(\|\eta\|_5)\|f\|_5.
\end{equation}

Under this setting, we take $\TP^5$ in the second equation of \eqref{MHDLST} and invoke \eqref{good5} to get the evolution equation of the Alinhac good unknowns
\begin{equation}\label{goodeq}
\p_t\VV=-\pa\QQ+\bp(\TP^5\bp\eta)+\underbrace{\p_t(\TP^5\eta\cdot\pa v)-C(q)+[\TP^5,\bp](\bp\eta)}_{=:\ff_0}.
\end{equation}

Taking $L^2(\Omega)$ inner product with $\VV$, we get
\begin{equation}\label{goodenergy0}
\frac12\frac{d}{dt}\io|\VV|^2\dy=-\io\pa\QQ\cdot\VV\dy+\io\left(\bp(\TP^5\bp\eta)\right)\cdot\VV\dy+\io\ff_0\cdot\VV\dy,
\end{equation}where the last term can be directly controlled:
\begin{equation}\label{I3}
\io\ff_0\cdot\VV\dy\lesssim P(\|\eta\|_5,\|v\|_5,\|\p_t v\|_4,\|q\|_5,\|\bz\|_5,\|\bp\eta\|_5)\lesssim P(E_1(t)).
\end{equation}

Then we integrate $\bp$ by parts in the second integral of \eqref{goodenergy0} to produce the tangential energy of the magnetic field $\bp\eta$. Note that since $\bz\cdot N=0$ on $\p{\Omega}$ and $\dive \bz=0$, no boundary term appears in this step. 
\begin{equation}\label{tgb5}
\begin{aligned}
&\io\left(\bp(\TP^5\bp\eta)\right)\cdot\VV\dy=-\io(\TP^5\bp\eta)\cdot\bp\VV\dy\\
=&-\io(\TP^5\bp\eta)\cdot\bp(\TP^5\p_t\eta)\dy+\io(\TP^5\bp\eta)\cdot\bp(\TP^5\eta\cdot\pa v)\dy\\
=&-\frac12\frac{d}{dt}\io\left|\TP^5(\bp\eta)\right|^2+\io(\TP^5\bp\eta^{\alpha})([\TP^5,\bp]v_{\alpha}+\bp(\TP^5\eta\cdot\pa v_{\alpha}))\dy\\
\lesssim&-\frac12\frac{d}{dt}\io\left|\TP^5(\bp\eta)\right|^2+P(E_1(t)).
\end{aligned}
\end{equation}

Next, we analyze the first integral of \eqref{goodenergy0}. Integrate by parts, using Piola's identity $\p_{\mu}A^{\mu\alpha}=0$ and invoking \eqref{goodid}, we get
\begin{equation}\label{I1}
\begin{aligned}
-\io\pa\QQ\cdot\VV\dy=&\io\QQ(\pa\cdot\VV)\dy\underbrace{-\ig A^{3\alpha}\QQ\VV_{\alpha}\dS}_{=:J}-\int_{\Gamma_0}\underbrace{A^{3\alpha}\QQ\VV_{\alpha}}_{=0}\dS\\
=&-\io\QQ~C(v)\dy+J\lesssim\|\QQ\|_{0}\|C(v)\|_0+J\\
\lesssim&~P(\|\eta\|_5,\|q\|_5,|v\|_5)+J,
\end{aligned}
\end{equation}where the boundary integral on $\Gamma_0$ vanishes due to $\eta|_{\Gamma_0}=\text{ Id}$ and thus $A^{3\alpha}\VV_{\alpha}=\TP^5 v_3=0$. Therefore, it remains to analyze the boundary integral $J$.

\subsection{Boundary estimates and cancellation structure}\label{sect RTST}

The boundary integral now reads
\begin{equation}\label{J0}
\begin{aligned}
J=&-\ig A^{3\alpha} \QQ\VV_{\alpha}\dS\\
=&-\ig A^{3\alpha}\TP^5 q\VV_{\alpha}\dS\underbrace{+\ig A^{3\alpha}(\TP^5\eta\cdot\pa q)\VV_{\alpha}\dS}_{=:\text{ RT}}\\
=&\underbrace{-\ig \TP^5(A^{3\alpha} q)\VV_{\alpha}\dS}_{=:\text{ ST}}+\ig q (\TP^5 A^{3\alpha})\VV_{\alpha}\dS +\ig \sum_{k=1}^4\binom{5}{k}\TP^{5-k} A^{3\alpha} \TP^k q \VV_{\alpha}\dS +\RT\\
=:&\ST+J_1+J_2+\RT.
\end{aligned}
\end{equation}

\subsubsection{Non-weighted boundary energy: Rayleigh-Taylor sign condition}\label{sect RT}

The term RT and the Rayleigh-Taylor sign condition yield the non-weighted boundary energy. Recall that $\VV=\TP^5 v-\TP^5\eta\cdot\pa v$, then we have
\begin{equation}\label{RT0}
\begin{aligned}
\RT=&~\ig A^{3\alpha}\TP^5\eta_{\beta}A^{3\beta}\p_3 q\TP^5 v_{\alpha}\dS\\
&-\ig A^{3\alpha}\TP^5\eta_{\beta}A^{3\beta}\p_3 q\TP^5 \eta_{\gamma}A^{\mu\gamma}\p_\mu v_{\alpha}\dS\\
&+\sum_{i=1}^2\ig A^{3\alpha}\TP^5\eta_{\beta}A^{i\beta}\TP_i q(\TP^5v_{\alpha}-\TP^5\eta\cdot \pa v_{\alpha})\dS\\
=:&\RT_1+\RT_2+\RT_3.
\end{aligned}
\end{equation}

The term $\RT_1$ gives the boundary energy term by writting $v_{\alpha}=\p_t\eta_{\alpha}$.
\begin{equation}\label{RT1}
\begin{aligned}
\RT_1=&~\ig A^{3\alpha}\TP^5\eta_{\beta}A^{3\beta}\p_3 q\p_t\TP^5\eta_{\alpha}\dS\\
=&~-\frac12\frac{d}{dt}\ig(-\p_3 q)\left|A^{3\alpha}\TP^5\eta_{\alpha}\right|^2\dS+\ig (\p_t A^{3\alpha})A^{3\beta}\TP^5\eta_{\beta}\p_3 q\TP^5\eta_{\alpha}\dS+\frac12\ig\p_t\p_3q\left|A^{3\alpha}\TP^5\eta_{\alpha}\right|^2\dS\\
=:&-\frac12\frac{d}{dt}\ig(-\p_3 q)\left|A^{3\alpha}\TP^5\eta_{\alpha}\right|^2\dS+\RT_{11}+\RT_{12}.
\end{aligned}
\end{equation}

The term$\RT_{12}$ can be directly controlled by
\begin{equation}\label{RT12}
\RT_{12}\lesssim |\p_t\p_3 q|_{L^{\infty}}\left|A^{3\alpha}\TP^5\eta_{\alpha}\right|_0^2\lesssim P(E_1(t)).
\end{equation}

The term$\RT_{11}$ is exactly \red{canceled} by$\RT_2$ after plugging $\p_t A^{3\alpha}=-A^{3\gamma}\p_{\mu}v_{\gamma}A^{\mu\alpha}$:
\begin{equation}\label{RT11}
\RT_{11}=-\ig A^{3\gamma}\p_{\mu}v_{\gamma}A^{\mu\alpha}A^{3\beta}\TP^5\eta_{\beta}\p_3 q\TP^5\eta_{\alpha}\dS=-\RT_{2}.
\end{equation}

Finally, invoking $q=-\sigma\sqrt{g}\Delta_g\eta\cdot\nn=\sigma Q(\TP\eta)\TP^2\eta\cdot\nn$, we can control$\RT_3$ by the weighted energy and trace lemma
\begin{equation}\label{RT3}
\begin{aligned}
\RT_3=&~\sigma\ig A^{3\alpha}\TP^5\eta_{\beta}A^{i\beta} \TP_i(Q(\TP\eta)\TP^2\eta\cdot\nn)(\TP^5v_{\alpha}-\TP^5\eta\cdot\pa v_{\alpha})\dS\\
\lesssim&~P(|\p\eta|_{L^{\infty}})|\TP^3\eta|_{L^{\infty}}|\sqrt{\sigma}\TP^5\eta|_0(|\sqrt{\sigma}\TP^5v|_0+|\sqrt{\sigma}\TP^5\eta|_0|\pa v|_{L^{\infty}})\\
\lesssim&~P(\|\eta\|_{3})\|\eta\|_5\underbrace{\|\sqrt{\sigma}\eta\|_{5.5}(\|\sqrt{\sigma}v\|_{5.5}+\|\sqrt{\sigma}\eta\|_{5.5}\| v\|_{3})}_{\leq \sqrt{\sigma E_2}(\sqrt{\sigma E_2}+\sqrt{E_1}\sqrt{\sigma E_2})}\\
\lesssim&~P(E_1(t))(\sigma E_2(t)).
\end{aligned}
\end{equation}

Summarizing \eqref{RT0}-\eqref{RT3} and using the a priori assumption \ref{RTapriori} and $A^{3\alpha}=\sqrt{g}\nn^{\alpha}$, we conclude the estimate of RT by
\begin{equation}\label{RT}
\int_0^T\RT\dt\lesssim -\frac{c_0}{4}\left|\TP^5\eta_{\alpha}\cdot\nn\right|_0^2+\int_0^TP(E_1(t))(\sigma E_2(t))\dt.
\end{equation}

\subsubsection{Control of the weighted boundary energy: surface tension}\label{sect ST}

Now we analyze the term ST, where the surface tension gives the $\wt$-weighted top-order boundary energy. Invoking $A^{3\alpha}q=-\sigma\sqrt{g}g^{ij}\nn^{\alpha}\nn^{\beta}\TP_i\TP_j\eta^{\beta}$, we get
\begin{equation}\label{ST0}
\begin{aligned}
\ST=&~\sigma\ig\sqrt{g}g^{ij}\nn^{\alpha}\nn^{\beta}\TP^5\TP_i\TP_j\eta_{\beta}(\TP^5 v_{\alpha}-\TP^5\eta\cdot\pa v_{\alpha})\dS\\
&+5\sigma\ig\TP(\sqrt{g}g^{ij}\nn^{\alpha}\nn^{\beta})\TP^4\TP_i\TP_j\eta_{\beta}(\TP^5 v_{\alpha}-\TP^5\eta\cdot\pa v_{\alpha})\dS\\
&+\sigma\ig\TP^5(\sqrt{g}g^{ij}\nn^{\alpha}\nn^{\beta})\TP_i\TP_j\eta_{\beta}(\TP^5 v_{\alpha}-\TP^5\eta\cdot\pa v_{\alpha})\dS\\
&+\sum_{k=2}^4\sigma\ig\binom{5}{k}\TP^k(\sqrt{g}g^{ij}\nn^{\alpha}\nn^{\beta})\TP^{5-k}\TP_i\TP_j\eta_{\beta}(\TP^5 v_{\alpha}-\TP^5\eta\cdot\pa v_{\alpha})\dS\\
=:&\ST_1+\ST_2+\ST_3+\ST_4.
\end{aligned}
\end{equation}

The term$\ST_4$ can be directly controlled with the help of $\wt$-weighted energy:
\begin{equation}\label{ST4}
\begin{aligned}
\ST_4\lesssim&~10|\TP^2(\sqrt{g}g^{ij}\nn^{\alpha}\nn^{\beta})|_{L^{\infty}}|\sqrt{\sigma}\TP^3\TP^2_{ij}\eta_{\beta}|_0(|\sqrt{\sigma}\TP^5v|_0+|\sqrt{\sigma}\TP^5\eta|_0|\pa v|_{L^{\infty}})\\
&+10|\TP^3(\sqrt{g}g^{ij}\nn^{\alpha}\nn^{\beta})|_{L^4}|\sqrt{\sigma}\TP^2\TP^2_{ij}\eta_{\beta}|_{L^4}(|\sqrt{\sigma}\TP^5v|_0+|\sqrt{\sigma}\TP^5\eta|_0|\pa v|_{L^{\infty}})\\
&+5|\sqrt{\sigma}\TP^4(\sqrt{g}g^{ij}\nn^{\alpha}\nn^{\beta})|_{0}|\TP\TP^2_{ij}\eta_{\beta}|_{L^{\infty}}(|\sqrt{\sigma}\TP^5v|_0+|\sqrt{\sigma}\TP^5\eta|_0|\pa v|_{L^{\infty}})\\
\lesssim&~P(E_1(t))(\sigma E_2(t)).
\end{aligned}
\end{equation}

In$\ST_1$, we first integrate $\TP_i$ by parts:
\begin{equation}\label{ST10}
\begin{aligned}
\ST_1\overset{\TP_i}{=}&~-\sigma\ig\sqrt{g}g^{ij}\nn^{\alpha}\nn^{\beta}\TP^5\TP_j\eta_{\beta}\TP_i(\TP^5 v_{\alpha}-\TP^5\eta\cdot\pa v_{\alpha})\dS\\
&-\sigma\ig\TP_i(\sqrt{g}g^{ij}\nn^{\alpha}\nn^{\beta})\TP^5\TP_j\eta_{\beta}(\TP^5 v_{\alpha}-\TP^5\eta\cdot\pa v_{\alpha})\dS\\
=:&\ST_{11}+\ST_{12}.
\end{aligned}
\end{equation}

In$\ST_{11}$, we write $v_{\alpha}=\p_t\eta_{\alpha}$ to be 
\begin{equation}\label{ST110}
\begin{aligned}
\ST_{11}=&~-\sigma\ig\sqrt{g}g^{ij}\nn^{\alpha}\nn^{\beta}\TP^5\TP_j\eta_{\beta}\TP_i(\TP^5 \p_t\eta_{\alpha}-\TP^5\eta\cdot\pa v_{\alpha})\dS\\
=&~\red{-\frac{\sigma}{2}\frac{d}{dt}\ig\sqrt{g}g^{ij}\left(\TP^5\TP_i\eta\cdot\nn\right)\left(\TP^5\TP_j\eta\cdot\nn\right)\dS}-\frac{\sigma}{2}\ig\sqrt{g}\p_t g^{ij}(\TP^5\TP_i\eta\cdot\nn)(\TP^5\TP_j\eta\cdot\nn)\dS\\
&+\sigma\ig g^{ij}\p_t(\underbrace{\sqrt{g}\nn^{\alpha}}_{=A^{3\alpha}})\TP^5\TP_i\eta_{\alpha}(\TP^5\TP_j\eta\cdot\nn)\dS+\sigma\ig\sqrt{g}g^{ij}\nn^{\alpha}\nn^{\beta}\TP^5\TP_j\eta_{\beta}\TP_i\TP^5\eta_{\gamma}A^{\mu\gamma}\p_\mu v_{\alpha}\dS\\
&+\sigma\ig\sqrt{g}g^{ij}\nn^{\alpha}(\nn^{\beta}\TP^5\TP_j\eta_{\beta})(\TP^5\eta\cdot\TP_i(\pa v_{\alpha}))\red{\dS}\\
=:&\ST_{111}+\ST_{112}+\ST_{113}+\ST_{114}+\ST_{115}.
\end{aligned}
\end{equation}

\red{The term$\ST_{111}$ produce the weighted energy term because $\sqrt{g}g^{ij}$ is positive-definite
\begin{equation}\label{ST111}
\int_0^T\ST_{111}\lesssim-\frac12\left|\sqrt{\sigma}\TP^6\eta\cdot\nn\right|_0^2\bigg|^T_0.
\end{equation}}

The terms$\ST_{112}$ and$\ST_{115}$ can be directly controlled
\begin{align}
\label{ST112}\ST_{112}\lesssim &~|\sqrt{g}\p_t g^{ij}|_{L^{\infty}}\left|\sqrt{\sigma}\TP^6\eta\cdot\nn\right|_0^2\lesssim P(E_1(t))(\sigma E_2(t)).\\
\label{ST115}\ST_{115}\lesssim &~|\sqrt{g}g^{ij}\TP(\pa v)|_{L^{\infty}}\left|\sqrt{\sigma}\TP^6\eta\cdot\nn\right|_0\|\sqrt{\sigma}\eta\|_{5.5}\lesssim P(E_1(t))(\sigma E_2(t)).
\end{align}

In$\ST_{113}$, we use \eqref{n}-\eqref{an}, i.e., $\sqrt{g}\nn^{\alpha}=A^{3\alpha}$ and $\p_tA^{3\alpha}=-A^{3\gamma}\p_{\mu}v_{\gamma}A^{\mu\alpha}$ to produce cancellation with$\ST_{114}$
\begin{equation}\label{ST1134}
\ST_{113}=-\sigma\ig g^{ij}A^{3\gamma}\p_{\mu}v_{\gamma}A^{\mu\alpha}\TP^5\TP_i\eta_{\alpha}(\TP^5\TP_j\eta\cdot\nn)\dS=-\ST_{114}.
\end{equation}

Summarizing \eqref{ST110}-\eqref{ST1134}, we conclude the estimate of$\ST_{11}$ by choosing $\eps>0$ sufficiently small
\begin{equation}\label{ST11}
\int_0^T\ST_{11}\dt\lesssim-\frac{\sigma}{2}\left|\TP^5\TP\eta\cdot\nn\right|_0^2\bigg|^T_0+\int_0^TP(E_1(t))(\sigma E_2(t))\dt.
\end{equation}

Next we control$\ST_{12}$. First, we have
\begin{equation}\label{ST120}
\begin{aligned}
\ST_{12}=&~-\sigma\ig\TP_i(\sqrt{g}g^{ij}\nn^{\alpha})(\TP^5\TP_j\eta\cdot\nn)(\TP^5v_{\alpha}-\TP^5\eta\cdot\pa v_{\alpha})\dS\\
&-\sigma\ig\sqrt{g}g^{ij}\nn^{\alpha}(\TP_i \nn^{\beta})\TP^5\TP_j\eta_{\beta}(\TP^5v_{\alpha}-\TP^5\eta\cdot\pa v_{\alpha})\dS\\
=:&~\ST_{121}+\ST_{122}.
\end{aligned}
\end{equation}

The term$\ST_{121}$ can be directly controlled
\begin{equation}\label{ST121}
\ST_{121}\lesssim|\TP_i(\sqrt{g}g^{ij}\nn)|_{L^{\infty}}|\sqrt{\sigma}\TP^6\eta\cdot\nn|_0(|\sqrt{\sigma}\TP^5v|_0+|\sqrt{\sigma}\TP^5\eta|_0|\pa v|_{L^{\infty}})\lesssim P(E_1(t))(\sigma E_2(t)).
\end{equation}

To control$\ST_{122}$, we first integrate $\TP_j$ by parts.
\begin{equation}\label{ST1220}
\begin{aligned}
\ST_{122}=&~-\sigma\ig\sqrt{g}g^{ij}\nn^{\alpha}(\TP_i \nn^{\beta})\TP^5\eta_{\beta}\TP_j\TP^5v_{\alpha}\dS\\
&+\sigma\ig\sqrt{g}g^{ij}\nn^{\alpha}(\TP_i\nn^{\beta})\TP^5\eta_{\beta}\TP_j\TP^5\eta_{\gamma}(A^{\mu\gamma}\p_{\mu}v_{\alpha})\dS\\
&+\sigma\ig\sqrt{g}g^{ij}\nn^{\alpha}(\TP_i\nn^{\beta})\TP^5\eta_{\beta}\TP^5\eta_{\gamma}\TP_j(A^{\mu\gamma}\p_{\mu}v_{\alpha})\dS\\
&+\sigma\ig\TP_j(\sqrt{g}g^{ij}\nn^{\alpha}\TP_i \nn^{\beta})\TP^5\eta_{\beta}(\TP^5v_{\alpha}-\TP^5\eta\cdot\pa v_{\alpha})\dS\\
=:&\ST_{1221}+\ST_{1222}+\ST_{1223}+\ST_{1224}.
\end{aligned}
\end{equation}

The term$\ST_{1223}$ can be directly controlled by the weighted energy 
\begin{equation}\label{ST1223}
\ST_{1223}+\ST_{1224}\lesssim P(E_1(t))(\sigma E_2(t)).
\end{equation}

To control$\ST_{1221}$, we write $v_{\alpha}=\p_t\eta_{\alpha}$ and then integrate $\p_t$ by parts under the time integral. When $\p_t$ falls on $\sqrt{g}\nn^{\alpha}=A^{3\alpha}$, the cancellation structure analogous to \eqref{ST1134} is again produced.
\begin{equation}\label{ST1221}
\begin{aligned}
\int_0^T\ST_{1221}\dt\overset{\p_t}{=}&~\sigma\int_0^T\ig\p_t(g^{ij}\TP_i\nn^{\beta}\TP^5\eta_{\beta})(\TP_j\TP^5\eta_{\alpha}\nn^{\alpha})\dS\\
&+\sigma\int_0^T\ig\underbrace{(-A^{3\gamma}\p_\mu v_{\gamma}A^{\mu\alpha})}_{\p_t(\sqrt{g}\nn^{\alpha})}g^{ij}\TP_i\nn^{\beta}\TP^5\eta_{\beta}\TP_j\TP^5\eta_{\alpha}\dS\\
\lesssim&~\int_0^TP(E_1(t))(\|\sqrt{\sigma}v\|_{5.5}+\|\sqrt{\sigma}\eta\|_{5.5})|\sqrt{\sigma}\TP^6\eta\cdot\nn|_0\dt+(-\ST_{1222}).
\end{aligned}
\end{equation}

Therefore,$\ST_{122}$ is controlled by
\begin{equation}\label{ST122}
\int_0^T\ST_{122}\dt\lesssim\int_0^TP(E_1(t))(\sigma E_2(t))\dt,
\end{equation}which together with \eqref{ST11} and \eqref{ST121} gives the control of$\ST_1$
\begin{equation}\label{ST1}
\int_0^T\ST_1\dt\lesssim~-\frac{\sigma}{2}\ig\left|\TP^6\eta\cdot\nn\right|^2\dS\bigg|^T_0+\int_0^TP(E_1(t))(\sigma E_2(t))\dt.
\end{equation}

It remains to control$\ST_2$ and$\ST_3$ in \eqref{ST0}. From \eqref{ST10}, we find that$\ST_2$ has the same form as$\ST_{12}$, so we omit the analysis of$\ST_{2}$ and only list the result
\begin{equation}\label{ST2}
\int_0^T\ST_{2}\dt\lesssim\int_0^TP(E_1(t))(\sigma E_2(t))\dt.
\end{equation}As for$\ST_3$, we have
\begin{equation}\label{ST30}
\begin{aligned}
\ST_3=&~\sigma\ig\sqrt{g}g^{ij}\nn^{\alpha}(\TP^5\nn^{\beta})\TP_i\TP_j\eta_{\beta}(\TP^5v_{\alpha}-\TP^5\eta\cdot\pa v_{\alpha})\dS\\
&+\sigma\ig\sqrt{g}g^{ij}(\TP^5\nn^{\alpha})\nn^{\beta}\TP_i\TP_j\eta_{\beta}(\TP^5v_{\alpha}-\TP^5\eta\cdot\pa v_{\alpha})\dS\\
&+\sigma\ig\TP^5(\sqrt{g}g^{ij})\nn^{\alpha}\nn^{\beta}\TP_i\TP_j\eta_{\beta}(\TP^5v_{\alpha}-\TP^5\eta\cdot\pa v_{\alpha})\dS\\
&+\sum_{k=1}^4\sigma\ig\TP^k(\sqrt{g}g^{ij})\TP^{5-k}(\nn^{\alpha}\nn^{\beta})\TP_i\TP_j\eta_{\beta}(\TP^5v_{\alpha}-\TP^5\eta\cdot\pa v_{\alpha})\dS\\
=:&\ST_{31}+\ST_{32}+\ST_{33}+\ST_{34},
\end{aligned}
\end{equation}where$\ST_{34}$ can be directly controlled
\begin{equation}\label{ST34}
\ST_{34}\lesssim P(|\eta|_{W^{3,\infty}})|\sqrt{\sigma}\eta|_5(|\sqrt{\sigma}\TP^5v|_0+|\sqrt{\sigma}\TP^5\eta|_0|\pa v|_{L^{\infty}})\lesssim P(E_1(t))(\sigma E_2(t)).
\end{equation}

To control$\ST_{31}$ and$\ST_{32}$, we need to invoke \eqref{tpn} to get
\[
\TP^5\nn^{\alpha}=-\TP^4\left(g^{kl}(\TP\TP_k\eta\cdot\nn)\TP_l\eta^{\alpha}\right)=-g^{kl}(\TP^5\TP_k\eta\cdot\nn)\TP_l\eta^{\alpha}-[\TP^4,g^{kl}\TP_l\eta^{\alpha}](\TP\TP_k\eta\cdot\nn),
\]and thus plug it into$\ST_{31}$ and$\ST_{32}$:
\begin{equation}\label{ST31}
\begin{aligned}
\ST_{31}=&-\sigma\ig\sqrt{g}g^{ij}g^{kl}(\TP^5\TP_k\eta\cdot\nn)\TP_l\eta^{\beta}\TP_i\TP_j\eta_{\beta}\nn_{\alpha}(\TP^5v_{\alpha}-\TP^5\eta\cdot\pa v_{\alpha})\dS\\
&-\sigma\ig\sqrt{g}g^{ij}\left([\TP^4,g^{kl}\TP_l\eta^{\alpha}](\TP\TP_k\eta\cdot\nn)\right)\TP_i\TP_j\eta_{\beta}\nn_{\alpha}(\TP^5v_{\alpha}-\TP^5\eta\cdot\pa v_{\alpha})\dS\\
\lesssim&~P(E_1(t))|\sqrt{\sigma}\TP^6\eta\cdot\nn|_0(|\sqrt{\sigma}\TP^5v|_0+|\sqrt{\sigma}\TP^5\eta|_0|\pa v|_{L^{\infty}})\\
&+P(E_1(t))|\sqrt{\sigma}\TP^5\eta|_0(|\sqrt{\sigma}\TP^5v|_0+|\sqrt{\sigma}\TP^5\eta|_0|\pa v|_{L^{\infty}})\\
\lesssim&~P(E_1(t))(\sigma E_2(t)),
\end{aligned}
\end{equation}and similarly
\begin{equation}\label{ST32}
\begin{aligned}
\ST_{32}\lesssim&~P(E_1(t))|\sqrt{\sigma}\TP^6\eta\cdot\nn|_0(|\sqrt{\sigma}\TP^5v|_0+|\sqrt{\sigma}\TP^5\eta|_0|\pa v|_{L^{\infty}})\\
&+P(E_1(t))|\sqrt{\sigma}\TP^5\eta|_0(|\sqrt{\sigma}\TP^5v|_0+|\sqrt{\sigma}\TP^5\eta|_0|\pa v|_{L^{\infty}})\\
\lesssim&~P(E_1(t))(\sigma E_2(t)),
\end{aligned}
\end{equation}

For$\ST_{33}$, we use the identity \eqref{tp ggij} to get
\begin{align*}
\TP^5(\sqrt{g}g^{ij})=&~\sqrt{g}\left(\frac12 g^{ij}g^{kl}-g^{ik}g^{jl}\right)(\TP^5\TP_k\eta^{\mu}\TP_l\eta_{\mu}+\TP_k\eta^{\mu}\TP^5\TP_l\eta_{\mu})\\
&+\underbrace{\left[\TP^4,\TP_l\eta^{\mu}\left(\frac12 g^{ij}g^{kl}-g^{ik}g^{jl}\right)\right]\TP\TP_k\eta_{\mu}+\left[\TP^4,\TP_k\eta^{\mu}\left(\frac12 g^{ij}g^{kl}-g^{ik}g^{jl}\right)\right]\TP\TP_l\eta_{\mu}}_{R^{ij}_{33}},
\end{align*}and thus
\begin{equation}\label{ST330}
\begin{aligned}
\ST_{33}=&~\sigma\ig\sqrt{g}\left(\frac12 g^{ij}g^{kl}-g^{ik}g^{jl}\right)(\TP^5\TP_k\eta^{\mu}\TP_l\eta_{\mu}+\TP_k\eta^{\mu}\TP^5\TP_l\eta_{\mu})\nn^{\alpha}\nn^{\beta}\TP_i\TP_j\eta_{\beta}(\TP^5v_{\alpha}-\TP^5\eta\cdot\pa v_{\alpha})\dS\\
&+\sigma\ig R_{33}^{ij}\nn^{\alpha}\nn^{\beta}\TP_i\TP_j\eta_{\beta}(\TP^5v_{\alpha}-\TP^5\eta\cdot\pa v_{\alpha})\dS\\
=:&\ST_{331}+\ST_{332}.
\end{aligned}
\end{equation}

The term$\ST_{332}$ can be directly controlled
\begin{equation}\label{ST332}
\ST_{332}\lesssim P(E_1(t))|\sqrt{\sigma}\eta|_5(|\sqrt{\sigma}\TP^5v|_0+|\sqrt{\sigma}\TP^5\eta|_0|\pa v|_{L^{\infty}}).
\end{equation}

In$\ST_{331}$, we should first integrate the derivative $\TP_k$ in $\TP^5\TP_k\eta^{\mu}$ (resp. $\TP_l$ in $\TP^5\TP_l\eta_{\mu}$) by parts
\begin{equation}\label{ST3310}
\begin{aligned}
\ST_{331}=&-\sigma\ig\sqrt{g}\left(\frac12 g^{ij}g^{kl}-g^{ik}g^{jl}\right)\TP^5\eta^{\mu}\TP_l\eta_{\mu}\nn^{\alpha}\nn^{\beta}\TP_i\TP_j\eta_{\beta}\TP_k(\TP^5v_{\alpha}-\TP^5\eta\cdot\pa v_{\alpha})\dS\\
&-\sigma\ig\sqrt{g}\left(\frac12 g^{ij}g^{kl}-g^{ik}g^{jl}\right)\TP_k\eta^{\mu}\TP^5\eta_{\mu}\nn^{\alpha}\nn^{\beta}\TP_i\TP_j\eta_{\beta}\TP_l(\TP^5v_{\alpha}-\TP^5\eta\cdot\pa v_{\alpha})\dS\\
&-\sigma\ig\TP_k\left(\sqrt{g}\left(\frac12 g^{ij}g^{kl}-g^{ik}g^{jl}\right)\TP_l\eta_{\mu}\nn^{\alpha}\nn^{\beta}\TP_i\TP_j\eta_{\beta}\right)\TP^5\eta^{\mu}(\TP^5v_{\alpha}-\TP^5\eta\cdot\pa v_{\alpha})\dS\\
&-\sigma\ig\TP_l\left(\sqrt{g}\left(\frac12 g^{ij}g^{kl}-g^{ik}g^{jl}\right)\TP_k\eta_{\mu}\nn^{\alpha}\nn^{\beta}\TP_i\TP_j\eta_{\beta}\right)\TP^5\eta^{\mu}(\TP^5v_{\alpha}-\TP^5\eta\cdot\pa v_{\alpha})\dS\\
=:&\ST_{3311}+\ST_{3312}+\ST_{3313}+\ST_{3314},
\end{aligned}
\end{equation}where$\ST_{3313}$ and$\ST_{3314}$ can be directly controlled
\begin{equation}\label{ST33134}
\ST_{3313}+\ST_{3314}\lesssim P(E_1(t))(\sigma E_2(t)).
\end{equation}

For$\ST_{3311}$ and$\ST_{3312}$, we need to write $v_{\alpha}=\p_t\eta_{\alpha}$ and then integrate $\p_t$ by parts. For simplicity, we only show the control of$\ST_{3311}$.
\begin{align}
\label{ST33111}\int_0^T\ST_{3311}\dt=&-\sigma\int_0^T\ig\sqrt{g}\left(\frac12 g^{ij}g^{kl}-g^{ik}g^{jl}\right)\TP^5\eta^{\mu}\TP_l\eta_{\mu}\nn^{\alpha}\nn^{\beta}\TP_i\TP_j\eta_{\beta}\TP_k\TP^5\p_t\eta_{\alpha}\dS\\
\label{ST33112}&+\sigma\int_0^T\ig\sqrt{g}\left(\frac12 g^{ij}g^{kl}-g^{ik}g^{jl}\right)\TP^5\eta^{\mu}\TP_l\eta_{\mu}\nn^{\alpha}\nn^{\beta}\TP_i\TP_j\eta_{\beta}\TP_k\TP^5\eta_{\gamma}A^{\mu\gamma}\p_{\mu}v_{\alpha}\dS\\
\label{ST33113}&+\sigma\int_0^T\ig\sqrt{g}\left(\frac12 g^{ij}g^{kl}-g^{ik}g^{jl}\right)\TP^5\eta^{\mu}\TP_l\eta_{\mu}\nn^{\alpha}\nn^{\beta}\TP_i\TP_j\eta_{\beta}\TP^5\eta\cdot\TP_k(\pa v_{\alpha})\dS\\
\label{ST331111}\overset{\p_t}{=}&~\sigma\int_0^T\ig\sqrt{g}\left(\frac12 g^{ij}g^{kl}-g^{ik}g^{jl}\right)\TP^5v^{\mu}\TP_l\eta_{\mu}\nn^{\beta}\TP_i\TP_j\eta_{\beta}(\TP_k\TP^5\eta_{\alpha}\nn^{\alpha})\dS\\
\label{ST331112}&+\sigma\int_0^T\ig\left(\frac12 g^{ij}g^{kl}-g^{ik}g^{jl}\right)\TP^5\eta^{\mu}\TP_l\eta_{\mu}\p_t(\sqrt{g}\nn^{\alpha})\nn^{\beta}\TP_i\TP_j\eta_{\beta}\TP_k\TP^5\eta_{\alpha}\dS\\
\label{ST331113}&+\sigma\int_0^T\ig\sqrt{g}\p_t\left(\left(\frac12 g^{ij}g^{kl}-g^{ik}g^{jl}\right)\TP_l\eta_{\mu}\nn^{\beta}\TP_i\TP_j\eta_{\beta}\right)\TP^5\eta^{\mu}(\nn^{\alpha}\TP_k\TP^5\eta_{\alpha})\dS+\eqref{ST33112}+\eqref{ST33113}.
\end{align}

Note that $\p_tA(\sqrt{g}\nn^{\alpha})=\p_tA^{3\alpha}=-A^{3\gamma}\p_\mu v_{\gamma}A^{\mu\alpha}$, we know $\eqref{ST331112}+\eqref{ST33112}=0$. Then the remaining quantities \eqref{ST33113}, \eqref{ST331111} and \eqref{ST331113} can all be directly controlled
\begin{align}
\label{ST331114} \eqref{ST33113}\lesssim&\int_0^T|\sqrt{\sigma}\eta|_5^2P(E_1(t))\leq \int_0^TP(E_1(t))(\sigma E_2(t))\dt,\\
\label{ST331115} \eqref{ST331111}\lesssim&\int_0^T|\sqrt{\sigma}\TP^6\eta\cdot\nn|_0|\sqrt{\sigma}v|_5P(E_1(t))\leq \int_0^TP(E_1(t))(\sigma E_2(t))\dt,\\
\label{ST331116} \eqref{ST331113}\lesssim&\int_0^T|\sqrt{\sigma}\TP^6\eta\cdot\nn|_0|\sqrt{\sigma}\eta|_5P(E_1(t))\leq \int_0^TP(E_1(t))(\sigma E_2(t))\dt.
\end{align}

Combining \eqref{ST330}-\eqref{ST331116}, we get the control of$\ST_{33}$
\begin{equation}\label{ST33}
\int_0^T\ST_{33}\dt\lesssim\int_0^TP(E_1(t))(\sigma E_2(t))\dt,
\end{equation}whic together with \eqref{ST34}, \eqref{ST31} and \eqref{ST32} gives the control of$\ST_3$
\begin{equation}\label{ST3}
\int_0^T\ST_{3}\dt\lesssim\int_0^TP(E_1(t))(\sigma E_2(t))\dt.
\end{equation}Finally, \eqref{ST0}, \eqref{ST4}, \eqref{ST1}, \eqref{ST2} and \eqref{ST3} gives the control of the boundary terms contributed by the surface tension as well as the $\wt$-weighted boundary energy
\begin{equation}\label{ST}
\int_0^T\ST\dt\lesssim-\frac{\sigma}{2}\ig\left|\TP^6\eta\cdot\nn\right|^2\dS\bigg|^T_0+\int_0^TP(E_1(t))(\sigma E_2(t))\dt.
\end{equation}

\subsubsection{Control of the error terms}\label{sect RTSTerror}
It remains to control $J_1$ and $J_2$ in \eqref{J0}. Note that $q=-\sigma\sqrt{g}\Delta_g\eta\cdot \nn=\sigma Q(\TP\eta)\TP^2\eta\cdot\nn$ on the boundary and $A^{3\alpha}=\TP_1\eta\times\TP_2\eta$. The term $J_2$ can be directly controlled
\begin{equation}\label{J2}
\begin{aligned}
J_2=&~5\sigma\ig \TP^4 A^{3\alpha} \TP(Q(\TP\eta)\TP^2\eta) (\TP^5 v_{\alpha}-\TP^5\eta\cdot\pa v_{\alpha})\dS\\
&+10\sigma\ig \TP^3 A^{3\alpha} \TP^2(Q(\TP\eta)\TP^2\eta) (\TP^5 v_{\alpha}-\TP^5\eta\cdot\pa v_{\alpha})\dS\\
&+10\sigma\ig \TP^2 A^{3\alpha} \TP^3(Q(\TP\eta)\TP^2\eta) (\TP^5 v_{\alpha}-\TP^5\eta\cdot\pa v_{\alpha})\dS\\
&+5\sigma \ig \TP A^{3\alpha} \TP^4 (Q(\TP\eta)(\TP^2\eta\cdot\nn))(\TP^5 v_{\alpha}-\TP^5\eta\cdot\pa v_{\alpha})\dS\\
\lesssim&~  P(|\TP\eta|_{L^{\infty}})|\TP^3\eta|_{L^{\infty}} \left(|\wt\eta|_5+|\wt\TP^6\eta\cdot n|_0\right)\left(|\wt \TP^5 v|_0+|\wt \TP^5\eta|_0|\pa v|_{L^{\infty}}\right)\\
\lesssim&~P(|\TP\eta|_{L^{\infty}})|\TP^3\eta|_{L^{\infty}}|\p v|_{L^{\infty}} \left(\sqrt{E_1}+\sqrt{\sigma E_2}\right)\sqrt{\sigma E_2}\leq P(E_1(t))(\sigma E_2(t)).
\end{aligned}
\end{equation}

The control of $J_1$ needs more delicate computation. Recall that $A^{3\alpha}=(\TP_1\eta\times\TP_2\eta)^{\alpha}$, we have that
\begin{equation}\label{J10}
\begin{aligned}
J_1=&\ig\sigma\hh(\TP^5\TP_1\eta\times\TP_2\eta)\cdot\TP^5\p_t\eta\dS+\ig\sigma\hh(\TP_1\eta\times\TP^5\TP_2\eta)\cdot\TP^5\p_t\eta\dS\\
&-\ig\sigma\hh(\TP^5\TP_1\eta\times\TP_2\eta)\cdot(\TP^5\eta_{\gamma}A^{\mu\gamma}\p_{\mu}v)\dS-\ig\sigma\hh(\TP_1\eta\times\TP^5\TP_2\eta)\cdot(\TP^5\eta_{\gamma}A^{\mu\gamma}\p_{\mu}v)\dS\\
&+\sum_{k=1}^4\ig\sigma\hh(\TP^k\TP_1\eta\times\TP^{4-k}\TP_2\eta)\cdot(\TP^5 v-\TP^5\eta\cdot\pa v)\dS\\
=:&~J_{11}+J_{12}+J_{13}+J_{14}+J_{15}.
\end{aligned}
\end{equation}

Again, the term $J_{15}$ is directly controlled by the weighted energy
\begin{equation}\label{J15}
J_{15}\leq P(E_1(t))(\sigma E_2(t)).
\end{equation}Below we only show the control of $J_{11}$ and $J_{13}$, and the control of $J_{12}$ and $J_{14}$ follows in the same way. For $J_{11}$, we integrate $\p_t$ by parts to get
\begin{equation}\label{J110}
\begin{aligned}
\int_0^TJ_{11}\dt=&-\int_0^T\ig\sigma\hh(\TP^5\TP_1 v\times\TP_2\eta)\cdot\TP^5 \eta\dS-\int_0^T\ig\sigma\hh(\TP^5\TP_1 \eta\times\TP_2v)\cdot\TP^5 \eta\dS\\
&-\int_0^T\ig\sigma\p_t\hh(\TP^5\TP_1 \eta\times\TP_2\eta)\cdot\TP^5 \eta\dS+\ig\sigma\hh(\TP^5\TP_1 \eta\times\TP_2\eta)\cdot\TP^5 \eta\dS\bigg|^T_0\\
=:&~J_{111}+J_{112}+J_{113}+J_{114}.
\end{aligned}
\end{equation}

Next we integrate $\TP_1$ by parts in $J_{111}$ and use the vector identity $(\uu\times\vv)\cdot\ww=-(\uu\times\ww)\cdot\vv$ to get
\begin{equation}\label{J1110}
\begin{aligned}
J_{111}\overset{\TP_1}{=}&\int_0^T\ig\sigma\hh(\TP^5 v\times\TP_2\eta)\cdot\TP_1\TP^5 \eta\dS\dt\\
&+\int_0^T\ig\sigma\TP_1\hh(\TP^5 v\times\TP_2\eta)\cdot\TP^5 \eta\dS\dt+\int_0^T\ig\sigma\hh(\TP^5 v\times\TP_1\TP_2\eta)\cdot\TP^5 \eta\dS\dt\\
=&\underbrace{-\int_0^T\ig\sigma\hh(\TP^5 \TP_1\eta\times\TP_2\eta)\cdot\TP^5 v\dS\dt}_{=-\int_0^T J_{11}\dt}\\
&+\int_0^T\ig\sigma\TP_1\hh(\TP^5 v\times\TP_2\eta)\cdot\TP^5 \eta\dS\dt+\int_0^T\ig\sigma\hh(\TP^5 v\times\TP_1\TP_2\eta)\cdot\TP^5 \eta\dS\dt\\
\lesssim&-\int_0^T J_{11}\dt+\int_0^TP(E_1(t))(\sigma E_2(t)).
\end{aligned}
\end{equation}
Therefore, we have
\begin{equation}\label{J1100}
\int_0^TJ_{11}\dt\lesssim \frac12(J_{112}+J_{113}+J_{114})+\int_0^TP(E_1(t))(\sigma E_2(t))\dt.
\end{equation}

Next we need to control $J_{114}$ by $\PP_0+P(E_1(T))(\sigma E_2(T))\int_0^TP(E_1(t))(\sigma E_2(t))\dt$. For that, we need the following identity
\[
\delta^{\alpha\beta}=\nn^{\alpha}\nn^{\beta}+g^{\ij}\TP_i\eta^{\alpha}\TP_j\eta^{\beta},
\]which yields
\begin{equation}\label{J1140}
\begin{aligned}
J_{114}=&\ig\sigma\hh (\TP^5\TP_1\eta\times\TP_2\eta)_{\alpha}\nn^{\alpha}\nn^{\beta}\TP^5\eta_{\beta}\dS+\ig\sigma\hh (\TP^5\TP_1\eta\times\TP_2\eta)_{\alpha}g^{\ij}\TP_i\eta^{\alpha}\TP_j\eta^{\beta}\TP^5\eta_{\beta}\dS\\
=:&J_{1141}+J_{1142}.
\end{aligned}
\end{equation}

In $J_{1141}$ we integrate $\TP_1$ by parts and \red{use $\eps$-Young's inequality to get}
\begin{equation}\label{J1141}
\begin{aligned}
J_{1141}\overset{\TP_1}{=}&-\ig\sigma\hh (\TP^5\eta\times\TP_2\eta)_{\alpha}\nn^{\alpha}\nn^{\beta}\TP_1\TP^5\eta_{\beta}\dS-\ig\sigma\TP_1(\mathcal{H}\nn^{\alpha}\nn^{\beta}\TP_2\eta)\TP^5\eta\TP^5\eta\dS\\
\lesssim&~\red{P(N_0)\left(|\sqrt{\sigma}\TP^5\eta|_0|\sqrt{\sigma}\TP^6\eta\cdot\nn|_0+\sigma|\TP^5\eta|_0^2\right)}\\
\lesssim&~\red{\eps|\sqrt{\sigma}\TP^6\eta\cdot\nn|_0^2+P(N_0)\sigma\|\eta\|_{5.5}^2\leq\eps|\sqrt{\sigma}\TP^6\eta\cdot\nn|_0^2+P(N_0)\left(\sigma\|\eta_0\|_{5.5}^2+\int_0^T\sigma\|v(t)\|_{5.5}^2\dt\right),}
\end{aligned}
\end{equation}
\red{where the first term on the right side can be absorbed by $E(t)$ if we pick $\eps>0$ suitably small.}
%where we used $\TP^5\eta|_{t=0}=\mathbf{0}$.

In $J_{1142}$, we notice that the integral vanishes if $i=2$ due to $(\TP^5\TP_1\eta\times\TP_2\eta)\cdot\TP_2\eta=0$. If $i=1$, then it can be controlled in the same way as $J_{1141}$
\begin{equation}\label{J1142}
\begin{aligned}
J_{1142}=&\ig\sigma\hh (\TP^5\TP_1\eta\times\TP_2\eta)\cdot\TP_1\eta~g^{1j}\TP_j\eta^{\beta}\TP^5\eta_{\beta}\dS=-\ig\sigma\hh \underbrace{(\TP_1\eta\times\TP_2\eta)}_{=\sqrt{g}\nn}\cdot\TP^5\TP_1\eta~g^{1j}\TP_j\eta^{\beta}\TP^5\eta_{\beta}\dS\\
\lesssim&~\red{P(N_0)|\sqrt{\sigma}\TP^5\TP_1\eta\cdot\nn|_0|\sqrt{\sigma}\TP^5\eta|_0\lesssim \eps|\sqrt{\sigma}\TP^6\eta\cdot\nn|_0^2+P(N_0)\left(\sigma\|\eta_0\|_{5.5}^2+\int_0^T\sigma\|v(t)\|_{5.5}^2\dt\right).}
\end{aligned}
\end{equation}

The term $J_{113}$ can be controlled in the same way as $J_{114}$ so we omit the proof. Thus we already get
\begin{equation}\label{J1134}
J_{113}+J_{114}\lesssim P(E(T))\int_0^T P(E(t))\dt.
\end{equation}

It remains to analyze $\dfrac{1}{2}J_{112}$ which should be controlled together with $J_{13}$. Again we have
\begin{equation}\label{J1120}
\begin{aligned}
\frac12 J_{112}=&\frac12 \int_0^T\ig\sigma\hh(\TP^5\TP_1 \eta\times\TP_2v)_{\alpha}\nn^{\alpha}\nn^{\beta}\TP^5 \eta_{\beta}\dS+\frac12 \int_0^T\ig\sigma\hh(\TP^5\TP_1 \eta\times\TP_2v)_{\alpha}g^{ij}\TP_i\eta^{\alpha}\TP_j\eta^{\beta}\TP^5 \eta_{\beta}\dS\\
=:&~J_{1121}+J_{1122},
\end{aligned}
\end{equation}and the control of $J_{1121}$ follows in the same way as \eqref{J1141} by integrating $\TP_1$ by parts
\begin{equation}\label{J1121}
J_{1121}\lesssim\int_0^TP(E_1(t))(\sigma E_2(t)) \dt,
\end{equation}

For $J_{1122}$ we need to do further decomposition
\begin{equation}\label{J11220}
\begin{aligned}
J_{1122}=&-\frac12 \int_0^T\ig\sigma\hh(\TP_i\eta\times\TP_2v)_{\alpha}g^{ij}\TP^5\TP_1 \eta^{\alpha}\TP_j\eta^{\beta}\TP^5 \eta_{\beta}\dS\dt\\
=&-\frac12\int_0^T\ig\sigma\hh \left((\TP_i\eta\times\TP_2v)_{\gamma}\nn^{\gamma}\nn^{\alpha}\TP^5\TP_1\eta_{\alpha}\right)g^{ij}\TP_j\eta^{\beta}\TP^5 \eta_{\beta}\dS\dt\\
&-\frac12\int_0^T\ig\sigma\hh \left((\TP_i\eta\times\TP_2v)_{\gamma}g^{kl}\TP_k\eta^{\gamma}\TP_l\eta^{\alpha}\TP^5\TP_1\eta_{\alpha}\right)g^{ij}\TP_j\eta^{\beta}\TP^5 \eta_{\beta}\dS\dt\\
=:&~J_{11221}+J_{11222},
\end{aligned}
\end{equation}where $J_{11221}$ is directly controlled by
\begin{equation}\label{J11221}
J_{11221}\lesssim\int_0^TP(E_1(t))|\sqrt{\sigma}\TP^5\TP_1\eta\cdot\nn|_0|\sqrt{\sigma}\TP^5\eta|_0 \dt\lesssim\int_0^TP(E_1(t))+\sigma E_2(t) \dt.
\end{equation}

In $J_{11222}$, the integral vanishes if $i=k$, so we only need to investigate the cases $(i,k)=(1,2)$ and $(i,k)=(2,1)$, which contribute to
\begin{equation}\label{J11222}
\begin{aligned}
J_{11222}=&-\frac12\int_0^T\ig\sigma\hh \left((\TP_1\eta\times\TP_2v)_{\gamma}g^{2l}\TP_2\eta^{\gamma}\TP_l\eta^{\alpha}\TP^5\TP_1\eta_{\alpha}\right)g^{1j}\TP_j\eta^{\beta}\TP^5 \eta_{\beta}\dS\dt\\
&-\frac12\int_0^T\ig\sigma\hh \left((\TP_2\eta\times\TP_2v)_{\gamma}g^{1l}\TP_1\eta^{\gamma}\TP_l\eta^{\alpha}\TP^5\TP_1\eta_{\alpha}\right)g^{2j}\TP_j\eta^{\beta}\TP^5 \eta_{\beta}\dS\dt\\
=&-\frac12\int_0^T\ig\sigma\hh \left((\TP_1\eta\times\TP_2v)\cdot\TP_2\eta\right)\left(g^{2l}\TP_l\eta\cdot\TP^5\TP_1\eta\right)\left(g^{1j}\TP_j\eta\cdot\TP^5 \eta\right)\dS\dt\\
&+\frac12\int_0^T\ig\sigma\hh \left((\TP_1\eta\times\TP_2v)\cdot\TP_2\eta\right)\left(g^{1l}\TP_l\eta\cdot\TP^5\TP_1\eta\right)\left(g^{2j}\TP_j\eta\cdot\TP^5 \eta\right)\dS\dt.
\end{aligned}
\end{equation}

Next, we analyze $J_{13}$. First, we do the following decomposition
\begin{align*}
\TP^5\eta_{\gamma}A^{\mu\gamma}\p_{\mu}v_{\alpha}=&~\TP^5\eta_{\beta}\nn^{\beta}\nn_{\gamma}A^{\mu\gamma}\p_{\mu}v_{\alpha}+\TP^5\eta_{\beta}g^{ij}\TP_i\eta^{\beta}\underbrace{\TP_j\eta_{\gamma}A^{\mu\gamma}}_{=\delta^{\mu}_j}\p_{\mu}v_{\alpha}\\
=&~\TP^5\eta_{\beta}\nn^{\beta}\nn_{\gamma}A^{\mu\gamma}\p_{\mu}v_{\alpha}+\TP^5\eta_{\beta}g^{ij}\TP_i\eta^{\beta}\TP_jv_{\alpha},
\end{align*}and thus
\begin{equation}\label{J130}
\begin{aligned}
J_{13}=&~\ig\sigma\hh(\TP^5\TP_1\eta\times\TP_2\eta)\cdot(\TP^5\eta_{\beta}\nn^{\beta}\nn_{\gamma}A^{\mu\gamma}\p_{\mu}v)\dS+\ig\sigma\hh(\TP^5\TP_1\eta\times\TP_2\eta)\cdot(\TP^5\eta_{\beta}g^{ij}\TP_i\eta^{\beta}\TP_jv)\dS\\
=:&~J_{131}+J_{132}.
\end{aligned}
\end{equation}

The term $J_{131}$ can be controlled similarly as $J_{1141}$ in \eqref{J1141}, i.e., integrating $\TP_1$ by parts,
\begin{equation}\label{J131}
J_{131}\lesssim P(E_1(t))(\sigma E_2(t)).
\end{equation}

In $J_{132}$, we need to do further decomposition
\begin{equation}\label{J1320}
\begin{aligned}
J_{132}=&\ig\sigma\hh(\TP^5\TP_1\eta\times\TP_2\eta)_{\gamma}\nn^{\gamma}\nn^{\alpha}(\TP^5\eta_{\beta}g^{ij}\TP_i\eta^{\beta}\TP_jv_{\alpha})\dS+\ig\sigma\hh(\TP^5\TP_1\eta\times\TP_2\eta)_{\gamma}g^{kl}\TP_k\eta^{\gamma}\TP_l\eta^{\alpha}(\TP^5\eta_{\beta}g^{ij}\TP_i\eta^{\beta}\TP_jv_{\alpha})\dS\\
=:&~J_{1321}+J_{1322}.
\end{aligned}
\end{equation}

The integral in $J_{1322}$ vanishes if $k=2$. When $k=1$, we again use the vector identity $(\uu\times\vv)\cdot\ww=-(\uu\times\ww)\cdot\vv$ and invoke $(\TP_1\eta\times\TP^2\eta)=\sqrt{g}\nn$ to get
\begin{equation}\label{J1322}
\begin{aligned}
J_{1322}=&\ig\sigma\hh(\TP^5\TP_1\eta\times\TP_2\eta)_{\gamma}g^{1l}\TP_1\eta^{\gamma}\TP_l\eta^{\alpha}(\TP^5\eta_{\beta}g^{ij}\TP_i\eta^{\beta}\TP_jv_{\alpha})\dS\\
=&-\ig\sigma\hh\left((\TP_1\eta\times\TP_2\eta)\cdot\TP^5\TP_1\eta\right)g^{1l}\TP_l\eta^{\alpha}(\TP^5\eta_{\beta}g^{ij}\TP_i\eta^{\beta}\TP_jv_{\alpha})\dS\\
\lesssim&~|\sqrt{\sigma}\TP^5\TP_1\eta\cdot\nn|_0|\sqrt{\sigma}\TP^5\eta|_0 |\hh~g^2\TP\eta~\TP\eta~\TP v|_{L^{\infty}}\lesssim P(E_1(t))(\sigma E_2(t))
\end{aligned}
\end{equation}

We recall $\nn_{\gamma}=\sqrt{g}^{-1}(\TP_1\eta\times\TP_2\eta)_{\gamma}$ and use the vector identities $(\uu\times\vv)\cdot\ww=-(\uu\times\ww)\cdot\vv$ and $\uu\times(\vv\times\ww)=(\uu\cdot\ww)\vv-(\uu\cdot\vv)\ww$ to get
\begin{align*}
(\TP^5\TP_1\eta\times\TP_2\eta)\cdot(\TP_1\eta\times\TP_2\eta)=&-(\TP^5\TP_1\eta\times(\TP_1\eta\times\TP_2\eta))\cdot\TP_2\eta\\
=&-(\TP^5\TP_1\eta\cdot\TP_2\eta)\underbrace{(\TP_1\eta\cdot\TP_2\eta)}_{=g_{12}=-(\det g)g^{12}}+(\TP^5\TP_1\eta\cdot\TP_1\eta)\underbrace{(\TP_2\eta\cdot\TP_2\eta)}_{=g_{22}=(\det g) g^{11}}.
\end{align*}Plugging this into $J_{1321}$ yields
\begin{equation}\label{J13210}
\begin{aligned}
J_{1321}=&\ig\sigma\hh(\TP^5\TP_1\eta\times\TP_2\eta)\cdot(\TP_1\eta\cdot\TP_2\eta)\sqrt{g}^{-1}\nn^{\alpha}(\TP^5\eta_{\beta}g^{ij}\TP_i\eta^{\beta}\TP_jv_{\alpha})\dS\\
=&\ig\sigma\hh(g^{1l}\TP_l\eta\cdot\TP^5\TP_1\eta )(g^{ij}\TP_i\eta\cdot\TP^5\eta)(\TP_jv_{\alpha}\nn^{\alpha}\sqrt{g})\dS\\
=&\ig\sigma\hh(g^{1l}\TP_l\eta\cdot\TP^5\TP_1\eta )(g^{1i}\TP_i\eta\cdot\TP^5\eta)(\TP_1v\cdot(\TP_1\eta\times\TP_2\eta))\dS+\ig\sigma\hh(g^{1l}\TP_l\eta\cdot\TP^5\TP_1\eta )(g^{2i}\TP_i\eta\cdot\TP^5\eta)(\TP_2v\cdot(\TP_1\eta\times\TP_2\eta))\dS\\
=:&~J_{13211}+J_{13212}.
\end{aligned}
\end{equation}

Integrating $\TP_1$ by parts in $J_{13211}$, the highest order term is the same as $J_{13211}$ itself but with a minus sign. Therefore,
\begin{equation}\label{J13211}
\begin{aligned}
J_{13211}=&-\frac12\ig\sigma(g^{1l}\TP_l\eta\cdot\TP^5\eta)(g^{1i}\TP_i\eta\cdot\TP^5\eta)\TP_1(\hh\TP_1v\cdot(\TP_1\eta\times\TP_2\eta))\dS\\
\lesssim&~|\sqrt{\sigma}\TP^5\eta|_0^2 P(|\TP\eta|_{L^{\infty}})(|\TP^2\eta~\TP v|_{L^{\infty}})\lesssim P(E_1(t))(\sigma E_2(t)).
\end{aligned}
\end{equation}

Now $J_{13212}$ reads
\begin{equation}\label{J13212}
J_{13212}=-\ig\sigma\hh(g^{1l}\TP_l\eta\cdot\TP^5\TP_1\eta )(g^{2i}\TP_i\eta\cdot\TP^5\eta)((\TP_1\eta\times\TP_2v)\cdot\TP_2\eta)\dS,
\end{equation}which together with \eqref{J11222} yields that
\begin{equation}
\begin{aligned}
J_{11222}+J_{13212}=&-\frac12\ig\sigma\hh(g^{1l}\TP_l\eta\cdot\TP^5\TP_1\eta )(g^{2i}\TP_i\eta\cdot\TP^5\eta)((\TP_1\eta\times\TP_2v)\cdot\TP_2\eta)\dS\\
&-\frac12\ig\sigma\hh(g^{1l}\TP_l\eta\cdot\TP^5\eta )(g^{2i}\TP_i\eta\cdot\TP^5\TP_1\eta)((\TP_1\eta\times\TP_2v)\cdot\TP_2\eta)\dS,
\end{aligned}
\end{equation}and thus integrating $\TP_1$ by parts in the first integral yields the cancellation with the second integral due to the symmetry
\begin{equation}\label{J11J13}
\begin{aligned}
J_{11222}+J_{13212}=&~\frac12\ig\sigma\hh(\TP_1(g^{1l}\TP_l\eta)\cdot\TP^5\eta)(g^{2i}\TP_i\eta\cdot\TP^5\eta)((\TP_1\eta\times\TP_2v)\cdot\TP_2\eta)\dS\\
&+\frac12\ig\sigma\hh(g^{1l}\TP_l\eta\cdot\TP^5\eta )(\TP_1(g^{2i}\TP_i\eta)\cdot\TP^5\eta)((\TP_1\eta\times\TP_2v)\cdot\TP_2\eta)\dS\\
&+\frac12\ig\sigma\hh(g^{1l}\TP_l\eta\cdot\TP^5\eta )(g^{2i}\TP_i\eta\cdot\TP^5\eta)\TP_1((\TP_1\eta\times\TP_2v)\cdot\TP_2\eta)\dS\\
\lesssim&~|\sqrt{\sigma}\TP^5\eta|_0^2 P(E_1(t))\leq P(E_1(t))(\sigma E_2(t)).
\end{aligned}
\end{equation}

Summarizing \eqref{J10}-\eqref{J110}, \eqref{J1134}-\eqref{J11221}, \eqref{J130}-\eqref{J11J13}, we conclude the estimate of $J_1$ by
\begin{equation}\label{J1}
\int_0^TJ_1\dt\lesssim P(E_1(T))(\sigma E_2(T))\int_0^TP(E_1(t))\dt+\int_0^TP(E_1(t))(\sigma E_2(t))\dt.
\end{equation}

Finally, combining \eqref{J0}, \eqref{RT}, \eqref{ST}, \eqref{J2} and \eqref{J1}, we conclude the $\TP^5$-boundary estimate by
\begin{equation}\label{J}
\int_0^TJ\dt\lesssim -\frac{c_0}{4}\left|\TP^5\eta\cdot \nn\right|_0^2-\frac{\sigma}{2}\left|\TP^6\eta\cdot\nn\right|_0^2+P(E_1(T))(\sigma E_2(T))\int_0^TP(E_1(t))\dt+\int_0^TP(E_1(t))(\sigma E_2(t))\dt.
\end{equation}

\subsection{Finalizing the tangential estimate of spatial derivatives}\label{sect tgv5b5}

Summarizing \eqref{goodenergy0}-\eqref{I1} and \eqref{J}, we conclude the estimate of the Alinhac good unknowns by
\begin{equation}\label{goodenergy}
\|\VV\|_0^2+\frac{c_0}{4}\left|\TP^5\eta\cdot \nn\right|_0^2+\frac{\sigma}{2}\left|\TP^6\eta\cdot\nn\right|_0^2\bigg|_{t=T}\lesssim P(\|v_0\|_5)+P(E_1(T))(\sigma E_2(T))\int_0^TP(E_1(t))\dt+\int_0^TP(E_1(t))(\sigma E_2(t))\dt.
\end{equation}Finally, from the definition of the good unknowns \eqref{good5} and $\TP^5\eta|_{t=0}=\mathbf{0}$, we know 
\[
\|\TP^5 v(T)\|_0^2\lesssim\|\VV(T)\|_0^2+\|\TP^5\eta(T)\|_0^2\|\pa v(T)\|_{L^{\infty}}^2\lesssim \|\VV(T)\|_0^2+ P(E_1(T))\int_0^T\|\TP^5 v(t)\|_0^2\dt,
\]and thus
\begin{equation}\label{tgs5}
\|\TP^5 v\|_0^2+\|\TP^5\bp\eta\|_0^2+\frac{c_0}{4}\left|\TP^5\eta\cdot \nn\right|_0^2+\frac{\sigma}{2}\left|\TP^6\eta\cdot\nn\right|_0^2\bigg|_{t=T}\lesssim P(\|v_0\|_5)+P(E_1(T))(\sigma E_2(T))\int_0^TP(E_1(t))(\sigma E_2(t))\dt.
\end{equation}

\section{Tangential estimates of time derivatives}\label{sect tgmix}

Now we consider the tangential estimates involving time derivatives. Due to the appearance of time derivatives, we no longer need the Alinhac good unknowns thanks to the fact that $A$ has the same spatial regularity as $\p_t A$. Instead, we adapt similar techniques developed in our previous work \cite{luozhangMHDST3.5} to reveal the subtle cancellation structure and $\wt$-weighted energy terms on the boundary. 

\subsection{Full time derivatives}\label{sect tgt5}

First, we do the $\p_t^5$-estimate which is the most difficult part in the tangential estimates of time derivatives. The reason is two-fold
\begin{itemize}
\item We do not have any estimate for the full-time derivatives of $q$,
\item The full-time derivatives $\p_t^5 v$ and $\p_t^5\bp\eta$ only have $L^2$ interior regularity and thus have no control on the boundary due to the failure of the trace lemma.
\end{itemize}

From the second equation of \eqref{MHDLST}, we have
\begin{equation}\label{tgt50}
\begin{aligned}
&\frac12\|\p_t^5v\|_0^2+\frac12\|\p_t^5\bp\eta\|_0^2\bigg|^T_0\\
=&\underbrace{-\int_0^T\io\p_t^5(A^{\mu\alpha}\p_\mu q)\p_t^5v_{\alpha}\dy\dt}_{=:I}\\
&+\int_0^T\io\bp\p_t^5(\bp\eta^{\alpha})\p_t^5v_{\alpha}\dy\dt+\int_0^T\io\p_t^5(\bp v)\cdot\p_t^5(\bp\eta)\dy\dt.
\end{aligned}
\end{equation}

It is not difficult to see that the second integral cancels with the third one after integrating $\bp$ by parts. Therefore it suffices to control $I$ that reads
\begin{equation}\label{I0}
\begin{aligned}
I=&-\int_0^T\io A^{\mu\alpha}\p_t^5\p_\mu q\p_t^5v_{\alpha}\dy\dt-5\int_0^T\io \p_tA^{\mu\alpha}\p_t^4\p_\mu q\p_t^5v_{\alpha}\dy\dt\\
&-10\int_0^T\io \p_t^2A^{\mu\alpha}\p_t^3\p_\mu q\p_t^5v_{\alpha}\dy\dt-10\int_0^T\io \p_t^3A^{\mu\alpha}\p_t^2\p_\mu q\p_t^5v_{\alpha}\dy\dt\\
&-5\int_0^T\io \p_t^4A^{\mu\alpha}\p_t\p_\mu q\p_t^5v_{\alpha}\dy\dt-\int_0^T\io \p_t^5A^{\mu\alpha}\p_\mu q\p_t^5v_{\alpha}\dy\dt\\
=:&~I_1+I_2+I_3+I_4+I_5+I_6,
\end{aligned}
\end{equation}where $I_2\sim I_5$ can all be directly controlled
\begin{equation}\label{I25}
I_2+\cdots+I_5\lesssim \int_0^T P(\|\p_t^5v\|_0,\|\p_t^4 v,\p_t^4 q\|_1,\|\p_t^3 v,\p_t^3 q\|_2,\|\p_t^2v,\p_t^2 q\|_3,\|\p_t v,\p_t q,v,q,\eta\|_3)\dt
\end{equation}

For $I_1$, we integrate $\p_{\mu}$ by parts and then invoke the surface tension equation to get
\begin{equation}\label{I10}
\begin{aligned}
I_1=&-\int_0^T\ig A^{3\alpha}\p_t^5 q\p_t^5v_{\alpha}\dS\dt+\int_0^T\io A^{\mu\alpha}\p_t^5q\p_{\mu}\p_t^5v_{\alpha}\dy\dt\\
=&-\int_0^T\ig \p_t^5(A^{3\alpha} q)\p_t^5v_{\alpha}\dS\dt+\int_0^T\ig\p_t^5A^{3\alpha}q\p_t^5v_{\alpha}\dS\dt\\
&+\sum_{k=1}^4\binom{5}{k}\int_0^T\ig\p_t^kA^{3\alpha}\p_t^{5-k}q\p_t^5v_{\alpha}\dS\dt+\int_0^T\red{\io} A^{\mu\alpha}\p_t^5q\p_{\mu}\p_t^5v_{\alpha}\dS\dt\\
=:&~I_{11}+I_{12}+I_{13}+I_{14}.
\end{aligned}
\end{equation}

Apart from $I_{11}$, the most difficult term is $I_{12}$ since $\p_t^5 v$ cannot be controlled on the boundary. However, to produce a cancellation, we can integrate $\p_{\mu}$ by parts in $I_{6}$. Invoking Piola's identity and the boundary conditions on $\Gamma_0$, we have
\begin{equation}\label{I60}
\begin{aligned}
I_6\overset{\p_{\mu}}{=}&-\int_0^T\ig\p_t^5A^{3\alpha} q \p_t^5v_{\alpha}\dS\dt+\underbrace{\int_0^T\io\p_t^5A^{\mu\alpha}q\p_t^5\p_{\mu}v_{\alpha}\dy\dt}_{=:I_{61}}\\
=&-I_{12}+I_{61}.
\end{aligned}
\end{equation}

Next we control $I_{61}$. Recall that $A^{1\cdot}=\TP_2\eta\times\p_3\eta,~A^{2\cdot}=\p_3\eta\times\TP_1\eta,~A^{3\cdot}=\TP_1\eta\times\TP_2\eta$ which implies
\begin{equation}\label{I61}
\begin{aligned}
I_{61}=&\int_0^T\io q\p_t^5A^{1\alpha}\p_t^5\TP_1v_{\alpha}+q\p_t^5A^{2\alpha}\p_t^5\TP_2v_{\alpha}+q\p_t^5A^{3\alpha}\p_t^5\p_3v_{\alpha}\dy\dt\\
=&\int_0^T\io q(\p_t^4\TP_2 v\times \p_3\eta)\cdot\TP_1\p_t^5 v\dy\dt+\int_0^T\io q(\TP_2 \eta\times \p_t^4\p_3v )\cdot\TP_1\p_t^5 v\dy\dt\\
&+\int_0^T\io q(\p_t^4\p_3 v\times \TP_1\eta)\cdot\TP_2\p_t^5 v\dy\dt+\int_0^T\io q(\p_3 \eta\times \p_t^4\TP_1v )\cdot\TP_2\p_t^5 v\dy\dt\\
&+\int_0^T\io q(\p_t^4\TP_1 v\times \TP_2\eta)\cdot\p_3\p_t^5 v\dy\dt+\int_0^T\io q(\TP_1 \eta\times \p_t^4\TP_2v )\cdot\p_3\p_t^5 v\dy\dt\\
&+\text{ lower order terms}=:I_{611}+\cdots+I_{616}+\text{ lower order terms}.
\end{aligned}
\end{equation} From the vector identity $(\uu\times \vv)\cdot\ww=(\vv\times\ww)\cdot\uu$, we can divide these 6 terms into 3 pairs: $I_{611}+I_{614}$, $I_{612}+I_{615}$ and $I_{613}+I_{616}$.
\begin{align}
\label{I6a} I_{611}+I_{614}=&\underbrace{\io q(\p_t^4\TP_2 v\times \p_3\eta)\cdot\TP_1\p_t^4 v\dy\bigg|^T_0}_{I_{6a}}+\text{ lower order terms},\\
\label{I6b} I_{612}+I_{615}=&\underbrace{\io q(\p_t^4\p_3 v\times \TP_1\eta)\cdot\TP_2\p_t^4 v\dy\bigg|^T_0}_{I_{6b}}+\text{ lower order terms},\\
\label{I6c} I_{613}+I_{616}=&\underbrace{\io q(\p_t^4\TP_1 v\times \TP_2\eta)\cdot\p_3\p_t^4 v\dy\bigg|^T_0}_{I_{6c}}+\text{ lower order terms}.
\end{align}

Now we shall control $I_{6a},~I_{6b},~I_{6c}$ by $\PP_0+P(E_1(T))\int_0^TP(E_1(t))\dt$. For $I_{6a}$, we can equivalently write it to be
\[
I_{6a}=\io q\epsilon^{\alpha\beta\gamma}\p_t^4\TP_2v_{\alpha}\p_3\eta_{\beta}\TP_1\p_t^4v_{\gamma}\dy.
\]

\red{We take $\beta=1$ for an example and the cases $\beta=2,3$ are treated in the same way.} When $\beta=1$, $\alpha,\gamma$ can only be 2 or 3. Note that $\epsilon^{312}=-\epsilon^{213}=1$, we know $I_{6a}|_{\beta=1}$ only consists of two terms
\begin{equation}\label{I6a10}
I_{6a}|_{\beta=1}=\io q \p_3\eta_1 \TP_1\p_t^4v_3\TP_2\p_t^4v_2\dy-\io q \p_3\eta_1 \TP_1\p_t^4v_1\TP_2\p_t^4v_3\dy.
\end{equation}Then we integrate $\TP_2$ by parts in the first term and $\TP_1$ in the second term. We notice that the highest order terms cancel with each other
\red{\begin{equation}\label{I6a1}
\begin{aligned}
I_{6a}|_{\beta=1}=&-\io q\p_3\eta_1(\underbrace{\TP_1\TP_2\p_t^4v_3\p_t^4v_2-\p_t^4v_2\TP_1\TP_2\p_t^4v_3}_{=0})\dy\\
&-\io\TP_2 (q\p_3\eta_1)\TP_1\p_t^4v_1\p_t^4v_2-\TP_1 (q\p_3\eta_2)\p_t^4v_2\TP_2\p_t^4v_1\dy+\text{ lower order terms}\\
\lesssim&~\|q\p\eta\|_3\|\p_t^4 v\|_0\|\p_t^4 v\|_1\lesssim\eps\|\p_t^4 v\|_1^2+\PP_0+\int_0^T P(E_1(t))\dt.
\end{aligned}
\end{equation}}

% Again we expand the components of the cross product to get
%\[
%I_{6b}=\io q\epsilon^{\alpha\beta\gamma}\p_t^4\p_3v_{\alpha}\TP_1\eta_{\beta}\TP_2\p_t^4v_{\gamma}\dy.
%\]When $\beta=2,3$, $\TP_1\eta_{\beta}|_{t=0}=0$ and thus $I_{6b}|_{\beta=2,3}\lesssim\PP_0+P(E_1(T))\int_0^TP(E_1(t))\dt$. When $\beta=1$, we write
%\[
%I_{6b}|_{\beta=1}=\io q\TP_1\eta_1(\TP\p_t^4v_2\p_3\p_t^4v_3-\TP_2\p_t^4v_3\p_3\p_t^4v_2)\dy.
%\] 
The estimates of $I_{6b}$ and $I_{6c}$ can proceed in the same way, but integrating by parts in $y_3$ yields an extra boundary term.  \red{We only show how to control $I_{6b}$ when $\beta=1$. Again we can integrate $\p_3,\TP_1$ by parts in each term respectively and produce similar cancellation as in \eqref{I6a1}. We also have to control the boundary term as follows
\begin{equation}\label{I6b1}
\begin{aligned}
I_{6b1}:=&\ig q\TP_1\eta_1\p_t^4v_2\TP_2\p_t^4v_3\dS=\ig \sigma\hh\TP_1\eta_1\p_t^4v_2\TP_2\p_t^4v_3\dS\\
\lesssim&~\sqrt{\sigma}\|\sqrt{\sigma}\p_t^4 v\|_{1.5}\|\p_t^4 v\|_{0.5}|\TP^2\eta|_{L^{\infty}}|\TP\eta|_{L^{\infty}}\lesssim\varepsilon\|\sqrt{\sigma}\p_t^4 v\|_{1.5}^2+\sigma P(N_0)\|\p_t^4 v\|_{0.5}^2,
\end{aligned}
\end{equation}%where we used $\hh|_{t=0}=0$ due to the appearance of $\TP^2\eta$ in $\hh$. 
where the first term can be absorbed by $\sigma E_2(t)$ and the second term can be controlled either using smallness of $\sigma$ or using interpolation (not necessarily for small $\sigma$)
\begin{equation}\label{I6b'}
\|\p_t^4 v\|_{0.5}^2\lesssim\|\p_t^4 v\|_{1}\|\p_t^4 v\|_{0}\lesssim\varepsilon\|\p_t^4 v\|_1^2+\|\p_t^4v\|_0^2\lesssim\varepsilon E_1(t)+\PP_0+\int_0^T\|\p_t^5 v\|_0^2\dt.
\end{equation}}

Therefore, we get 
\begin{equation}\label{I6}
I_{6}+I_{12}\lesssim\eps(\|\p_t^4 v\|_1^2+\red{\|\sqrt{\sigma}\p_t^4 v\|_{1.5}^2})+\PP_0+P(E_1(T))(\sigma E_2(T))\int_0^T P(E_1(t))\dt.
\end{equation}

We then start to control $I_{11}$ in \eqref{I10} which is expected to produce the weighted boundary energy $\sigma|\TP\p_t^4v\cdot\nn|_0^2$. Below we use  $\eql$ to denote equality modulo error terms that are effective of lower order. Plugging the surface tension equation $A^{3\alpha}q=-\sigma\sqrt{g}\Delta_g\eta^{\alpha}$ and the identity \eqref{tplapg}, we get
\begin{equation}\label{I110}
\begin{aligned}
I_{11}\eql&-\sigma\int_0^T\ig\sqrt{g}g^{ij}\nn^{\alpha}\nn^{\beta}\p_t^4\TP_jv_{\beta}\TP_i\p_t^5v_{\alpha}\dS\dt\\
&-\sigma\int_0^T\ig\sqrt{g}(g^{ij}g^{kl}-g^{ik}g^{jl})\TP_j\eta^{\alpha}\TP_k\eta^{\lambda}\TP_l\p_t^4v^{\lambda}\TP_i\p_t^5v_{\alpha}\dS\dt\\
=:&~I_{111}+I_{112}.
\end{aligned}
\end{equation}
And the term $I_{111}$ contributes to the weighted boundary energy
\red{\begin{equation}\label{I111}
\begin{aligned}
I_{111}=&-\frac{\sigma}{2}\ig\sqrt{g}g^{ij}\left(\p_t^4\TP_i v\cdot\nn\right)\left(\p_t^4\TP_j  v\cdot\nn\right)\dS\bigg|^T_0+\sigma\int_0^T\ig\p_t(\sqrt{g}g^{ij}\nn^{\alpha}\nn^{\beta})\p_t^4\TP_i v_{\alpha}\p_t^4\TP v_{\beta}\dS\dt\\
=&: I_{1111}+I_{1112}
\end{aligned}
\end{equation}where $I_{1111}$ gives the $\sqrt{\sigma}$-weighted boundary energy because $\sqrt{g}g^{ij}$ is positive-definite.
\begin{equation}\label{I1111}
I_{1111}\lesssim -\frac{\sigma}{2} \left|\sqrt{\sigma}\TP\p_t^4 v\cdot\nn\right|_0^2\bigg|^T_0,
\end{equation}}and $I_{1112}$ is directly controlled by the weighted energy
\begin{equation}\label{I1112}
I_{1112}\lesssim\int_0^T|P(\TP\eta)\TP v|_{L^{\infty}}\|\sqrt{\sigma} \p_t^4 v\|_{1.5}^2\dt\lesssim \int_0^TP(E_1(t))(\sigma E_2(t))\dt.
\end{equation}

The control of $I_{112}$ requires a remarkable identity, first introduced by Coutand-Shkoller \cite[(12.10)-(12.11)]{coutand2007LWP} to rewrite $I_{112}$. Here we only list the result and we refer to our previous work \cite[(4.33)-(4.37)]{GLZ} for the details.
\begin{equation}\label{I1120}
I_{112}=\int_0^T\ig\frac{\sigma}{\sqrt{g}}(\p_t\det\AAA^{1}+\det\AAA^2+\det\AAA^3),
\end{equation}where
\[\AAA^1=
\begin{bmatrix}
\TP_1\eta_{\mu}\p_t^4\TP_1 v^{\mu}&\TP_1\eta_{\mu}\TP_2\p_t^4 v^{\mu}\\
\TP_2\eta_{\mu}\p_t^4\TP_1 v^{\mu}&\TP_2\eta_{\mu}\TP_2\p_t^4 v^{\mu}
\end{bmatrix},
\] and $\AAA^2_{ij}=\TP_i v_{\mu}\TP_j\p_t^4 v^{\mu}$ and $\AAA^{3}_{ij}=\TP_i\eta_{\mu}\TP_j\p_t^4 v^{\mu}$. Under this setting we have
\begin{equation}\label{I112}
I_{112}\lesssim\epsilon\|\wt\p_t^4 v\|_{1.5}^2+\PP_0+\int_0^T\|\wt \p_t^4 v\|_{1.5}^2 P(E_1(t))\dt.
\end{equation}

It now remains to control $I_{13}$ in \eqref{I10}. First we recall that $A^{3\cdot}=\TP_1\eta\times\TP_2\eta$.
\begin{equation}\label{I130}
\begin{aligned}
I_{13}\eql&~5\int_0^T\ig \p_t q (\p_t^3\TP v\times\TP\eta)\cdot\p_t^5 v\dS\dt+10\int_0^T\ig \p_t^2 q (\p_t^2\TP v\times\TP\eta)\cdot\p_t^5 v\dS\dt\\
&+10\int_0^T\ig \p_t^3 q (\p_t\TP v\times\TP\eta)\cdot\p_t^5 v\dS\dt+5\int_0^T\ig \p_t^4 q (\TP v\times\TP\eta)\cdot\p_t^5 v\dS\dt\\
=:&~I_{131}+\cdots+I_{134}.
\end{aligned}
\end{equation}
We only show the control of $I_{131}$ and the rest three terms are easier. For $I_{131}$, we integrate $\p_t$ by parts to get
\begin{equation}\label{I131}
\begin{aligned}
I_{131}\eql&-5\sigma\int_0^T\ig\p_t\hh (\p_t^4\TP v\times\TP\eta)\cdot\p_t^4 v\dS\dt+5\sigma\ig \p_t\hh(\p_t^3\TP v\times\TP\eta)\cdot\p_t^4 v\dS\\
\eql&-5\sigma\int_0^T\ig\p_t\hh (\p_t^4\TP v\times\TP\eta)\cdot\p_t^4 v\dS\dt-5\sigma\ig \p_t\hh(\p_t^3v\times\TP\eta)\cdot\TP\p_t^4 v\dS\\
\lesssim&\sqrt{\sigma}\int_0^T\|\p_t^4 v\|_1\|\wt\p_t^4 v\|_{1.5}|\TP\eta~\p_t\hh|_{L^{\infty}}+5\sigma\|\p_t^4 v\|_{1.5}\|\p_t^3 v\|_{1}|\p_t\hh~\TP\eta|_{L^{\infty}}\\
\lesssim&\sqrt{\sigma}\int_0^TP(E_1(t))(\sqrt{\sigma~E_2(t)})\dt+5\sigma\|\p_t^4 v\|_{1.5}\|\p_t^3 v\|_{1}|\p_t\hh~\TP\eta|_{L^{\infty}}.
\end{aligned}
\end{equation}The last term should be controlled by using Young's inequality
\begin{align*}
&\sigma\|\p_t^4 v\|_{1.5}\|\p_t^3 v\|_{1}|\p_t\hh~\TP\eta|_{L^{\infty}}\lesssim\sigma\left(\eps\|\p_t^4 v\|_{1.5}^2+\|\p_t^3 v\|_{1}^2|\p_t\hh~\TP\eta|_{L^{\infty}}^2\right)\\
\lesssim&~\eps\|\wt\p_t^4 v\|_{1.5}^2+\sigma\|\p_t^3 v\|_{1}^2|\p_t\hh~\TP\eta|_{L^{\infty}}^2\\
\lesssim&~\eps\|\wt\p_t^4 v\|_{1.5}^2+\sigma\left(\PP_0+\int_0^T P(E_1(t))\dt\right).
\end{align*} 

\red{Finally, we need to control $I_{14}=\int_0^T\io A^{\mu\alpha}\p_t^5q\p_{\mu}\p_t^5v_{\alpha}\dS\dt$. The proof is parallel to our previous work \cite[(4.41)-(4.49)]{GLZ}. In the proof, we will see that the bound for initial data $\|\p_t^4 q(0)\|_1$ is necessary. }
\red{By the div-free condition for $v$, we can write $$A^{\mu\alpha}\p_\mu\p_t^5v_{\alpha}=\p_t^5(\underbrace{\nabla_A\cdot v}_{=0})-\p_t^5 A^{\mu\alpha}\p_{\mu}v_{\alpha}+\text{ lower order terms}$$ and thus it suffices to control $$L:=-\int_0^T\io \p_t^5 q \p_t^5 A^{\mu\alpha}\p_{\mu}v_{\alpha} \dy\dt.$$}
\red{Then we write $\p_t^5A^{\mu\alpha}=-\p_t^4(A^{\mu\nu}\p_{\beta} v_{\nu}A^{\beta\alpha})=-A^{\mu\nu}\p_{\beta} \p_t^4v_{\nu}A^{\beta\alpha}+\cdots$, where the omitted terms are of lower order and can be directly controlled by integrating by parts in $t$. This also explains why we need the bounds for $\|\p_t^4 q(0)\|_1$. So we have
\[
L\eql \int_0^T\io \p_t^5 q A^{\mu\nu}\p_{\beta} \p_t^4v_{\nu}A^{\beta\alpha} \p_{\mu}v_{\alpha} \dy\dt=\int_0^T\io \p_t^5 (A^{\beta\alpha} q) A^{\mu\nu}\p_{\beta} \p_t^4v_{\nu}\p_{\mu}v_{\alpha} \dy\dt+\cdots
\]where the omitted terms can be directly controlled. Next we integrate $\p_\beta$ by parts and use $A^{3\alpha} q=-\sigma\sqrt{g}\Delta_g \eta\cdot\nn \nn^{\alpha}$ to get
\begin{equation}
\begin{aligned}
L\eql&-\int_0^T\ig \sigma \p_t^5 (\sqrt{g}\Delta_g \eta\cdot\nn \nn^{\alpha}) A^{\mu\nu} \p_t^4v_{\nu}\p_{\mu}v_{\alpha} \dS\dt\\
&-\int_0^T\io A^{\beta\alpha}\p_{\beta}\p_t^5 q A^{\mu\nu}\p_t^4 v_\nu \p_\mu v_{\alpha}\dy\dt\\
=&:L_1+L_2.
\end{aligned}
\end{equation}}
\red{The term $L_2$ can be directly controlled if we integrate by parts in $t$
\begin{equation}\label{I14L2}
\begin{aligned}
L_2\lesssim&\int_0^T \|\p_t^4 q\|_1\|\p_t^5 v\|_0 \|\eta\|_3^2\|v\|_3-\io A^{\beta\alpha}\p_{\beta}\p_t^4 q A^{\mu\nu}\p_t^4 v_\nu \p_\mu v_{\alpha}\dy\bigg|^T_0\\
\lesssim&\PP_0+\int_0^T P(E_1(t))\dt+\eps\|\p_t^4 q\|_1^2.
\end{aligned}
\end{equation}}
\red{For the term $L_1$, we note that $\Delta_g\eta\cdot\nn=P(\TP\eta)\TP^2\eta\cdot\nn$, so we can integrate one $\TP$ by parts to get
\begin{equation}\label{I14L1}
L_1\eql\int_0^T\ig P(\TP\eta)\p v~(\wt\TP\p_t^4 v\cdot\nn)(\wt\TP\p_t^4 v)\dS\dt\lesssim P(N_0)\int_0^T \sigma E_2(t)\dt.
\end{equation}}
\red{Therefore we have the control of $I_{14}$
\begin{equation}\label{I14}
I_{14}\lesssim \eps\|\p_t^4 q\|_1^2+\PP_0+\int_0^T P(E_1(T))(\sigma E_2(T))\dt.
\end{equation}}

Summarizing \eqref{tgt50}-\eqref{I25}, \eqref{I6}-\eqref{I131} and \eqref{I14}, we conclude the $\p_t^5$ estimates as follows
\begin{equation}\label{tgt5}
\begin{aligned}
&\|\p_t^5 v\|_0^2+\|\p_t^5\bp\eta\|_0^2+\frac{\sigma}{2}\left|\TP\p_t^4 v\cdot\nn\right|_0^2\\
\lesssim&~ \eps\|\wt\p_t^4 v\|_{1.5}^2+\eps\|\p_t^4 q\|_1^2+\PP_0+ P(E_1(T))(\sigma E_2(T))\int_0^T P(E_1(t))(\sigma E_2(t))\dt.
\end{aligned}
\end{equation}

Since we consider the zero surface tension limit, we can take $\sigma>0$ sufficiently small to absorb the term $\sigma E_1(t)$ to LHS in the final step.

\subsection{Mixed space-time derivatives}\label{sect tg5mix}

Replace $\p_t^5$ in Section \ref{sect tgt5} by $\dd^{4}\p_t$ with $\dd=\TP$ or $\p_t$, we can similarly get the $\dd^4\p_t$ tangential estimates. The proof is similar and even easier since $\p_t^k v$ ($1\leq k\leq 4$) has at least $H^{5-k}$ regularity and may be controlled in $H^{4.5-k}$ norm on the boundary by using the trace lemma. So we omit the proof and only list the result. For the related details, one can refer to our previous work \cite[Section 4.2]{GLZ}.
\begin{equation}\label{tgmix}
\begin{aligned}
&\sum_{k=1}^4\|\TP^{5-k}\p_t^k v\|_{0}^2+\|\TP^{5-k}\p_t^k\bp\eta\|_{0}^2+\frac{\sigma}{2}\left|\TP^{6-k}\p_t^{k-1} v\cdot\nn\right|_0^2\\
\lesssim&~ \eps(\sigma E_2(T)+E_1(t))+\PP_0+ P(E_1(T))(\sigma E_2(T))\int_0^T P(E_1(t))(\sigma E_2(t))\dt.
\end{aligned}
\end{equation}

\section{Control of weighted Sobolev norms}\label{sect weighted}

\subsection{Weighted div-curl estimates}

The div-curl estimate for $\|\sqrt{\sigma} v\|_{5.5}^2$ and $\|\sqrt{\sigma} \bp\eta\|_{5.5}^2$ is done similarly as in Section \ref{sect. div bound} and \ref{sect. curl bound}, \red{while we no longer convert the remaining part to interior tangential estimates but preserve the normal trace as in \eqref{divcurl2} in Lemma \ref{hodge}. For $0\leq k\leq 4$ we have
\begin{align}
\label{divcurlv2} \|\sqrt{\sigma}\p_t^k v\|_{5.5-k}^2&\lesssim C\left(\|\sqrt{\sigma} v\|_0^2+\|\sqrt{\sigma}\dive_A \p_t^k v\|_{4.5-k}^2+\|\sqrt{\sigma}\curl_A \p_t^k v\|_{4.5-k}^2+|\sqrt{\sigma}\TP^{5-k}\p_t^k v\cdot\nn|_0^2\right),\\
\label{divcurlb2} \|\sqrt{\sigma}\p_t^k \bp\eta\|_{5.5-k}^2&\lesssim C\left(\|\sqrt{\sigma} b\|_0^2+\|\sqrt{\sigma}\dive_A b\|_{4.5-k}^2+\|\sqrt{\sigma}\curl_A \p_t^k b\|_{4.5-k}^2+|\sqrt{\sigma}\TP^{5-k}\p_t^k \bp\eta\cdot\nn|_0^2\right)
\end{align}where the constant $C=C(|\TP\eta|_{W^{1,\infty}},\|\sqrt{\sigma}\eta\|_{5.5-k}^2)>0$ depends linearly on $\|\eta\|_{5.5-k}^2$ and $b=\bp\eta$.}

For the divergence, we directly get
\begin{equation}\label{wtdivvb1}
\red{\sigma\|\dive_A v\|_{4.5}^2+\sigma\|\dive_A \bp\eta\|_{4.5}^2=0},
\end{equation}and similarly
\begin{equation}\label{wtdivvb2}
\sum_{k=1}^4\sigma\|\dive_A \p_t^kv\|_{4.5-k}^2+\sigma\|\dive_A \p_t^k\bp\eta\|_{4.5-k}^2\lesssim P(E(0))+\int_0^T P(E(t))\dt.
\end{equation}

The weighted curl estimates follow similarly as in section \ref{sect. curl bound}. By taking $\p^{4.5}$ to \eqref{curlbeq}, testing it with $\sigma \p^{4.5} \curl_A v$, and then integrating $\bp$ by parts, we have
\begin{equation}\label{energy sigma curl}
\begin{aligned}
\frac{d}{dt}\frac{1}{2}\left(\int_{\Omega} |\sqrt{\sigma}\p^{4.5}\curl_{A}v|^2+\int_{\Omega} |\sqrt{\sigma}\p^{4.5} \curl_A (\bp \eta) |^2\right)
=\int_{\Omega} \left(\sqrt{\sigma}[\p^{4.5}, \bp] \curl_A \bp\eta + \sqrt{\sigma}\p^{4.5}\mathcal{F} \right)(\sqrt{\sigma}\p^{4.5} \curl_A v)\\
+\int_{\Omega} (\sqrt{\sigma}[\p^{4.5}\curl_A,  \bp] v)(\sqrt{\sigma} \p^{4.5}\curl_A (\bp\eta))+\int_{\Omega} \sqrt{\sigma}\p^{4.5} \curl_{\p_tA} (\bp \eta)(\sqrt{\sigma} \p^{4.5}\curl_A (\bp\eta)).
\end{aligned}
\end{equation}
Since
\begin{align*}
\|\sqrt{\sigma}\mathcal{F}\|_{4.5}^2 \leq& ~\|\sqrt{\sigma}\curl_{\p_t A} v\|_{4.5}^2+\|\sqrt{\sigma}[\curl_A, \bp]\bp\eta\|_{4.5}^2 \\
\leq&~P(\|\bz\|_{5.5},  \|\eta\|_5, \|v\|_5, \|\bp\eta\|_5, \|\sqrt{\sigma}\eta\|_{5.5}, \|\sqrt{\sigma}v\|_{5.5}, \|\sqrt{\sigma}\bp\eta\|_{5.5}), 
\end{align*}
then the terms on the RHS of \eqref{energy sigma curl} can be controlled by $P(E(t))$ as well. Therefore, 
\begin{equation}
\|\sqrt{\sigma}\curl_A v\|_4^2+\|\sqrt{\sigma}\curl_A \bp\eta\|_4^2 \lesssim \int_0^T P(E(t))\red{\dt},
\end{equation} which yields
\red{\begin{equation}\label{wtcurlvb1}
\|\sqrt{\sigma}\curl_A v\|_{4.5}^2+\|\sqrt{\sigma}\curl_A \bp\eta\|_{4.5}^2 \lesssim P(E(0))+ \int_0^T P(E(t))\dt,
\end{equation}and similarly
\begin{equation}\label{wtcurlvb2}
\sum_{k=1}^4\sigma\|\curl_A \p_t^kv\|_{4.5-k}^2+\sigma\|\curl_A \p_t^k\bp\eta\|_{4.5-k}^2\lesssim \sum_{k=1}^4 P(E(0))+\int_0^T P(E(t))\dt.
\end{equation}}

\subsection{Control of the weighted boundary norms}

We still need to control \red{$\wt|\TP^{5-k}\p_t^k v\cdot \nn|_0$ and $\wt|\TP^{5-k}\p_t^k\bp\eta\cdot \nn|_0$} for $0\leq k\leq 4$. For the boundary estimates of $v$, we find that they are exactly the energy terms contributed by surface tension (cf. \eqref{tgs5}, \eqref{tgt5}-\eqref{tgmix}).

As for $\bp\eta$, when $k\geq 1$, we can directly control them by the norms of $v$:
\begin{equation}\label{wtbbdry}
\wt|\TP^{5-k}\p_t^k\bp\eta\cdot \red{\nn}|_0=\wt|\TP^{5-k}\p_t^{k-1}\bp v \cdot \red{\nn}|_0\lesssim\|\bz\|_{L^{\infty}}\|\wt\p_t^{k-1}\p v\|_{5.5-k}+\text{ lower order terms}.
\end{equation}

When $k=0$, we need to control $\wt|\TP^5\bp\eta\cdot\nn|_0$. This term naturally appears as a boundary energy term contributed by the surface tension in the $\TP^4\bp$ tangential estimate, which can be proceeded in the same way as $\TP^4\p_t$-estimate by just replacing $\p_t$ by $\bp$. The reason for that is that $\bp\eta$ and $\eta$ have the same spatial regularity, which is similar to the fact that $\p_t\eta=v$ has the same spatial regularity as $\eta$. In other words, the tangential derivative $\bp$ (note that $\bz\cdot N=0$ on $\p\Omega$!) plays the same role as a time derivative if it falls on the flow map $\eta$. We just list the result of $\TP^4\bp$-estimate
\begin{equation}\label{tgsb}
\|\TP^4 \bp v\|_0^2+\|\TP^4 \bp^2\eta\|_0^2+\frac{\sigma}{2}\left|\TP^5\bp\eta\cdot\nn\right|_0^2\bigg|_{t=T}\lesssim \PP_0+P(E_1(T))(\sigma E_2(T))\int_0^TP(E_1(t))(\sigma E_2(t))\dt.
\end{equation}

\section{The zero surface tension limit}\label{sect 0STlimit}

Now we conclude the energy estimates as follows. First, \eqref{div est} and \eqref{curl est} give the div-curl control of the non-weighted Sobolev norms. Then the interior tangential estimates obtained in \eqref{divcurlv1}-eqref{divcurlb1} are established in \eqref{tgs5} for the spatial derivatives and \eqref{tgt5}-\eqref{tgmix} for the space-time derivatives. In the control of the non-weighted Sobolev norms, the weighted energy $\sigma E_2(t)$ is needed to close the energy estimates. 

The $\wt$-weighted div-curl estimates are established in \eqref{wtdivvb2} and \eqref{wtcurlvb2}. Then we notice that the boundary normal traces are \red{exactly the aforementioned $\wt$-weighted boundary energies contributed by surface tension}. Finally, we can get the following energy estimates
\begin{equation}
E(T)=E_1(T)+\sigma E_2(T)\lesssim \eps E(T)+P(E(0))+P(E(T))\int_0^TP(E(t))\dt,
\end{equation}which together with Gr\"onwall inequality implies that there exists some $T>0$, independent of $\sigma$, such that
\begin{equation}
\sup_{0\leq t\leq T} E(t)\leq P(E(0))\leq \mathcal{C}.
\end{equation}

Finally, we prove the zero surface tension limit. Assume $(w,\bp\zeta,r)$ to be the solution of the incompressible MHD system without surface tension \eqref{MHDL0ST} and $(v^{\sigma},\bp\eta^{\sigma},q^{\sigma})$ to be the solution of \eqref{MHDLST} with $\sigma>0$. Theorem \ref{main energy} and Sobolev embedding implies that 
\begin{align}
\|v^{\sigma}\|_{C^1([0,T]\times\Omega)}^2+\|\bp\eta^{\sigma}\|_{C^1([0,T]\times\Omega)}^2+\|q^{\sigma}\|_{C^1([0,T]\times\Omega)}^2\lesssim \mathcal{C}.
\end{align} 
 By Morrey's embedding, we can prove $v^{\sigma},\bp\eta^{\sigma},q^{\sigma}\in C_t^1H_y^4([0,T]\times\Omega)\hookrightarrow C_t^1C_y^{2,\frac12}([0,T]\times\Omega),$ which implies the equi-continuity of $(v^{\sigma},\bp\eta^{\sigma},q^{\sigma})$ in $C^1([0,T]\times\Omega)$. By Arzel\`a-Ascoli lemma, we prove the uniform convergence (up to subsequence) of $(v^{\sigma},\bp\eta^{\sigma},q^{\sigma})$ as $\sigma\to 0_+$, and the limit is the solution $(w,\bp\zeta,r)$ to \eqref{MHDL0ST}.
 
\begin{appendix}
\section{Appendix: The initial data and compatibility conditions}\label{appendix}
 The required regularity of the initial data is proved by solving the time-differentiated MHD system restricted at $\{t=0\}$. Namely, given $(\eta_0, v_0, \bz, q_0)$, and let $E(t)$ be defined as in \eqref{energy}. We show
\begin{align}\label{data2}
E(0)\leq \mathcal{C},
\end{align}
where $\mathcal{C}$ is defined in Theorem \ref{main energy}. 

\subsection{One time derivative}\label{sect data 1}
First, we control 
$$\|\p_t v(0)\|_4, \|\sqrt{\sigma}\p_t v(0)\|_{4.5}, \|\p_t b(0)\|_4, \|\sqrt{\sigma}\p_t b(0)\|_{4.5}.$$
Since $\p_t v(0)-\bp \bz+\nab_{A_0} q_0=0$ and $\p_t b(0)=\p_t\bp \eta(0)=\bp v_0$, we have
\begin{align}
\|\p_t v(0)\|_{4}\lesssim\|\bz\|_{4}\|\bz\|_{5}+\|\eta_0\|_{5}\|q_0\|_{5},\quad \|\sqrt{\sigma}\p_t v(0)\|_{4.5}\lesssim\|\bz\|_{4.5}\|\sqrt{\sigma}\bz\|_{5.5}+\|\eta_0\|_5\|\sqrt{\sigma}q_0\|_{5.5}+\|\sqrt{\sigma}\eta_0\|_{5.5}\|q_0\|_5,
\end{align}
and 
\begin{align}
\|\p_t b(0)\|_{4}\lesssim\|\bz\|_{4}\|v_0\|_{5},\quad \|\sqrt{\sigma}\p_t b(0)\|_{4.5}\lesssim\|\bz\|_{4.5}\|\sqrt{\sigma}v_0\|_{5.5}.
\end{align}
These imply that we need to bound $\|q_0\|_5$ and $\|\sqrt{\sigma} q_0\|_{5.5}$. Invoking the elliptic equation verified by $q_0$, i.e., 
\begin{equation}\label{q0eq}
\begin{cases}
-\Delta_{A_0} q_0= -(\p_t A(0)^{\mu\alpha})\p_\mu (v_0)_\alpha-(\bp A_0^{\mu\alpha})\p_\mu(\bz)_\alpha+A_0^{\mu\alpha}(\p_\mu \bz\cdot \p)(\bz)_\alpha,~~~&\text{ in }\Omega,\\
q_0=\sigma\mathcal{H}_0=-\sigma\Delta_{g_0}\eta_0\cdot\nn_0,~~~&\text{ on }\Gamma(0),\\
\frac{\p q_0}{\p N}=0,~~~&\text{ on }\Gamma_{\text{fix}},\\
\end{cases}
\end{equation}
the elliptic estimate then yields
\begin{align}
\label{q0 bound}\|q_0\|_5 \leq&~ P(\|\eta_0\|_4, \|v_0\|_4, \|\bz\|_4,  \sigma |\mathcal{H}_0|_{4.5}),\\
\label{q0wt bound} \|\sqrt{\sigma}q_0\|_{5.5} \leq&~ P(\|\eta_0\|_{4.5}, \|v_0\|_{4.5}, \|\bz\|_{4.5},  \sigma^{\frac{3}{2}} |\mathcal{H}_0|_{5}),
\end{align}
where 
\begin{equation}\label{eta bdy high 1}
\sigma|\mathcal{H}_0|_{4.5}  = \sigma |\lap_{g_0} \eta_0\cdot \nn_0|_{4.5},\quad \sigma^{\frac{3}{2}}|\mathcal{H}_0|_{5}  = \sigma^{\frac{3}{2}} |\lap_{g_0} \eta_0\cdot \nn_0|_{5}.
\end{equation}

\subsection{Two and more time derivatives}\label{sect datat}
\red{To control 
$$
\|\p_t^k v(0),\p_t^k b(0)\|_{5-k}, \quad 2\leq k\leq 5
$$
 and 
 $$
 \|\wt\p_t^\ell v(0),\wt\p_t^\ell b(0)\|_{5.5-\ell}, \quad 2\leq \ell \leq 4
 $$
  we have to study the time-differentiated elliptic equation restricted on $\{t=0\}$ to control $\|\p_t^{k-1} q(0)\|_{5-k}$, $\|\wt\p_t^{\ell-1} q(0)\|_{5.5-\ell}$. In light of the first remark after Theorem \ref{main energy},  we require the extra boundary regularity of $\eta_0, v_0$ when controlling these quantities. This seems to be necessary as if we alternatively consider using the Neumann boundary conditions, we would have to control the term with one more time derivative on the velocity and thus we could never close the loop.}
   %So, we then analyze how the initial data of these time derivatives depend on the initial data $\eta_0,v_0,\bz$ and how much regularity we may have to require.

\subsubsection{Two time derivatives}\label{sect data2}
When $k=2$, the control of $\p_t^2 v(0)$ requires the bounds for $\|q_t(0)\|_4$ and $\|\sigma q_t(0)\|_{4.5}$. The Dirichlet boundary condition then contributes to $|\sigma v_0|_{5.5}$ and $|\sigma^{\frac32} v_0|_6$. Specifically, we study the equations for $q_1:=\p_t q(0)$
\begin{equation}\label{q1eq}
\begin{aligned}
-\Delta_{A_0} q_1 &= -(\p_t^2 A(0)^{\mu\alpha})\p_\mu (v_0)_\alpha-(\p_t A(0)^{\mu\alpha})\p_\mu (\p_t v(0))_\alpha-(\bp \p_t A(0)^{\mu\alpha})\p_\mu(\bz)_\alpha-(\bp A_0^{\mu\alpha})\p_\mu(\p_t b(0))_\alpha\\
&+\p_t A(0)^{\mu\alpha}(\p_\mu \bz\cdot \p)(\bz)_\alpha+A_0^{\mu\alpha}(\p_\mu \bz\cdot \p)(\p_t b(0))_\alpha,
\end{aligned}
\end{equation}
with the boundary conditions
\begin{align*}
q_1=\sigma\p_t \mathcal{H}(0),~~~&\text{ on }\Gamma(0),\\
\frac{\p q_1}{\p N}=-(\p_t A^{3\alpha})q_0,~~~&\text{ on }\Gamma_{\text{fix}}.
\end{align*}
We have
\begin{align}\label{q1 bound}
\|q_1\|_4 \leq P(\|\eta_0\|_4, \|v_0\|_4, \|\p_t v(0)\|_3, \|\bz\|_4, \|\p_t b(0)\|_3,  \|q_0\|_4, \sigma |\p_t \mathcal{H}(0)|_{3.5} ),
\end{align}
and
\begin{align}\label{q1wt bound}
 \|\sqrt{\sigma}q_1\|_{4.5} \leq P(\|\eta_0\|_{4.5}, \|v_0\|_{4.5}, \|\p_t v(0)\|_{3.5}, \|\bz\|_{4.5}, \|\p_t b(0)\|_{3.5},  \|q_0\|_{4.5}, \sigma^{3/2} |\p_t\mathcal{H}(0)|_{4} ).
\end{align}
Since $\mathcal{H} = -\sigma\sqrt{g}\Delta_g \eta\cdot \nn=-\sigma g^{ij}\TP^2_{ij}\eta\cdot(\TP_1\eta\times\TP_2\eta)$, we have $\p_t\hh=\sigma \left(g^{ij}\TP^2_{ij}v\cdot(\TP_1\eta\times\TP_2\eta)+\TP^2\eta\TP v P(\TP\eta)\right).$ So, the control of $\sigma |\p_t\mathcal{H}(0)|_{3.5} $ and $\sigma^{3/2} |\p_t\mathcal{H}(0)|_{4}$ respectively reduce to
\begin{align}
\red{\sigma\left |\p_t (\Delta_{g} \eta \cdot \nn )(0)\right|_{3.5} \leq P(\|v_0\|_{5}, \|\eta_0\|_{5}, \sigma |v_0|_{5.5}, \sigma |\eta_0|_{5.5})}, \label{v bdy high1}\\
\red{\sigma^{3/2} \left|\p_t (\Delta_{g} \eta\cdot \nn) (0)\right|_4\leq P(\|v_0\|_{5}, \|\eta_0\|_{5}, \|\sqrt{\sigma} \eta_0\|_{5}, \|\sqrt{\sigma} \eta_0\|_{5.5}, \sigma^{\frac{3}{2}} |v_0|_6, \sigma^{\frac{3}{2}} |\eta_0|_{6})}\label{v bdy high2}.
\end{align}

\subsubsection{Three and four time derivatives}\label{sect data3}
We consider the $\p_t^{k-1}$-differentiated elliptic equation for $q$ at $t=0$
\begin{equation}\label{qkeq}
-\Delta_{A_0} q_{k-1}=-\p_{\nu}\left([\BB^{\mu\nu},\p_t^{k-1}]\p_\mu q\right)-\p_t^{k-1} \left(\p_t A^{\nu\alpha} \p_{\nu}v_{\alpha}-A^{\nu\alpha}\p_\nu\bz\cdot\p\bp\eta_{\alpha}-\bp A^{\nu\alpha} \p_\nu\bp\eta_{\alpha}\right)\big|_{t=0}=:f_{k-1}.
\end{equation}

When $k=3$, the control of $\|\p_t^3 v(0)\|_2$ requires the bounds for $\|\p_t^2 q(0)\|_3$, whose Dirichlet boundary condition reads 
 $$q_2=-\sigma\p_t^2 \hh(0)=-\sigma\sqrt{g}\Delta_{g_0}v_t(0)\cdot\nn_0+\cdots.$$ 
\red{Here and throughout, $\cdots$ denotes all boundary terms without the top order time derivative on the velocity, which do not contribute to the highest order boundary terms of $\eta_0$ and $v_0$.}
 \red{Using $v_t(0)=\btp\eta_0-\nabla_{A_0}q_0$ on $\Gamma$ and $\Delta_{g_0} f=P(\TP\eta)\TP^2f$, we need to control $|\sigma\TP^2\nabla_{A_0}q_0\cdot\nn_0|_{2.5}$ and $|\sigma\TP^2\btp^2\eta_0\cdot\nn_0|_{2.5}$. The latter is controlled by 
\[
|\sigma(\TP^2\btp^2\eta_0)\cdot\nn_0|_{2.5}\lesssim \|\bz\|_5^2|\sigma \Delta_{g_0}\eta_0\cdot\nn_0|_{4.5}.
\]}

\red{To control the term involving $q_0$, we need the following normal trace lemma.
\begin{lem}\label{normaltrace}
Assume the vector field $X$ satisfies $X\in L^2(\Omega)$ and $\nabla_{A_0}\cdot X=0$, then 
\begin{equation}
|X\cdot \nn_0|_{-\frac12}\lesssim \|X\|_0+\|\nabla_{A_0}\cdot X\|_0.
\end{equation}
Similarly, one has for $s\geq 1$
\begin{equation}
|X\cdot \nn_0|_{s-\frac12}\lesssim C(|\TP\eta_0|_{W^{1,\infty}},\|\eta_0\|_{s+\frac12}) \left(\|\TP^s X\|_0+\|\nabla_{A_0}\cdot X\|_{s-1}\right).
\end{equation}
\end{lem}
\begin{proof}
We only prove the case $s=0$ as one can later replace $X$ by $\TP^s X$. The second inequality relies on $\|\eta_0\|_{s+\frac12}$ because $\TP^s$ may fall on $\nn_0$. For $s=0$, we pick an arbitrary test function $\varphi\in H^\frac12(\p\Omega)$ with its harmonic extension $\Phi\in H^1(\Omega)$. Then by the divergence theorem
\[
\ig X\cdot\nn_0 \varphi\dS=\io\nabla_{A_0}\cdot X\Phi\dy+\io X\cdot\nabla_{A_0}\Phi\dy
\]and taking supremum over all $\varphi$ finished the proof.
\end{proof}}

\red{Applying Lemma \ref{normaltrace} to $q_0$, we get $|\sigma\TP^2\nabla_{A_0}q_0\cdot\nn_0|_{2.5}\lesssim \sigma\|\TP^5\nabla_{A_0}q_0\|_0+\|\sigma\Delta_{A_0} q_0\|_4,$ which shows that we need the bounds for $\|\sigma\nabla q_0\|_5$. This can be controlled by invoking the standard elliptic estimate (parallel to \cite[Prop. 5.3]{gu2016construction}):
\begin{equation}\label{q0 bdy high}
\|\sigma\nabla_{A_0} q_0\|_5\lesssim C(|\sigma\eta_0|_{5.5})(\|\sigma\Delta_{A_0} q_0\|_{4}+\sigma^2|\hh_0|_{5.5})\lesssim P(|\sigma\eta_0|_{5.5},\|\eta_0,v_0,\bz\|_5,\sigma^2|\Delta_{g_0}\eta_0\cdot\nn_0|_{5.5}).
\end{equation}}
\red{Summarizing the bounds above, we get
\begin{equation}\label{q2 bound}
\|q_2\|_3\leq ~P(\|\eta_0,v_0,\bz\|_5,|\sigma\eta_0|_{5.5},|\sigma\Delta_{g_0}\eta_0\cdot\nn_0|_{4.5},|\sigma^2\Delta_{g_0}\eta_0\cdot\nn_0|_{5.5})
\end{equation}}
\red{Similarly, for $\sqrt{\sigma}$-weighted norms, we have $\|\wt\p_t^3 v(0)\|_{2.5}\leq \|\wt\bp^2 v_t(0)\|_{2.5}+\|\wt\nabla_{A_0}q_2(0)\|_{2.5}$, and then
\begin{align}
\|\wt\bp^2 v_t(0)\|_{2.5}\lesssim&~ \wt\|v_t(0)\|_{4.5}+\|\wt q_2(0)\|_{3.5},\\
\label{q2wt bound}\|\wt q_2\|_{3.5}\lesssim&~ P(\|\wt\eta_0,\wt v_0,\wt \bz\|_{5.5}, |\sigma^{\frac32}\eta_0|_6, |\sigma^{\frac32}\Delta_{g_0}\eta_0\cdot\nn_0|_{5}, \sigma^{\frac52}|\Delta_{g_0}\eta_0\cdot\nn_0|_6).
\end{align}}
 
\red{For $k=4$ we follow the same strategy to get 
\[
\|q_3\|_2\leq \|f_3\|_0+|\sigma\sqrt{g}\Delta_{g_0}\p_t^2 v(0)\cdot\nn_0|_{1.5}+\cdots\lesssim \|f_3\|_0+|\sigma\TP^2 \nabla_{A_0}q_1\cdot\nn_0|_{1.5}+|\sigma\TP^2 \btp^2 v_0\cdot\nn_0|_{1.5}+\cdots
\]where the term $f_3$ can be controlled by the aforementioned quantities. The third term is controlled by $|\sigma v_0|_{5.5}$. Then the second term is controlled by using a parallel argument as in Subsection \ref{sect data2},
\begin{align}
\label{q1 bdy high} \|\sigma\nabla_{A_0} q_1\|_4\lesssim&~\|f_1\|_3+ P(\sigma\|\eta_0,v_0,\bz,q_0\|_5,\sigma|v_0,\eta_0|_{5.5}, \sigma^2|v_0,\eta_0|_{6.5}),\\
\label{q3 bound} \|q_3\|_2\lesssim&~  P(\|\eta_0,v_0,\bz,q_0\|_5,\sigma|v_0,\eta_0|_{5.5}, \sigma^2|v_0,\eta_0|_{6.5}),\\
\label{q3 bound'} \|\wt q_3\|_{2.5}\lesssim &~ P(\wt\|\eta_0,v_0,\bz,q_0\|_{5.5},\sigma^{\frac32}|v_0,\eta_0|_{6}, \sigma^{\frac52}|v_0,\eta_0|_{7}).
\end{align}}

\subsubsection{Top-order time derivatives}\label{sect data5}
\red{Finally, we control $\|\p_t^4 q(0)\|_1$. Consider the $\p_t^4$-differentiated elliptic equations restricted on $\{t=0\}$. Below $q_4$ stands for $\p_t^4 q(0)$
\begin{equation}\label{q04eq}
-\Delta_{A_0} q_4=-\p_{\nu}\left([\BB^{\mu\nu},\p_t^4]\p_\mu q\right)-\p_t^4 \big(\p_t A^{\nu\alpha} \p_{\nu}v_{\alpha}-A^{\nu\alpha}\p_\nu\bz\cdot\p\bp\eta_{\alpha}-\bp A^{\nu\alpha} \p_\nu\bp\eta_{\alpha}\big)\big|_{t=0}=:f_4.
\end{equation}
We alternatively consider $\|\TJ^{-1}\p_t^4 q(0)\|_2\gtrsim\|\p_t^4 q(0)\|_1$. The Dirichlet boundary condition reads 
\[
q_4=-\sigma P(\TP\eta)\TP^2\p_t^3v(0)\cdot\nn_0+\cdots
\]}
\red{Applying $\TP^{-1}$ to \eqref{q04eq} and using standard elliptic estimates, we get
\begin{equation}
\|\TJ^{-1}q _4\|_2\lesssim \|\TJ^{-1} f_4\|_0+|\sigma P(\TP\eta)\TP^2\p_t^3v(0)\cdot\nn_0|_{0.5},
\end{equation}where the first term is directly bounded
\[
\|\TJ^{-1} f_4\|_0\leq \|f_4\|_0\leq P\left(\sum_{j=0}^2\|\p_t^j v(0)\|_3,\|\p_t^3v(0)\|_2,\|\p_t^4 v(0)\|_1,\|\bz,\eta_0,q_0,q_1,q_2\|_3,\|q_3\|_2\right).
\]}
\red{Invoking $\p_t v=\bp^2\eta_0-\nabla_{A_0}q_0$ twice, we have 
\begin{equation}
|\sigma P(\TP\eta)\TP^2\p_t^3v(0)\cdot\nn_0|_{0.5}\lesssim P(|\TP\eta|_{W^{1,\infty}})\left(|\sigma\nabla_{A_0} q_2\cdot\nn_0|_{2.5}+\|\bz\|_5^4|\sigma\Delta_{g_0}\eta_0\cdot\nn_0|_{4.5}+|\sigma\nabla_{A_0}q_0\cdot\nn_0|_{4.5}\right),
\end{equation}where the second term and the third term also appear in Subsection \ref{sect data3}. Applying Lemma \ref{normaltrace} to the first boundary term and invoking \eqref{qkeq} for $k=3$, we get $|\sigma\nabla_{A_0} q_2\cdot\nn_0|_{2.5}\lesssim\|\sigma \nabla_{A_0}q_2\|_3+\|\sigma\Delta_{A_0}q_2\|_2=\|\sigma \nabla_{A_0}q_2\|_3+\|f_2\|_2,
$, where the top order terms in $f_2$ is $\p\p_t^2v(0)$, $\p\p_t^2\bp\eta (0)$ and $\p^2\p_t q(0)$ which have been controlled before. }

\red{For the term $\|\sigma \nabla_{A_0}q_2\|_3$, we do not use the Dirichlet boundary condition to control as in Subsection \ref{sect data3} because this will further increase the regularity of the mean curvature. Instead, we use the Neumann boundary condition as in the proof of Proposition \ref{lem qell}. This still makes sense, as we have already solved $q_2$ and $\p_t^3v(0)$ at this point. Following \eqref{qnormal}, it suffices to control $\|\sigma \TP^3\nabla_{A_0}q_2\|_0$. Then we take $\TP^3$ in \eqref{qkeq} for $k=3$ and convert the right-hand side to the divergence form. Since the source term contains up to three time derivatives that have been controlled before, we can use the same method as in Proposition \ref{lem qell} to finish the control of the initial data.}

\subsection{The compatibility conditions}
The above analysis yields that the initial data of the $\sigma>0$ problem has to satisfy the compatibility conditions up to the fourth order, where the $j$-th order ($0\leq j\leq 4$) compatibility conditions read
\begin{equation} 
\begin{aligned}
q_{j}=\sigma\p_t^{j}\mathcal{H}(0) \quad &\text{on}\,\,\Gamma(0),\\ \quad 
 \p_3 q_j=0\quad &\text{on}\,\,\Gamma_{\text{fix}}.
 \end{aligned}
\end{equation}

\red{We conclude that the initial value of the energy functional $E(0):=E_1(0)+\sigma E_2(0)$ depends on the following bounds
\[
\|\eta_0,v_0,\bz\|_5,~~\wt\|\eta_0,v_0,\bz\|_{5.5},~~\sum_{k=0}^{3}\sigma^{1+\frac{k}{2}}\left|\eta_0,v_0\right|_{5.5+\frac{k}{2}},~~\sum_{k=0}^{3}\sigma^{1+\frac{k}{2}}\bigg|\underbrace{\Delta_{g_0}\eta_0\cdot\nn_0}_{=\hh_0}\bigg|_{4.5+\frac{k}{2}}.
\]}

Furthermore, this can be carried over to the $\sigma=0$ problem, namely, the first line is replaced by $q_j =0$ on $\Gamma(0)$. And \red{all $\sigma$-weighted norms vanish when there is no surface tension, and that initial data is exactly for the $\sigma=0$ case as in \cite{gu2016construction}.} 

\end{appendix}

\end{document}